\documentclass[10.6pt, reqno]{amsart}
\usepackage{caption}
\usepackage{graphicx}
\usepackage{cancel}
\usepackage{mathtools}
\usepackage{subfigure}
\usepackage{slashed}
\usepackage{amsfonts}
\usepackage{amssymb}
\usepackage{accents}
\usepackage{amsmath}
\usepackage{amsxtra}
\usepackage{epstopdf}
\usepackage{latexsym}
\usepackage{mathrsfs}
\usepackage{leftidx}
\usepackage{tensor}
\usepackage{enumitem}

\newtheorem{theorem}{Theorem}[section]
\newtheorem{lemma}[theorem]{Lemma}

\newtheorem{proposition}[theorem]{Proposition}

\theoremstyle{remark}
\newtheorem{remark}[theorem]{Remark}

\numberwithin{equation}{section}

\def\bh{\mathbf{h}}

\def\rp#1{^{\!(#1)}}

\def\dn{{\Delta_0}}

\def\Hb{\underline{\H}}

\def\bb{{\mathbf{b}}}

\def\Er{\mbox{Er}}

\def\sn{{\slashed{\nabla}}}
\def\Q{\mathcal{Q}}
\def\sQ{\mathscr{Q}}

\def\bQ{{\textbf{Q}}}
\def\bR{{\textbf{R}}}

\def\ti{\tilde}

\def\bg{\mathbf{g}}

\def\I{{\mathcal I}}

\def\beaa{\begin{eqnarray*}}
\def\eeaa{\end{eqnarray*}}

\def\ba{\begin{array}}
\def\ea{\end{array}}

\def\be#1{\begin{equation} \label{#1}}
\def \eeq{\end{equation}}

\newcommand{\nn}{\nonumber}

\def\l{\langle}
\def\r{\rangle}

\def\cir{\overset\circ}
\def\nn{\nonumber}
\def\S{{\mathcal S}}

\def\ud#1{\underline{#1}}

\def\S2{{\mathbb S}^2}

\def\E{{\mathcal E}}

\def\W{{\mathcal W}}

\def\ub{\underline{u}}
\def\Lb{\underline{L}}

\def\D{{\mathcal D}}
\def\H{{\mathcal H}}

\def\N{{\mathcal N}}

\def\F{{\mathcal F}}

\def\c{\cdot}

\def\a{\alpha}
\def\b{\beta}

\def\ep{{\epsilon}}
\def\ve{{{\textbf{$\varepsilon$}}}}
\def\l{\langle}
\def\r{\rangle}
\def\ga{\gamma}

\def\O{\mathcal{O}}

\def\p{\partial}

\def\F{{\mathcal{F}}}

\def\C{{\mathcal C}}

\def\Lb{{\underline{L}}}

\def\Tr{\mbox{Tr}}

\def\f14{\frac{1}{4}}
\def\f12{{\frac{1}{2}}}

\def\t1a{t^{-\frac{1}{a}}}

\def\bm{{\bf m}}

\def\sn{{\slashed{\nabla}}}

\def\bR{{\emph{\bf{R}}}}

\def\ti{\tilde}

\def\I{{\mathcal I}}

\def\beaa{\begin{eqnarray*}}
\def\eeaa{\end{eqnarray*}}

\def\ba{\begin{array}}
\def\ea{\end{array}}
\def\be#1{\begin{equation} \label{#1}}
\def \eeq{\end{equation}}
\def\nn{\nonumber}

\def\l{\langle}
\def\r{\rangle}

\def\cir{\overset\circ}
\def\nn{\nonumber}
\def\S{{\mathcal S}}

\def\S2{{\mathbb S}^2}

\def\E{{\mathcal E}}

\def\ub{{\underline{u}}}

\def\Lb{\underline{L}}

\def\D{{\mathcal D}}
\def\H{{\mathcal H}}

\def\Nb{{\underline{\mathcal{N}}}}

\def\F{{\mathcal F}}

\def\c{\cdot}

\def\a{\alpha}
\def\b{\beta}

\def\l{\langle}
\def\r{\rangle}
\def\ga{\gamma}

\def\p{\partial}

\def\F{{\mathcal{F}}}

\def\Lb{{\underline{L}}}

\def\Tr{\mbox{Tr}}

\def\f14{\frac{1}{4}}
\def\f12{{\frac{1}{2}}}

\def\t1a{t^{-\frac{1}{a}}}

\def\bm{{\bf m}}

\newcommand{\bea}{\begin{eqnarray}}
\newcommand{\eea}{\end{eqnarray}}

\def\nn{\nonumber}

\newcommand{\les}{\lesssim}

\def\S{\mathcal{S}}

\def\cir#1{\stackrel{\circ}{#1}}

\def\sE{\mathscr{E}}
\def\sF{\mathscr{F}}

\def\sP{{\mathscr{P}}}

\def\sQ{{\mathscr{Q}}}

\def\BA#1{\mbox{BA}_#1}

\def\str#1#2{\stackrel{\{#1\}}{#2}}

\def\ud#1{\underline{#1}}

\def\ssm{{\smallsetminus}}

\begin{document}
\title[]
{On the exterior stability of nonlinear wave equations}
\author{Qian Wang}
\address{
Oxford PDE center, Mathematical Institute, University of Oxford, Oxford, OX2 6GG, UK}
  \email{qian.wang@maths.ox.ac.uk}
  \date{\today}
\begin{abstract}
We consider a general class of nonlinear wave equations, which admit trivial solutions and not necessarily verify  any form of null conditions.  They typically include various John's examples \cite{John, John2}, the reduced Einstein equations under wave coordinates  and  the irrotational fluids. For compactly supported small data, one can only have a semi-global result \cite{KJ, SKal1}, which states that the solutions are well-posed  upto a finite time-span depending on the size of the Cauchy data.  For  some of the equations of the class, the solutions blow up within a finite time  for the compactly supported data of any size.  For data prescribed on ${\mathbb R}^3\setminus B_R$ with small weighted energy,   without some form of null conditions on the nonlinearity,  the exterior stability is not expected to hold in the full domain of dependence, due to the known results of formation of shocks with  data on annuli. The classical method can only give the well-posedness upto a finite time.

 In this paper, we prove that, there exists a constant $R(\ga_0)\ge 2$, depending on the fixed weight exponent $\ga_0>1$ in the weighted energy norm, if the norm of the data are sufficiently small on  ${\mathbb R}^3\setminus B_R$ with the fixed number $R\ge R(\ga_0)$, the solution exists and is unique in the entire exterior of a schwarzschild cone initiating from $\{|x|=R\}$ (including the boundary) with small negative mass $-M_0$. $M_0$  is  determined according to the size of the initial data.  In this exterior region, by constructing the schwarzschild cone foliation, we can improve the linear behavior of wave equations in particular on the transversal derivative $\Lb \phi$. 
Such improvement enables us to control the nonlinearity violating the null condition without loss, and thus  show the solutions converge to the trivial solution. For semi-linear equations,  such stability region can be any close  to $\{|x|-t>R\}$ if the weighted norm of the data is sufficiently small on $\{|x|\ge R\}$.

The other interesting aspect of our method lies in  that it treats the massless and massive wave operator in a uniform way. Thus it works for equations with nonnegative variable potentials and an equation system with different  potentials.  As a quick application, we give the exterior stability result for Einstein (massive and massless) scalar fields. We prove the solution converges to a small static solution, stable in the entire exterior of a schwarzschild cone with positive mass, which then is  patchable to the interior results.
\end{abstract}
\maketitle
\section{Introduction}
In this paper, we consider nonlinear wave equations in ${\mathbb R}^{3+1}$  of the following form
\begin{equation}\label{eqn_1}
\left\{\begin{array}{lll}
\bg^{\a\b}(\phi, \p \phi) \p_\a \p_\b \phi=\N^{\a\b}(\phi) \p_\a \phi \p_\b \phi+q(x)\phi\\
\phi|_{t=0}=\phi_0, \quad \p_t \phi|_{t=0}=\phi_1
\end{array}\right.
\end{equation}
with the smooth function  $0\le q(x)\le 1$. $\bg(\phi, \p \phi)$ is a Lorentzian metric.   $\bg(y, \mathbf{P})$  is  smooth on  variables $y\in {\mathbb R}$ and ${\mathbf P}\in {\mathbb R}^4$.  $\bg^{\a\b}(0, \mathbf{0})=\bm^{\a\b},$ where $\bm$ denotes the Minkowski metric. \begin{footnote}{ We fix the convention that, in the Einstein summation convention, a Greek letter
is used for index taking values $0, 1, 2, 3.$ $x^0=t$ and $\p_0 =\p_t$. 
}\end{footnote} The functions   $\N^{\a\b}(y)$   are  smooth  for $y\in {\mathbb R}$.
\begin{footnote}
{Our proof still works if  $\N^{\a\b}$  also smoothly depend on $\p \phi$ and $q$ also depends on $t$  with nearly no modification. We keep them in the simple form  for ease of exposition.}
\end{footnote}

The most important case for us is $q\equiv0$. For convenience, we  assume the derivatives of $q$ satisfy \begin{footnote}{For a differential operator $P$, $P\rp{n}$ means applying  $P$  to the $n$-th order,  $P\rp{\le n}=\sum_{0\le i\le n }P\rp{i}$, and  $P\rp{0}=id$.}\end{footnote}
\begin{equation}\label{potential}
 |r^i \p\rp{i}  q|\les r^{-2-\eta}  \mbox{ with } i=1,\cdots n, \, \, r=|x|
\end{equation}
where $\eta>0$ is any fixed constant. \begin{footnote}{ $A\les B$ means $A\le c B$ with the constant $c\ge 1$. $A\approx B$ means $A\les B$ and $B\les A$.}\end{footnote}

Throughout this paper we set $H^{\a\b}=\bg^{\a\b}-\bm^{\a\b}$ and define $\bg^{\a\b}\p_\a \p_\b=\widetilde{\Box}_\bg$.
 In case  $\bg\equiv \bm$, the quasilinear wave equation (\ref{eqn_1}) becomes a semilinear equation.
\subsection{Main problem}
   The discussion in this part will  mainly focus on the case that $q\equiv 0$.  Incorporating  the nontrivial potential in the equation (\ref{eqn_1}) is mainly for  applying our method to  an equation system with various potentials. We emphasize that there is no assumption of any form of special structure on the quadratic nonlinear term on the righthand side of (\ref{eqn_1}). For the general class of wave equations   (\ref{eqn_1}), we consider to construct the global-in-time classical solution for   the generic Cauchy  data with small  weighted energy on  $\{|x|\ge R\}$.

Throughout this paper, we assume the initial data are not compactly supported in ${\mathbb R}^3$.   We first give the definition of the weighted norm for the initial data.
   Let $1< \ga_0<2$ be a fixed constant and  $q_0=\sup_{\{|x|\ge R\} } q(x)$.
We denote
\begin{equation}\label{3.21.4.18}
\begin{split}
\mathcal{E}_{k, \ga_0, R, q_0}&=\int_{\Sigma_0\cap\{ r\ge R\}} (1+r)^{\ga_0-2} |\phi_0|^2 dx\\
&+\sum\limits_{l\leq k}\int_{\Sigma_0\cap \{r\ge R\}}(1+r)^{\ga_0+2l}( |\p\p^l\phi_0|^2+|\p^l\phi_1|^2+ q_0|\p^l\phi_0|^2 )dx,
\end{split}
\end{equation}
where $\Sigma_0=\{t=0\}$.  The $q_0$ in the subindex  may be dropped if we only consider one single equation, instead of an equation system.
We may use $\E_{k, \ga_0}$ as a short-hand notation whenever there occurs no confusion.
   Here $R$ is a fixed constant and  $R\ge 2$.

For initial data with compact support, in either the semilinear or the quasilinear case, there holds only a semi-global existence  result, with the time-span of the solution depending on the size of the small data (See \cite{KJ,SKal1}.)  The finite life-span of the solution therein is actually sharp. In \cite{John} and \cite{John2}, examples of equations of the type (\ref{eqn_1}) are constructed which  does not admit global solution  for data of any size.  The potential flow of fluid equations are also typical examples of (\ref{eqn_1}) which do not verify null condition.  In the work of Christodoulou \cite{shock_dm} on the relativistic Euler equation  and the work of Speck \cite{shock_Speck} for the geometric wave equation, when a set of  small cauchy data is prescribed in an annulus,  the singularity  of the characteristic surfaces is formed  within  finite time. The semi-global well-posedness of the solution holds until  the formation of the shock.  Based on these well-known facts, one should not expect to have any global-in-time result if the data is compactly supported. For the generic data prescribed on $\{|x|\ge R\}$, due to the results of \cite{shock_dm} and \cite{shock_Speck},  one should not expect a global-in-time result in the full domain of dependence, which is exterior to the outgoing characteristic cone initiated from $\{|x|=R\}$ upto the boundary, no matter how small the data is, since the weighted norm (\ref{3.21.4.18}) of the small annuli-data can be sufficiently small.

When the data is non-compactly supported, the known energy method can only give the local-in-time result of well-posedness, even if the data are small on $\{|x|\ge R\}$. See Figure \ref{fig4} for the regions of well-posedness in the semilinear case by using the time foliation or the double null foliations. The $T(\ep)$ and $v_*(\ep)$ are both some finite numbers depending on the bound $\E_{N,\ga_0,R, 0}(\phi[0])\le \ep$.
We raise the question  that for the generic data with finite weighted energy, if the blow-up of the solution can only occur in a region interior to certain cone initiated from a sphere $\{|x|=R\}$. This is trivially true if the datum is compactly supported due to the standard argument of the finite speed of propagation. For the non-compactly supported data,  we give an affirmative answer by showing  the following result of exterior stability.
   \begin{theorem}[A rough statement of main result]\label{7.30.3.18}
Let $1<\ga_0<2$ be fixed.   Consider (\ref{eqn_1}). There exist a  universal constant\begin{footnote}{ Throughout the paper, a universal constant means a constant that depends only  on the constant $\eta$ and the bound in (\ref{potential}),  the bounds of  $|D\rp{\le 3} \bg|$ and $|D\rp{\le 3}\N|$ on small compact domains. }\end{footnote} $C\ge 1$, a small constant $\delta_1>0$ and a constant $R(\ga_0, C)\ge 2$ such that,
if $\phi[0]=(\phi_0, \phi_1)$ verifies  $\E_{3, \ga_0, R}\le \delta_1$ with  $R\ge R(\ga_0, C)$,  then with $M_0=C \delta^\f12_1$, exterior to a schwarzschild  cone of mass $-M_0$ \begin{footnote}{See the definition of the metric in (\ref{metric}).}\end{footnote} initiated from  $\{|x|=R\}$ (including the cone itself),  there exists a unique global-in-time solution which converges to the trivial solution  as $r\rightarrow \infty$ for any $t> 0$. The solution  not only has the standard  asymptotic behavior of the free wave, but also has improved global decay properties.
\begin{figure}[ht]
\centering
\includegraphics[width = 0.48\textwidth]{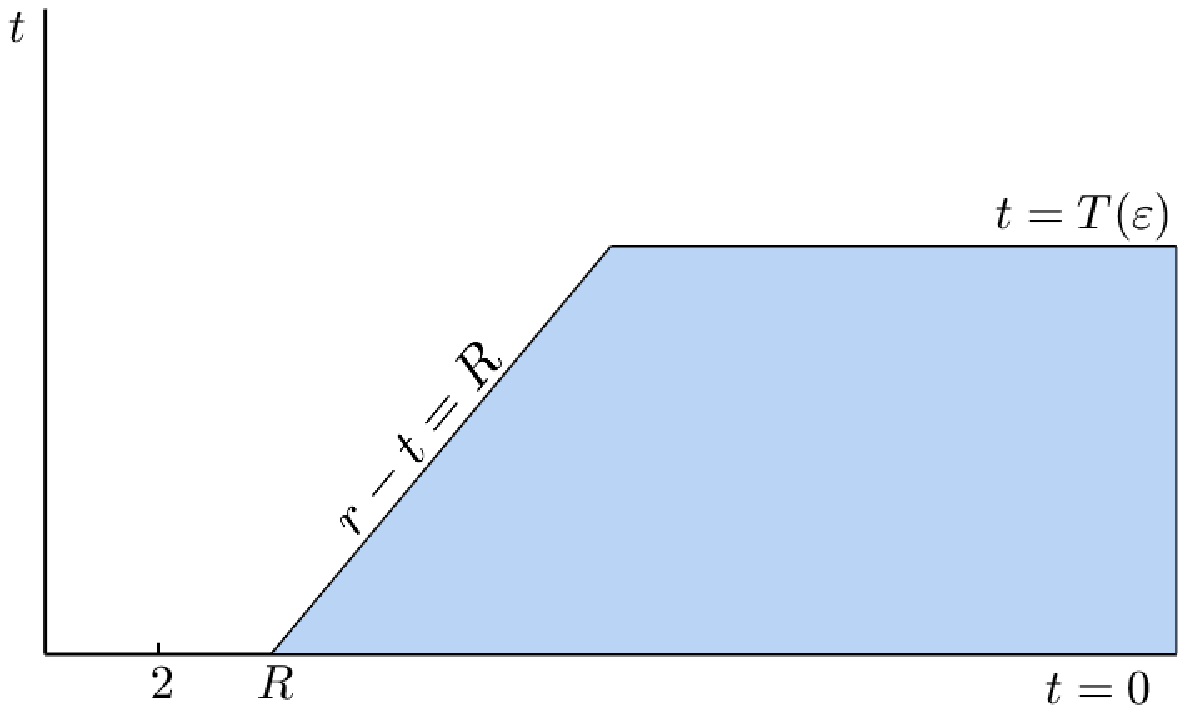}
  \includegraphics[width = 0.48\textwidth]{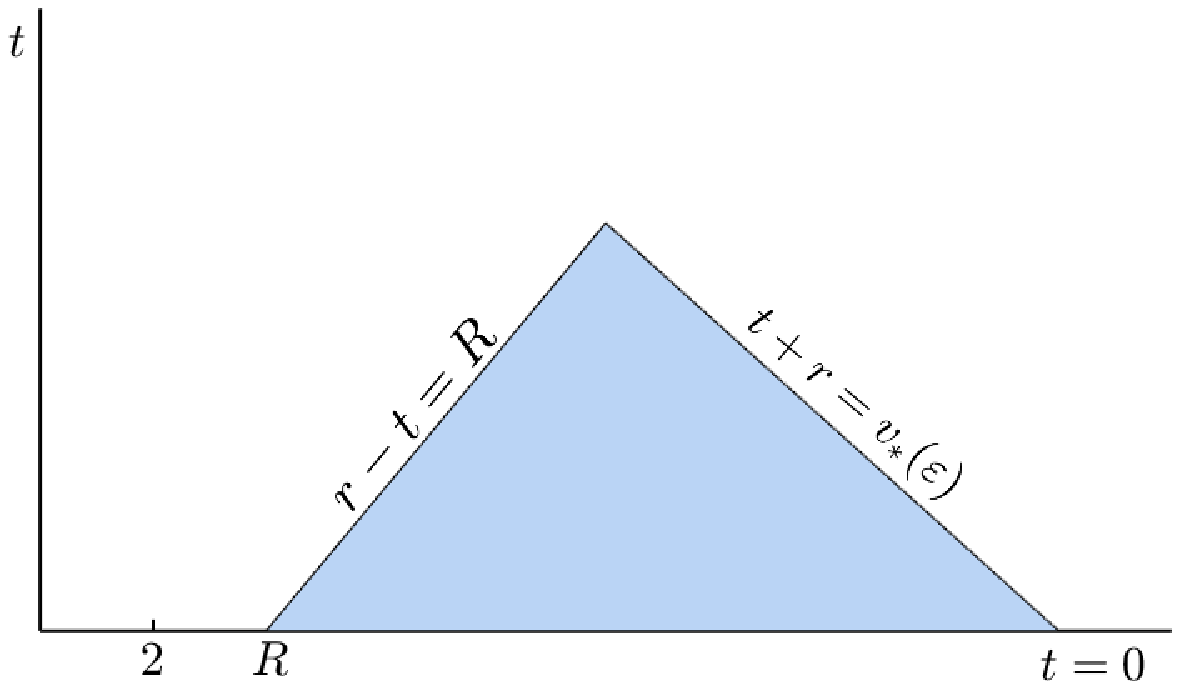}
  \vskip -0.7cm
  \caption{Illustration of the classical semi-global results}\label{fig4}
\end{figure}
   \end{theorem}
   \begin{figure}[ht]
\centering
\includegraphics[width = 0.48\textwidth]{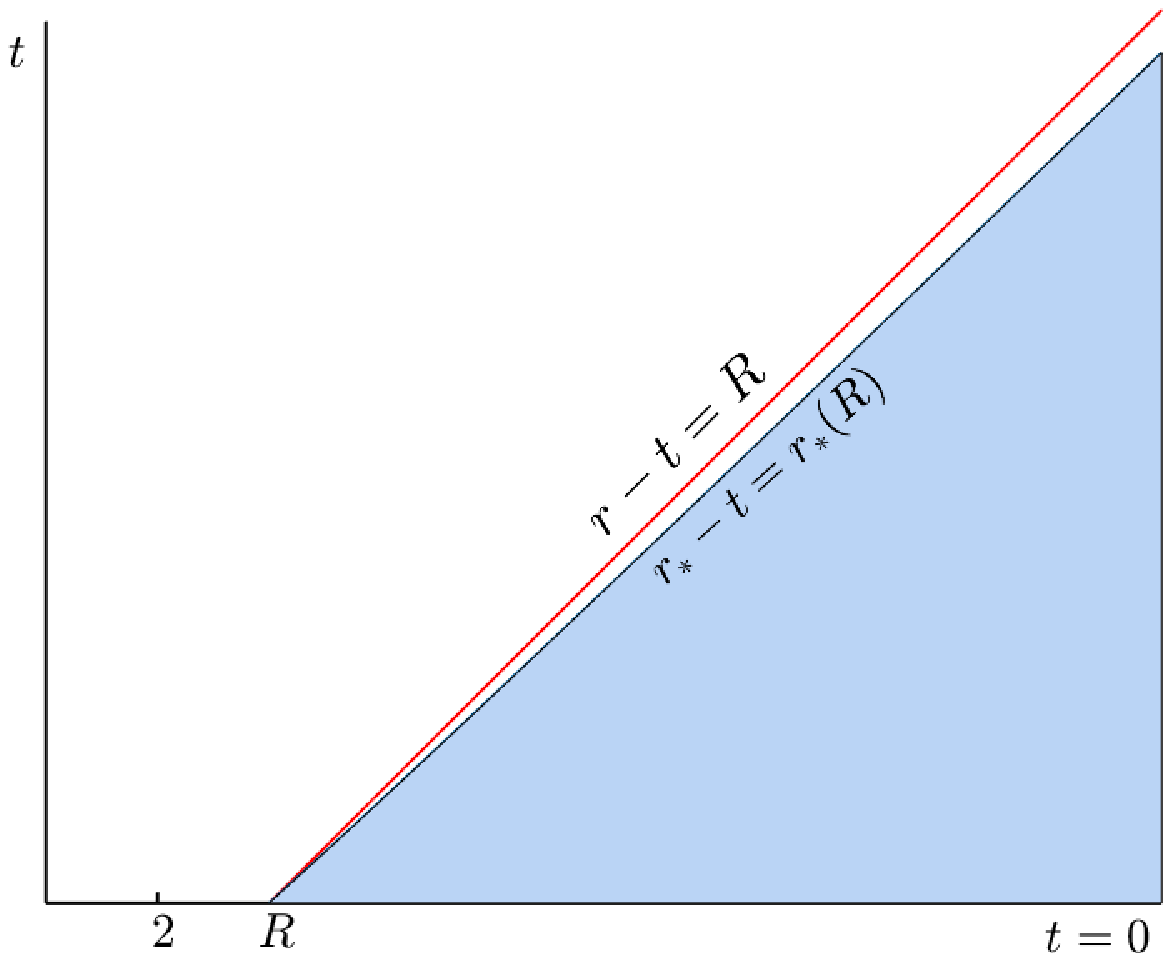}
  \includegraphics[width = 0.46\textwidth]{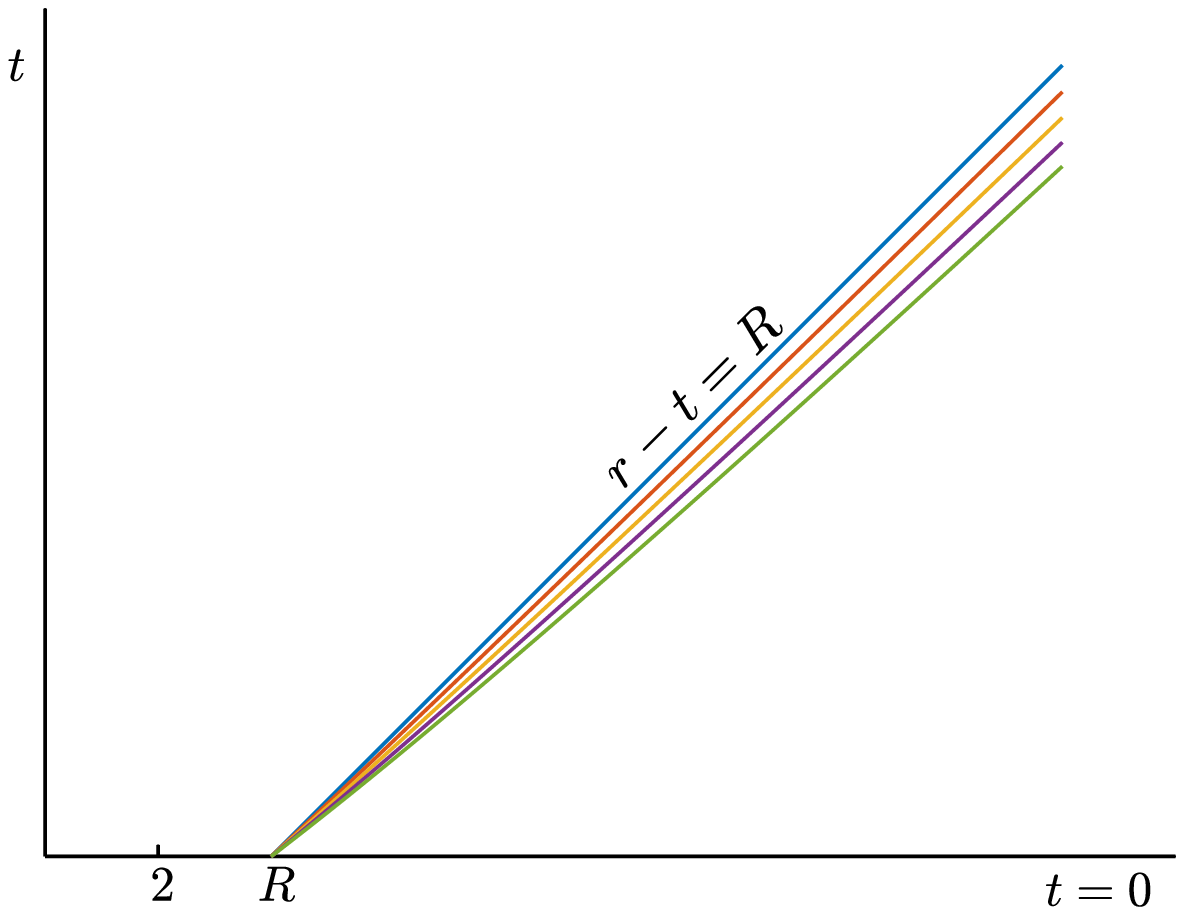}
  \vskip -0.7cm
  \caption{Illustration of the stability region of the main result for the semilinear equations.}\label{fig2}
\end{figure}
    One can directly apply the result to the generic  data with bounded weighted energy, provided that $R$ is sufficient large.
Such schwarzschild cone divides the spacetime into a stability region at its exterior, while all the singularities can only be formed in its  open interior, which are compactly supported for all $t>0$. This  gives the affirmative answer to our question.

  We have several remarks on the above rough statement of our result.
  \begin{enumerate}
\item[(1)]  In this result,   the mass of the schwarzschild metric is $-M_0$  with $|M_0|\ll \frac{1}{10}$ chosen according to the size of the initial data.  The definition of the metric  can be found in (\ref{metric}). For  the general equation (\ref{eqn_1}), we choose $M_0>0$. Correspondingly, the boundary of the exterior region is  slightly spacelike\begin{footnote}{Throughout the paper, spacelike, null or timelike are  in terms of the Minkowski metric.}\end{footnote}. For the semilinear case and Einstein scalar fields, we can have better results than the above statement. For the semilinear case, the stability region can be any close to $\{r> t+R\}$.  For Einstein equations, we can choose $M_0\le 0$. The corresponding schwarzschild cone is  timelike or null. This makes the exterior result patchable  with an interior result based on  the Minkowskian hyperboloidal foliation if we further assume the smallness of the initial data in $B_R$. \begin{footnote}{See a semilinear result \cite{FWY} for an example of such direct patching.}\end{footnote} We refer the readers to the main theorems, Theorem \ref{thm2}, Theorem \ref{thm_quasi} and Theorem \ref{eins_thm}, for detailed statements.

\item[(2)] If the data are compactly supported, the choice of $M_0=C \delta_1^{\f12}$ leads to a semi-global result. The life-span coincides with the standard almost global result \cite{KJ}.

\item[(3)] The sharp local-wellposedness result \cite{tata_smith, roughgeneral}  implies (\ref{eqn_1}) is local-in-time wellposed for data  $\phi[0]=(\phi_0, \phi_1)$  in the normed space $H^{3+\ep}({\mathbb R}^3)\times H^{2+\ep}({\mathbb R}^3),\, \ep>0$. In terms of the order of derivative, our data is at the level of  $H^4({\mathbb R}^3)\times H^3({\mathbb R}^3)$. In a standard regime of  commuting vector fields approach, it is optimal in terms of regularity.
    \end{enumerate}

\subsection{Review of history and inspiration}
In general, in ${\mathbb R}^{3+1}$, one can construct global-in-time solution of (\ref{eqn_1}) for generic small data only when the quadratic nonlinearity $\N^{\a\b}(\phi) \p_\a \phi \p_\b\phi$  verifies certain null condition.  The standard null condition, since it  was raised by Klainerman \cite{SK},  has been  deeply exploited for proving global existence results for various equations with such structures.  The case for the general equation (\ref{eqn_1}) with $q>0$ being a fixed constant is also  studied in  \cite{KKL} for small data with compact support. One can refer to \cite{SKNull, DM} for the results for quasilinear wave equations verifying null conditions. The null conditions, which are important algebraic cancelation structures can be found in  various important geometric or  physical field equations, such as wave maps,  Maxwell-Klein-Gordon equations,  Yang-Mills equations and Einstein equations. For the semilinear case, one can refer to \cite{Lindsterb, Yang2015, YangYu:MKG} for the global results of the massless Maxwell-Klein-Gordon equations, to \cite{mMaxwell, Psarelli} for the massive case, and to \cite{Moncrief1, Moncrief2} for the result of Yang-Mills equations. One can find in    \cite{MKGkl, YMkl, KriegerMKG4, OhMKG4}  the global well-posedness results with low or optimal regularity with large data for  Maxwell-Klein-Gordon equations and Yang-Mills equations.

The Einstein equation system is  an important example of the system of quasilinear equations that verifies null conditions in an intrinsic geometric framework, relative to the maximal foliation gauge. Under this gauge, the small data global-existence result was proved  by Klainerman and Christodoulou in the monumental work \cite{CK}. Under  the wave coordinate gauge, the reduced Einstein equation system takes the form of (\ref{eqn_1}), with the righthand side verifying a so-called weak null condition.

Let us compare the simplest example of equations with the weak null condition,
\begin{equation}\label{wn}
\Box_\bm \phi_1=-(\p_t\phi_2 )^2, \quad \Box_\bm \phi_2=0
\end{equation}
with the example constructed by John in \cite{John} which does not have a global solution for compactly supported data of any size
\begin{equation}\label{j1}
\Box_\bm \phi=-(\p_t\phi)^2.
\end{equation}
With $L=\p_t+\p_r$, $\Lb =\p_t-\p_r$, we can  decompose $2\p_t=L+\Lb$. Thus  the term of $(\Lb f)^2$ appears in the quadratic term of both  (\ref{wn}) and (\ref{j1}). The difference lies in that such bad term in $(\ref{wn})$ can be controlled by the better part of the system, since $\phi_2$ is actually a free wave solution. Although, one can only obtain the weaker decay property
\begin{equation}\label{7.14.2.18}
(r+t+1)|\phi_1|\les \ln t,
\end{equation}
 due to the appearance of such bad term, the solution $(\phi_1, \phi_2)$ exists globally. Typically, the weak null system consists of good equations which verify null conditions, and bad equations which formally  have the terms of $ \Lb f\c \Lb f$. It is important that the function $f$ verifies the good equations. Einstein equations under wave coordinates verify such weak null property (see \cite{Lindrod1} and \cite{Lindrod2}). The global results for small data are proved by Lindblad and Rodnianski. Clearly (\ref{j1}) gives an example that $\Lb f\c \Lb f$ appears in the equation of $f$ itself, which does not satisfy the weak null condition.

 For ease of discussion, let us consider the data which have compact support. In the domain of influence, by running a  standard energy argument for (\ref{j1}), we have
 \begin{equation}\label{7.17.1.18}
 \| \p\p\rp{n} \phi(t,\cdot)\|_{L^2({\mathbb R}^3)}\les  \| \p\p\rp{n} \phi(0,\cdot)\|_{L^2({\mathbb R}^3)}+ \int_0^t \|\p \rp{n} ((\p_t \phi)^2)(t',\cdot)\|_{L^2({\mathbb R}^3)} dt'.
 \end{equation}
 For simplicity we only consider one of the terms  in $\p \rp{n}( (\p_t \phi)^2) =\sum_{a+b=n}\p \rp{a}\p_t\phi\p\rp{b} \p_t\phi$, which is
 \begin{equation}\label{7.18.1.18}
\p_t \phi\c \p\rp{n}\p_t\phi.
 \end{equation}
 Note the standard decay of the free wave  for $\p_t \phi$ with small data is
     $$(|t-r|+1)^\f12(t+r+1)|\p \phi|\les \ep^\f12 .$$
 By a direct substitution, $ \| \p\p\rp{n} \phi(t,\cdot)\|_{L^2({\mathbb R}^3)}\les  t^{C\ep^\f12} \| \p\p\rp{n} \phi(0,\cdot)\|_{L^2({\mathbb R}^3)}$. However, with such energy  growth, one can not recover the linear behavior to $\p_t\phi$ without loss of decay in $t$-variable. With a weaker decay for $\p_t\phi$, we can not achieve the boundedness of energy even allowing growth. The only way to obtain the boundedness of energy is by setting $t\le T(\ve)<\infty$, which implies the semi-global result.

 If the quadratic nonlinearity verifies null conditions such as the null forms
 $$\bQ_0(\phi, \psi)=\p^\mu \phi \p_\mu \psi; \qquad  \bQ_{\a\b}(\phi,\psi)= \p_\a \phi \p_\b \psi-\p_\a \psi \p_\b \phi.$$

  Due to
 \begin{equation*}
 |\bQ(\phi, \psi)|\les| \bar \p \phi|\c |\p \psi|+|\bar\p \psi||\p \phi|
 \end{equation*}
 where $\bar \p =(L ,\sn)$  and  $\sn_i=\p_i-\frac{x^i}{r}\p_r, \, i=1,2,3$,   since the above structure is almost preserved if differentiated by the invariant vector fields of the Minkowski space,  and since the decay of $\bar \p\phi$ can be improved  to $(1+r+t)^{-2}$ by using the commuting vector fields approach, we can obtain an additional $(1+r+t)^{-1}$ decay in the error integral in (\ref{7.17.1.18}) compared with the case  for the equation (\ref{j1}). This implies the boundedness of energy easily.

Based on the above examples, clearly the presence of the  $\Lb \phi_1\c \Lb \phi_2$ type term in  the quadratic nonlinearities significantly changes the asymptotic behavior of the solution. It either does not allow the local-in-time solution to be extended, or causes the global solution to lose the sharp decay, which is the case for the simplest system (\ref{wn}). In the case that $\N^{\a\b} \equiv 0,\, q\equiv 0$,   Lindblad (\cite{Lind1}) proves the stability result for (\ref{eqn_1}) if the metric $\bg$ does not depend on $\p \phi$, where the loss of sharp decay occurs for $\phi$. (See also \cite{Alinhac1}.) So if losing the structure of null conditions either in the semilinear quadratic terms or the quasilinear terms, one should not expect the solution has the standard global linear behavior of a wave equation without loss.

 Prescribing data  on $\{r\ge R\}$ does not improve decay property in terms of $r$ parameter either. In  the standard exterior stability  results,  \cite{KN, KlN2} for Einstein equations and  \cite{mMaxwell} for the massive Maxwell-Klein-Gordon equation with arbitrary charge, the null conditions of the equations have played a crucial role.  If both the quasilinear and semilinear nonlinearity verify the null condition, the solution of (\ref{eqn_1}) should be  well-posed  in the entire domain of dependence of $\{r\ge R\}$ up to the characteristic boundary  if the data is sufficiently small.  To prove this standard result, one may have to check if the extrinsic approach of \cite{Lind1, Lindrod1, Lindrod2} is strong enough to capture the behavior of the characteristic surfaces in this situation, since it is not in the Einsteinian case. It is possible that one has to adopt the intrinsic approach such as in \cite{KN}.   Without any assumption on the nonlinear structures in (\ref{eqn_1}), one should not expect the solution exists in the entire $\{r\ge R+t\}$ for semilinear case, nor in the whole exterior of  the global  characteristic surfaces (upto the boundary) for the quasilinear case.
 For the quasilinear equations, the characteristic surface can be singular in finite time. Nevertheless, it is possible that the solution  remains regular in the majority of the domain of dependence, meanwhile concentrates in the remaining part.
This inspires us to extend the solution in subregions of the domain of dependence. We require such subregion to be  global in time. In the semilinear case, we  require the region can be  any close to $\{r\ge t+R\}$.

Note that the John's example in \cite{John2},
\begin{equation*}
\Box_\bm \phi=-\p_t\phi \Delta \phi
\end{equation*}
in the spherical symmetry case can be reduced to a Burgers' equation.
  For the latter, we know that shock can be  formed  once  certain monotonicity condition of data is broken and the shock can  occur along  any characteristic, no matter the support of the data is compact or not.
  Thus, it seems meaningless to discuss where singularities of the solutions are distributed.  However,  exactly due to this example, we can imagine that  if the classical solution is extended globally in time, their life-span along the  regular part of each characteristic surface may still be finite.
   Geometrically, the lightcone of schwarzschild spacetime  intersects with any lightcone of Minkowski space within finite time. So in the semilinear case,
if we use a schwarzschild cone initiated from $\{|x|=R\}$ as the boundary surface, this region matches the expectation indicated by the Burgers' equation. Such cone has to be spacelike since we need to obtain the positive  energy  flux on the boundary. Note that the region bounded by the schwarzschild cones with the small negative mass $-M_0$ can exhaust  $\{r> t+R\}$ by letting $M_0\rightarrow 0_+$. So such region can be any close to $\{r>t+R\}$, and identical to $\{r\ge t+R\}$ if and only if $M_0=0$. (See the second picture in  Figure \ref{fig2}.)

Physically, we separate the region where the solution may concentrate the most, along a schwarzschild cone, away from the domain of dependence. We then ask if we can achieve  any  improvement over the standard linear behavior of wave equations in the remaining region, and if the improvement is strong enough to control the nonlinearities.

\subsection{Strategy of the proofs}
The framework of our approach can be clearly seen in the proof for the semilinear case. The main idea is  based on a good combination of the boundedness of  the standard energy  and $r$-weighted energy along the foliations of  schwarzschild cones, which leads to an  improved set of asymptotic behavior  by applying the very flexible version of the weighted Sobolev inequalities in \cite{mMaxwell}.
In this part, we will  discuss the following aspects of our approach in this paper.
\begin{enumerate}
\item [(1)] The control  of the nonlinearity on the right of (\ref{eqn_1})  and the improvements compared with the known standard linear behaviors.
\item[(2)] The influence of the variable potential $q$ to our approach.
\item[(3)] The difficulties in the  quasilinear case caused by the nontrivial influence of the metric $\bg(\phi, \p \phi)$.
\item[(4)] The complexities in the application to Einstein scalar fields.
\end{enumerate}
We first explain how we treat the quadratic nonlinearity to achieve the boundedness of energy.
For the free wave equation and data $\phi[0]$ with $\E_{2,\ga_0,R}\le \ep$,  we can apply commuting vector fields to derive the standard decay property in $\{r\ge t+R\}$,
   \begin{equation}\label{7.13.2.18}
   r|\p \phi|\les \ep^\f12(r-t+1)^{-\f12\ga_0-\f12}.
   \end{equation}
       Under the null frame $L=\p_t+\p_r$, $\Lb =\p_t-\p_r$, and with $\sn_i=\p_i-\frac{x^i}{r}\p_r$,   the Cartesian component of covariant derivative on sphere $S_{t,r}$,
    the decay for $L \phi$ and $\sn \phi$ can be improved. Nevertheless the decay rate in terms of $r$ is unimprovable for $\Lb \phi$ in the region $\{r\ge t+R\}$.  If we run the standard energy argument for (\ref{j1}), we would still end up with having a finite-in-time result.

     We now consider to improve the asymptotic behavior of $\Lb \phi$  exterior to a slightly spacelike schwarzschild cone.
Let $u_0(M_0)=-r_*(M_0, R)$ where $r_*(M_0, r)$ is defined in  (\ref{5.2.3.18}). We may denote $u_0=u_0(M_0)$ for convenience. Let $u=t-r_*(M_0, r)$ and $\ub=t+r_*(M_0, r)$.
    In the region with $\{u\le u_0\}$, we  adopt foliations by level sets of $u$ and $\ub$, denoted by $\H_u$ and $\Hb_\ub$ respectively . The   standard energy for $\phi$ on $\H_u$ takes the form of
    \begin{equation}\label{7.14.3.18}
    \int_{\H_u} \frac{M_0}{r} |\Lb \phi|^2+|L \phi|^2+|\sn \phi|^2 d\mu_\H \les |u|^{-\ga_0} \E_{0,\ga_0, R},\, \ga_0>1
\end{equation}
where the $d\mu_\H$ is the area element of $\H_u$, comparable to $r^2 d\ub' d \mu_{{\mathbb S}^2}$, and $\E_{0,\ga_0, R}$ is defined in  (\ref{3.21.4.18}) for the initial data.

Next, we explain how such improvement  can essentially help us to control the error integral in the standard energy estimate for (\ref{j1}).  Let $E[\p \rp{n} \phi](\Sigma)$ denote the standard energy on the hypersurfaces $\Sigma$.   Let $ -\ub_1\le u_1\le u_0$ and $\ub_1$ be arbitrarily large. The energy argument gives
\begin{align}
&E[\p\rp{n}\phi](\H_{u_1}^{\ub_1})+E[\p \rp{n}\phi](\Hb_{\ub_1}^{u_1})\label{7.14.7.18}\\
& \les  \| \p\p\rp{n} \phi(0,\cdot)\|^2_{L^2(\Sigma_0\cap \{u\le u_1\})}+ \int_{\{ -\ub_1\le-\ub< u\le u_1 \}} |\p \rp{n} (\p_t \phi \p_t \phi) \p\p\rp{n} \phi  |.\nn
\end{align}
Here  the truncated hypersurfaces $\H_{u_1}^{\ub_1}$ and $\H_{\ub_1}^{u_1}$ are subsets of $\H_{u_1}$ and $\Hb_{\ub_1}$ respectively, both of which  are defined in (\ref{7.30.1.18}) in Section \ref{setup}.

Again we consider only the term  (\ref{7.18.1.18}) in the error terms. If we can recover the linear behavior (\ref{7.13.2.18}) with $r-t$  replaced by $|u|$ to $\p \phi$ under the assumption  that the  data verify $\E_{2,\ga_0,R}\le \ep$, we have
\begin{align}\label{7.30.2.18}
\int_{\{-\ub_1\le-\ub < u\le u_1\}} |\p \rp{n} \p_t \phi \c \p_t \phi  \c \p \p\rp{n} \phi  | \les  M_0^{-1} \ep^\f12 \int_ {-\ub_1}^{u_1} |u|^{-\f12\ga_0-\f12} E[\p\rp{n}\phi](\H_u^{\ub_1}) du.
\end{align}
  With $\ga_0>1$, we can achieve the boundedness of energy by Gronwall's inequality. For the nonlinear problem itself, we certainly can not directly use the linear behavior  to control error. The analysis will be based on a bootstrap argument. With $\ep \le C M_0^2$ and $C\ge 1$, we can close the bootstrap argument if $c_0 R^{1-\ga_0}< C^{-1}$, where $C$ is a fixed constant,  $c_0\ge 1$ is a universal constant. This is achievable if $R(\ga_0)$, the lower bound of $R$, satisfies the inequality.

Thus it is crucial for our nonlinear analysis  to achieve the linear behavior of (\ref{7.13.2.18}) with $r-t$ replaced by $|u|$,  without loss of the decay in $r$-variable. This  requires us to perform our analysis in a no-loss regime. The analysis of the full nonlinearities is  more involved than the sample term in (\ref{7.30.2.18}). We explain our basic principle below.

If we write the spacetime error integral, such as the last term of (\ref{7.14.7.18})  symbolically as
\begin{equation}\label{7.22.1.18}
\int |B_1||B_2| |B_3| dx dt.
\end{equation}
The known procedure for bounding the  error integral uniformly in the upper limit of $t$ is to identify one of the factors $|B_1|$ as $|G|$  which has a stronger decay in $r$ than (\ref{7.13.2.18}), if the standard null conditions are satisfied; or using the better feature of the equation (system) to guarantee  $|B_1|$ can achieve the standard global linear behavior, such as the pointwise decay  $(t+r+1)^{-1}$, which implies a control with growth of $t^{C\ep^\f12}$. The latter occurs when the weak null structure is available.  Without any of the extra structure, we rely on the sets of decay estimates in Section \ref{decay} and the energy flux of the schwarzschild cone foliation to form  the hierarchy of the analysis. For these bad terms $B_i$, we manage to recover the standard linear behavior, with the bound comparable to the size of the data, $\dn^\f12\approx \ep^\f12$.  They offer  good bounds $\|\cdot \|_{G}$. We also achieve a set of integrated estimates with the bound $\dn^\f12 M_0^{-\f12}$, denoted by $\|\c \|_{S}$, which are  stronger than the standard linear behavior. When applying H\"{o}lder's inequality, our basic principle  is  to bound the worst nonlinearity (\ref{7.22.1.18}) by
\begin{equation*}
\int |B_1||B_2||B_3| dx dt\les \|B_1\|_{G}\|B_2\|_{S}\|B_3\|_{S},
\end{equation*}
 which allows us to improve  the bootstrap argument and achieve the boundedness of  the norms $\|\cdot\|_{G}$ and $\|\cdot\|_{S}$ for the bad terms.

To derive the set of improved estimates $\|\cdot\|_S$ for bad terms,  we note  the linear behavior (\ref{7.14.3.18}) shows that once  the energy flux on $\H_u$ can be bounded in terms of the initial data, since the weighted energy of data is bounded, the flux  automatically  decays nicely in $|u|$. We can combine this property with the weighted Sobolev inequalities in \cite{mMaxwell} to  obtain
more improvements, in particular, on the integrated decay estimates. Below we list some of such improved estimates, which are important to the results in this paper.
\begin{align}
&\|r^\f12 \p Z\rp{b}\phi\|^2_{L_\ub^2 L_u^\infty L_\omega^4(u\le u_1)}\les \ep M_0^{-1} |u_1|^{-\ga_0+2\zeta(Z^b)}, \quad b\le n-1\label{7.14.6.18}\\
& \|r^\f12\p Z\rp{l} \phi\|^2_{L_\ub^2 L^\infty(u\le u_1)}\les \ep M_0^{-1}|u_1|^{-\ga_0+2\zeta(Z^l)},\quad l\le n-2\label{7.14.4.18}\\
&\| r^{-\f12} Z\rp{a}\phi\|^2_{L^2(\H_{u_1})}\les \ep M_0^{-1} |u_1|^{-\ga_0+2+2\zeta(Z^a)},\quad a\le n,\label{7.14.5.18}
 \end{align}
where $u_1\le u_0$ and $\E_{n,\ga_0, R}(\phi[0])\le \ep,\, n= 2,3 $. Here we denote the ordered product of vector fields as $Z^k= Z_1 \cdots Z_k$, with $Z\rp{k}$ the corresponding differential operator of $k$-th order, where  $Z_l\in \{\Omega_{ij}, \p\}$,  $\Omega_{ij}=x_i\p_j-x_j \p_i,\,  1\le i<j\le 3$.  $Z\rp{0}=id$.  The signature function is defined by
  $$\zeta(Z^k)=\sum_{i=1}^k\zeta(Z_i), \qquad \zeta(\Omega_{ij})=0, \qquad \zeta(\p)=-1.$$
  See Proposition \ref{5.20.1.18}, Proposition \ref{6.17.5.18} and Lemma \ref{6.25.2.18}   for the proofs of (\ref{7.14.6.18})-(\ref{7.14.5.18}) and more improved estimates.

  (\ref{7.14.6.18}) are crucial for treating other terms in (\ref{7.14.7.18}) so as to achieve the boundedness of energy without any loss.  The other two estimates are important for treating the quasilinear problem  without loss. (\ref{7.14.4.18}) is used to for proving the weighted energy estimate. (\ref{7.14.5.18}) is used to give an improved Hardy's inequality (\ref{6.17.2.18}), which is crucial to treat $[\widetilde\Box_{\bg}, Z\rp{n}]\phi$.

 Next  we comment on the influence of the potential $q$ to our approach.
\begin{enumerate}
\item[(1)] The scaling vectorfield  $S=t\p_t+r\p_r$ can not be used as a commuting vectorfield, which is similar to the massive case i.e.  $q>0$ is a fixed constant.
\item[(2)] The asymptotic behavior of the solution is similar to  the massless case i.e. $q=0$. Such set of decay is weaker than the standard decay of the massive  case.
\end{enumerate}
The problem with variable potential takes the essential difficulties from both the massive and massless wave equations.
To treat the potential term, we take the spirit of  the multiplier approach developed in \cite{mMaxwell},  and yet have to make  further improvements since the asymptotic behavior of the solution is much weaker.
We adopt merely $\{\Omega_{ij}, \p\}$ to obtain very good decay property for $\p \phi$,
\begin{equation}\label{8.5.1.18}
|u|^\f12 r |\Lb \phi|+ r^\frac{3}{2}|\sn \phi|+ r^\frac{3}{2} |L \phi|\les \ep^\f12 |u|^{-\frac{\ga_0}{2}}
\end{equation}
with decay of higher order derivatives included in Section \ref{setup}, Section \ref{decay} and Section \ref{quasi}.
If $q\equiv q_0$ where the constant $q_0>0$, we achieve
\begin{equation*}
q_0  r^\frac{5}{2} |\phi|^2\les \ep  u_+^{-\ga_0+\f12}.
\end{equation*}
The sharp decay for $\sn \phi$ is achieved when we commute rotation vectorfields $\Omega_{ij}$ upto the third order.  We can obtain sharp decay for $L \phi$ if we also employ the boost vector field up to the third order, which also improves the decay for $\phi$ to sharp if $q_0>0$ is a constant.\begin{footnote}{It works as well if $q\ge q_0$. }\end{footnote} Such improvements are not necessary for proving the main results.

      For the quasilinear problems, besides the difficulty caused by the semilinear quadratic error terms,  we have to solve the difficulties caused by the metric  $\bg(\phi, \p \phi)$.  In the sequel, we will explain our approach for solving the following issues.
      \begin{enumerate}
      \item [(1)]
        Due to the influence of the metric, how to define  the exterior region of  the spacetime  is actually a fundamental problem if not using the characteristic surfaces.
        \item [(2)] Technically, due to the influence of the metric, we have to modify the energy momentum tensor appropriately.
      \item [(3)] We use a multiplier approach to recover the linear behavior for the nonlinear solution. However it requires stronger decay property on $H(\phi, \p \phi)$ than the decay  for the free wave, in order to obtain the bounded $r$-weighted energy.
        \end{enumerate}
            None of these issues  arises in an intrinsic approach  such as \cite{KN}, which relies on the foliations of characteristic surfaces. They all are shifted to controlling the evolution of the geometry of the characteristic surfaces.   As explained, such intrinsic approach is not suitable for our problem.

  The first issue is linked to the positivity of the energy flux on the boundary $\H_{u_0}$.
           Even for equations verifying the standard null condition, one would encounter the same issue. By using a modified energy momentum tensor, we can compute that the energy density along $\H_u$ takes the form of
           \begin{equation}\label{7.14.1.18}
           (\frac{M_0}{r}-H^{\Lb \Lb})(\Lb \phi) ^2+H\p \phi \c \bar\p\phi+|\bar\p \phi|^2 .
           \end{equation}
            ( See the calculation in Lemma \ref{6.10.6.18}.)
   For Einstein equations, with $M_0=0$, the positivity of energy flux can be achieved if the data\begin{footnote}{See Theorem \ref{eins_thm} and Section \ref{Eins} for the set-up and the meaning of the data.}\end{footnote} are sufficiently small, since due to the positive mass theorem, 
   $$\lim_{|x|\rightarrow \infty} rH^{\Lb\Lb}(x, 0)=-\f12 m_0, \quad \quad m_0>0.$$
   In Section \ref{Eins}, we will take advantage of this fact to prove the improved result, Theorem \ref{eins_thm}.

    In general, (\ref{7.14.1.18}) is not coercively positive along the Minkowski cone, i.e.  $M_0=0$.
    To achieve the positive energy flux, in Section \ref{quasi}, we  choose $M_0$ according to the size of the data, so that there exists a small constant $M$ such that
    \begin{equation}\label{8.5.2.18}
    r(\frac{M_0}{r}-H^{\Lb\Lb})>M>0,  \mbox { and } |M_0|\les M.
    \end{equation}
    Other error terms in (\ref{7.14.1.18})  can then be absorbed. This treatment actually needs the smallness of $|r H^{\Lb \Lb }|$. Although one can see from (\ref{7.14.2.18}) that even for the system (\ref{wn}), which  has better structure  than the general case (\ref{eqn_1}), the solution $\phi$ does not have the sharp decay. But we manage to  achieve it in the region $\{u\le u_0\}$. There is a similar issue with the energy density on $\H_\ub$, which can be solved in the same way.

Therefore with a suitable choice of the mass $-M_0$ for  the boundary cone, we can  gain the control of $ M r^{-1}|\Lb \phi|^2 $ in the energy flux along the level set of $u(M_0)$, i.e. (\ref{7.14.3.18}) holds with $M_0$ replaced by the small constant $M>0$.
The choice of $M_0$ depends on the bound for the data,  so does the size of $M$. This allows us to follow the treatment for (\ref{j1}) to control the quadratic nonlinearity.

In \cite{Lind2}, in order to improve the asymptotic behavior  for the solution of  Einstein equations, the asymptotic schwarzschild coordinates and $r_*$  were employed  in the commuting vector field approach.  It was used for better approximating the wave operator,  since $\bg$ itself approaches a schwarzschild metric.  Our  foliation is  chosen to dominate over the influence of $H^{\Lb\Lb}$, thus is always slightly away from the characteristic surfaces of the asymptotic equations (See the definition in \cite{Lindrod2}). The schwarzschild cone foliation is used throughout the paper, even for treating the flat wave operator.  We use the  Minkowskian vector fields and Minkowski metric throughout the paper, so as to take advantage of the difference between the Minkowski geometry and the schwarzschild geometry.

Next, we discuss the other two issues.
Schematically, for both the semilinear and quasilinear cases, the main task is to bound the standard energy and $r$-weighted energy  in terms of initial data. For ease of discussion, we assume $q\equiv 0$ in the sequel. In the semilinear case, in terms of the standard energy momentum tensor $Q_{\a\b}=\p_\a \varphi\c \p_\b \varphi-\f12 \bm_{\a\b} \p^\mu \varphi \p_\mu \varphi$,
the standard energy is defined by $\int_\Sigma Q_{\a\b} X^\b {\mathbf n}^\a d\mu_\Sigma$, where $X=\p_t$ and  ${\mathbf n}^\a$  denotes the surface normal of $\Sigma$.
The $r$-weighted energy is defined by using $X=rL$ with suitable modifications in the energy current. The energy estimates are based on the following calculation
$$\p^\a (Q_{\a\b}X^\b)=\Box_\bm \varphi X \varphi+Q_{\a\b} \p^\a X^\b.$$
For the quasilinear operator, we have to make a modification, otherwise the righthand side contains $\p^2 \varphi$. One may adopt the intrinsic version,
\begin{equation*}
\p_\a \varphi \p_\b \varphi-\f12 \bg_{\a\b} \bg^{\mu\nu} \p_\mu \varphi\p_\nu\varphi
\end{equation*}
and lift or lower the indices by the metric $\bg$.  We construct the  energy momentum  tensor as follows,
\begin{equation*}
\tilde{\sQ}_{\a\b}[\varphi]=\p_\a\varphi\p_\b \varphi-\f12 \bm_{\a\b} \bg^{\rho \sigma} \p_\rho \varphi \p_\sigma \varphi+H_\a^\ga \p_\ga \varphi \p_\b \varphi
\end{equation*}
which is not  symmetric, nevertheless gives nice structures in the energy density under   the Minkowski background.

In the quasilinear case, the form of the energy momentum tensor, the choice of multiplier and the modification to the energy current are all very sensitive for proving the $r$-weighted energy estimates. Typically, bounding $r$-weighted energy requires more decay  than a free wave verifies. See \cite[Section 1 (3)]{Yang2015}, where an additional $r^{-\ve}$ decay is assumed, with $\ve>0$, even for the equations with null condition therein.
Our improvement is  however weaker than this assumption.
Our proof of the inequality for the weighted energy is a very delicate one. Since $\Lb \varphi$ term can not take the weight of $r$, we need to treat terms of  $f(H, \p H)(\Lb \varphi)^2$ carefully. We choose $X=r(L-H^{\Lb\Lb}\Lb),$ which is influenced by considering the asymptotic equation (see \cite{Lind1, Lindrod2, Yang2015}).  In Lemma \ref{4.15.1.18}, it turns out the construction of  energy current in (\ref{4.14.1.18}) leads to a good structure in the error terms.  Undesirable terms, such as
$
\int_{\{u\le u_0\}} r |L H^{\Lb \Lb} (\Lb \varphi)^2| dx dt,
$
are cancelled. \begin{footnote}{Decay in (\ref{8.5.1.18}) for $LH$  is not strong enough to control this term.}\end{footnote}
 We manage to use the estimate of $\|r^\f12\p H(\phi, \p \phi)\|_{L_\ub^2 L^\infty}$ in (\ref{7.14.4.18}) and the fluxes along schwarzschild cones to cope with error terms.

At last we comment on the treatment of the Einstein scalar fields. In comparison with Theorem \ref{7.30.3.18} (or Theorem \ref{thm_quasi}), $H$ converges to a small static solution instead of $0$. The static part slows down the decay properties of $H$. Fortunately for Theorem \ref{thm_quasi}, the derivation of the inequalities of energy and weighted energy relies more on the decay of  $\p H$, which is barely influenced. The framework of Section \ref{quasi} still works through. However, borderline terms appear in the commutator $[\widetilde\Box_\bg, Z\rp{n}]$, since $H$ has less decay in $|u|$. They are proved to be  harmless, when we show the boundedness of energies for $Z\rp{n}(\bh^1, \phi)$ \begin{footnote}{See Theorem \ref{eins_thm} for the definitions of $\bh^1, \phi$.}\end{footnote}with an induction on the signature $\zeta(Z^n)$ from $-n$.

 As future extensions of this work,  we believe the approach can be applied to give the global result for   the quasilinear wave systems with weak form of null conditions, if the small weighted data are prescribed throughout the initial slice. We also believe the result of Theorem \ref{thm_quasi} can be generalized to fluids with nontrivial vorticity. It would be also interesting to ask if there is any  global-in-time interior stability result for the equation (\ref{eqn_1}) with  small compactly supported data.

\subsection{Structure of the paper}In Section \ref{setup}, we give the details of the geometric set-up and introduce the main theorems, which are  Theorems \ref{thm2}, Theorem \ref{thm_quasi} and Theorem \ref{eins_thm}. In  Section \ref{prel}, we introduce the weighted Sobolev inequalities and derive some consequences of bounded standard and $r$-weighted energies, including some sharp $L^p$ type estimates in Lemma \ref{6.25.2.18}. In Section \ref{decay}, under the assumption of bounded energies upto $n$th-order, with $n=2$ or $3$, we derive the full set of decay properties in Proposition \ref{5.20.1.18} and Proposition \ref{6.17.5.18}. In Section \ref{semi}, we consider the semilinear case of (\ref{eqn_1}) and prove Theorem \ref{thm2} and \ref{thm1}. This section gives the main framework of our approach. Schematically, we divide it into three steps. We first derive the energy inequalities. Under the bootstrap  assumption of the smallness of energies up to $n=2$, we then employ the decay results in Section \ref{decay} to analyse the error $(\Box_\bm-q)Z\rp{n}\phi$. The final step is to achieve the boundedness theorem for the energies by substituting the   error estimates into the energy inequalities. In Section \ref{quasi}, we prove Theorem \ref{thm_quasi}. Due to the influence of metric, we need to make (\ref{8.5.2.18}) hold in $\{u\le  u_0\}$ which makes the bootstrap argument   more involved. We also need to obtain higher order energy control for treating  commutators $[\widetilde{\Box}_{\bg}, Z\rp{n}]$. We still run the same three steps as for the semilinear case. The analysis is more delicate in each part. In Section \ref{Eins}, we prove Theorem \ref{eins_thm}. We need to show the metric difference, which is one part of the solution is  convergent to a small static solution, while the scalar field converges to $0$ as $r\rightarrow \infty$.  By a simple reduction, we still solve the problem with data convergent  to $({\mathbf 0},0)$ at the spatial infinity.  However,  the static part in  $H$ has slower decay property. This may change the behavior of the wave operator. In Proposition \ref{7.28.3.18},  we confirm the inequalities for energy and weighted energy still hold under such background metric. We then analyze commutators in Lemma \ref{7.28.2.18}, which contain borderline terms. For the error terms not included  in (\ref{eqn_1}), we treat them in Lemma \ref{7.8.15.18}. At last we combine the  error estimates in Section \ref{quasi} to complete the proof.

\textbf{Acknowledgments.} The author is partially supported by RCF fund from the University of Oxford. The author would like to thank Pin Yu who mentioned the Klein-Gordon equation with small mass to the author around 2015, and wishes to thank Shiwu Yang for friendly and enlightening conversations.

\section{Set-up and main results}\label{setup}
We first construct the foliations that will play a very crucial role in improving the asymptotic behavior in this paper. We also need the construction to determine the stability region in the main results.

Let $|M_0|\ll 1$ be a constant. We set
$$h=M_0/r, \qquad  L'=L-h \Lb, \qquad \Lb'=\Lb-hL, \quad \forall\,  r\ge 1 .$$
We first give the optical function of the following metric
\begin{equation}\label{metric}
\mathfrak{g}=-\frac{r+M_0}{r-M_0}dt^2+\frac{r-M_0}{r+M_0} dr^2+ r^2 (d\theta^2+\sin ^2\theta d\phi^2).
\end{equation}
Note this metric has the same  lightcones initiated from $\{|x|=r, t=0\}$ as  the schwarzschild metric of the mass $-M_0$
\begin{equation*}
\bg_s=-\frac{r+M_0}{r-M_0}dt^2+\frac{r-M_0}{r+M_0} dr^2+ (r-M_0)^2 (d\theta^2+\sin^2\theta d\phi^2).
\end{equation*}
Suppose $u$ is an optical function of the metric $\mathfrak{g}$.
\begin{equation*}
\p_t u =\pm\frac{r+M_0}{r-M_0} \p_r u.
\end{equation*}
Let $t=\ga(r)$ be the null geodesic initiated from the sphere of radius $R_1$ at $t=0$. $u(\ga(r), r)=C$.
\begin{equation*}
\p_t u \dot{\ga}(r)=\p_r u.
\end{equation*}
Thus
\begin{equation*}
\dot{\ga}{}^\pm(r)=\pm\frac{r-M_0}{r+M_0}.
\end{equation*}
Then $\ga^\pm(r)=\ga^\pm(R_1)\pm (r-R_1-2M_0\ln \frac{r+M_0}{R_1+M_0})$. For convenience, we can regard $R_1=1$. Thus by setting
\begin{equation}\label{5.2.3.18}
r_*(M_0,r)=r-2M_0 \ln \frac{r+M_0}{1+M_0}
\end{equation}
we have $\ga^+(r)=\ga^+(1)+ r_*(r)-1$. And we can  regard the outgoing lightcone of the metric (\ref{metric}), initiated from $\{r=R_1\}$  as a ruled surface generated by $t=\ga^+(r)$,  $\forall\, \omega\in {\mathbb S}^2$. Identically, it  also is the level set of $t-\ga^+(r)=0.$  
We can set up the foliation of schwarzschild lightcones in $({\mathbb R}^{3+1}, \bm)$ with the help of the pair of  optical functions of (\ref{metric})
 $$ u=t-r_*,  \qquad \ub=t+r_*.$$
 Similarly, we can check that $\{\ub=C\}$ is the incoming null cone of $\mathfrak{g}$, which is a smooth ruled surface by incoming null geodesics.
Clearly, $u(0,1)=-1$ and $-u(0,r)\approx r$ when $r\ge 2$.\begin{footnote}{We assume $r\ge 2$ throughout the paper if not stated otherwise.}\end{footnote}
It is direct to compute the generators of the outgoing and incoming null geodesics
\begin{equation}\label{7.15.1.18}
-\mathfrak{g}^{\a\b}\p_\a u\p_\b=(1+h)^{-1} L', \qquad  -\mathfrak{g}^{\a\b} \p_\a \ub \p_\b= (1+h)^{-1} \Lb'.
\end{equation}
which are tangent to $\H_u$ and $\Hb_\ub$ respectively, and coincide with our construction.
We denote by $\N$, $\Nb$  the surface normals of $\H_u$ and $\Hb_\ub$ in terms of the Minkowski metric, which are normalized  in terms of  $\l \N, \p_t\r=-1$ and $\l \Nb, \p_t\r=-1$.  In view of (\ref{7.15.1.18}), it is easy to compute that
\begin{equation}\label{7.11.3.18}
\N=(1+h)^{-1} (L+h\Lb), \qquad \Nb=(1+h)^{-1}(\Lb+h L).
\end{equation}
Let
\begin{equation}\label{7.17.2.18}
 u_0(M_0)=u_{M_0}(0,R)=-r_*(M_0,R)
 \end{equation}
 with the fixed constant $R\ge 2$.  In case there occurs no confusion, we may use $u_0$ as a shorthand notation. We now consider the region in $({\mathbb R}^{3+1}, \bm)$ where $u\le u_0$.  By setting  $1+\bb^{-1}=\Lb (u)$, we can easily calculate the lapse function
 \begin{equation}\label{6.28.2.18}
\bb^{-1} =\frac{1-h}{1+h}.
 \end{equation}
Instead of using $t$ to parameterize $\H_u$ and $\Hb_\ub$, we will use $\ub$ and $u$.
It is straightforward to compute
 \begin{eqnarray*}
Lu =\Lb \ub=\frac{2M_0}{r+M_0}=1-\bb^{-1}&&L \ub =\Lb u = 2-\frac{2M_0}{r+M_0}=1+\bb^{-1}.
\end{eqnarray*}
Thus we can obtain for $ u_1\le u_0$, $-\ub_1\le u_0$ and $\omega\in{\mathbb S}^2$,
\begin{equation}\label{4.30.3.18}
\frac{d}{d\ub}= \frac{r}{2(r-M_0)} L' \mbox{ on } \H_{u_1}, \quad \frac{d}{du}= \frac{r}{2(r-M_0)} \Lb' \mbox{ on } \Hb_{\ub_1}.
\end{equation}
This implies
\begin{equation}\label{7.19.3.18}
\p_\ub  r=\f12 \bb \mbox{ on } \H_{u_1}\qquad \p_u r=-\f12 \bb \mbox{ on } \Hb_{\ub_1}.
\end{equation}
 By using $u, \ub$ level sets to foliate the spacetime, the standard area element is
$$d xdt= (2{r_*}'(r))^{-1} r^2 du d\ub d\omega= \f12 \bb r^2 du d\ub d\omega,$$
where $d\omega$ denotes the standard surface measure on the unit sphere ${\mathbb S}^2$.
Thus in view of (\ref{6.28.2.18}), the  area elements of $\H_u$ and $\Hb_\ub$ are
\begin{equation}\label{7.4.3.18}
d\mu_\H=\f12\bb r^2 d\ub d\omega, \qquad d\mu_{\Hb}=\f12 \bb r^2 du d\omega.
\end{equation}
Let $S_{u,\ub}=\H_u\cap \H_\ub$. where $\ub\ge -u$. For smooth functions $f$,   $\int_{S_{u,\ub}} f = \int_{S_{u,\ub}} r^2 f d\omega$.
\begin{footnote}{We may hide the standard area elements for the integral on the corresponding hypersurfaces or spheres, and hide the area element $dx dt$ if the integral is in a domain of the spacetime.}\end{footnote}   Clearly, $\bb$ is an increasing function of $h$,  $\bb=1+O(h)$. Thus
the area elements in (\ref{7.4.3.18})  are comparable to $\f12 r^2 d\ub d\omega$ and $\f12 r^2 du d\omega$.
Note that $\p_r u=-\bb^{-1}=-\p_r \ub$ on $\Sigma_t$, the area element on $\Sigma_t$ is $\bb r^2 du d\omega$. Thus on $\Sigma_t$, $ r^2 d\ub d\omega \approx dx\approx r^2 du d\omega$.

By the definition of $u$,  $ 1-\frac{t-u}{r}=M_0 O(\frac{\ln r}{r})$.
 Thus  we can derive
  \begin{equation}\label{6.28.3.18}
  r(S_{u, -u})=-u(1+o(M_0)), \qquad r(S_{-\ub,\ub})=\ub(1+o(M_0)).
  \end{equation}
  where the second identity  is an application of the first one, based on the fact that $\H_{-\ub}$ is  initiated from $S_{-\ub,\ub}$.

  We also have the basic fact that $r(u,\ub,\omega)$ is increasing about $\ub$ for fixed $(u, \omega)$ and decreases as $u$ increases if $(\ub, \omega)$ is fixed. Note also that  $r_*(r)\le \ub \le 2 r_*(r)-r_*(R)$. These two facts together with (\ref{6.28.3.18}) imply
  \begin{equation}\label{7.3.1.18}
  -u \les r_{\min}(\H_u), \qquad \ub \approx r\approx r_*(r).
  \end{equation}
 (\ref{6.28.3.18}), (\ref{7.3.1.18}) and the fact that $\ub\ge -u$   will be frequently used in our analysis, probably  without being  mentioned.

We  use $\H_{u}^{\ub}$ and $\Hb_{\ub}^{u}$ to denote the truncated level sets of $u$ and $\ub$ respectively.
\begin{equation}\label{7.30.1.18}
\begin{split}
&\H_{u}^{\ub}:=\{(t, x):  -u\le  \ub'\leq \ub\},\quad \Hb_{\ub}^{u}:=\{(t, x): -\ub \le u'\le u \},\\
&\D_u^\ub=\{(t,x):  -\ub\le-\ub'\le  u'\le u \},
\end{split}
\end{equation}
where $-\ub\le u\le u_0 $.

We  denote
$$\Sigma_0^{u_1, \ub_1}=\{-\ub_1\le  u \le u_1,  t=0\}=\{-u_1\le \ub\le \ub_1, t=0\},$$  and  may drop $\ub_1$ when  $\ub_1=\infty$.

We denote by  $E[f](\Sigma)$ and $\W_1[f](\Sigma)$  the   energy (flux) and weighted energy (flux) of the smooth function $f$ on the hypersurface $\Sigma$.   For the hypersurfaces  of interest to us,
\begin{equation}\label{3.19.1}
\begin{split}
 &E[f](\Sigma_0^{u, \ub})=\int_{\Sigma_0^{u, \ub}}|\p f |^2+q f^2 dx,  \\
  &E[f](\H_u^\ub)=\int_{\H_u^\ub}\f12 r^2\big(|L f|^2+\frac{M}{r} |\Lb f|^2+|\sn f|^2+q f^2 \big) d\ub' d\omega ,\\
 &E[f](\Hb_\ub^u)=\int_{\Hb_\ub^u}\f12r^2 (|\Lb f|^2+\frac{M}{r} |L f|^2+|\sn f|^2+q f^2) du' d\omega,\\
 &\W_1[f](\H_u^\ub)=\int_{\H_u^\ub} \f12 \big(r(L(r f))^2+r^3 \frac{M}{r}(|\sn f|^2+q f^2)\big)  d\ub' d\omega,\\
 &\W_1[f](\Hb_\ub^u)=\int_{\Hb_\ub^u}\f12  r^3 \big(|\sn f|^2+\frac{M}{r}|L f|^2 +q f^2\big) du' d\omega,\\
 & \W_1[f](\D_u^\ub)=\int_{\D_u^\ub} r^{-2}|L(r f)|^2+|\sn f|^2 dx dt,\\
 &\W_1[f](\Sigma_0^{u,\ub})=\int_{\Sigma_0^{u,\ub}} r^{-1} (L(r f))^2 + r(|\sn f|^2+q f^2) dx,
\end{split}
\end{equation}
where $M>0$ is a fixed constant to be specified.
Throughout the paper, $M$ and $h$ are chosen such that
\begin{equation}\label{7.1.1.18}
r|h|\le  M.
\end{equation}

Throughout the paper,  we set $u_+=-u$ and let  $Z\rp{n}$ be the $n$-th order differential operator $Z_1\cdots Z_n$, with  each $Z_m\in \{\p, \Omega_{ij}, 1\le i<j\le 3\}$.
  We are ready to state the main results of this paper.
\begin{theorem}\label{thm2}
Consider
\begin{equation}\label{3.18.1.18}
 \Box_\bm \phi=\N^{\a\b}(\phi) \p_\a \phi  \p_\b \phi+q(x)\phi
 \end{equation}
 with $0\le q \le 1$  satisfying (\ref{potential}) for $n=2$.
Let $1<\ga_0<2$  and $C>1$ be fixed constants. There exist  a small constant $0< \delta_1\ll\frac{1}{100}$ and a  universal constant $R(\ga_0, C)\ge 2$ such that  for  any  $0<M\le  \delta_1^\f12$,  if the initial data set $\phi[0]$ on $\{r\ge R\}$ with $R\ge R(\ga_0, C)$  verifies  $\E_{2,\ga_0, R}\le C M^2$,
 there exists a unique solution  in the entire region $\{u(M)\le u_0(M)\}$ for all $t>0$.  Here the function  $u(M)= t-r_*(M,r)$, with $r_*(M,r)$ defined in (\ref{5.2.3.18}), and $u_0(M)$ is defined in (\ref{7.17.2.18}). There hold for any    $-\ub\le u\le u_0(M)$  that
\begin{equation}\label{6.28.5.18}
\begin{split}
&E[Z\rp{n} \phi](\H_u^\ub)+E[Z\rp{n }\phi](\Hb_\ub^u)\les \E_{2,\ga_0, R}u_+^{-\ga_0+2\zeta(Z^n)},\\
&\W_1[Z\rp{n} \phi](\H_u^\ub)+\W_1[Z\rp{n} \phi](\Hb_\ub^u)+\W_1[Z\rp{n} \phi](\D_u^\ub)\les \E_{2,\ga_0, R}u_+^{1-\ga_0+2\zeta(Z^n)},
\end{split}
\end{equation}
where $n\le 2$,  $u$ and $\ub$ are shorthand notations for $u(M)$ and $\ub(M)$, $Z\in \{\Omega_{ij}, \p\}$.  There also hold the pointwise estimates
\begin{equation*}
r^3 |\sn \phi|^2+ r^2 u_+|\Lb  \phi|^2+r^3|L \phi|^2\les \E_{2, \ga_0, R} u_+^{-\ga_0}.
\end{equation*}

\end{theorem}

\begin{theorem}\label{thm1}
Consider (\ref{3.18.1.18}) with $0\le q\le 1$ satisfying  (\ref{potential}) for $n=2$.
Let $R\ge 2$ and $1<\ga_0<2$ be fixed. There exist small constants $0< \delta_1\ll\frac{1}{100}$, $\delta_0>0$ such that for any $0<M\le \delta^\f12_1$ ,  if the initial data set verifies
\begin{equation}\label{6.28.6.18}
\E_{2,\ga_0, R}\le \delta_0 M^2,
\end{equation}
 there exists a unique solution in  the entire region of  $\{u(M)\le u_0(M)\}$ for all $t>0$.
There hold the same set of energy estimates as in (\ref{6.28.5.18}) and the pointwise estimates for any  $-\ub\le u\le u_0(M)$.
\end{theorem}
\begin{remark}
 Let $S_M=\{u(M)\le -r_*(R, M)\}$, which is the exterior region to the schwarzschild outgoing cone initiated from $\{r=R\}$ including the boundary.  Clearly $S_{M_1}\subset \S_{M_2}$ if $M_1> M_2$.
Note that as $M\rightarrow 0$, $u(M)\rightarrow t-r$ and $S_M$ approaches the entire open exterior of $\{r-t=R\}$.  Then Theorem \ref{thm2} indicates that, for any $0<M\le \delta_1^\f12$ ,
the stability result can always holds in $S_M$ for  the set of non-compactly supported data  with the norm (\ref{3.21.4.18}) bounded by $M^2$, provided that  $R\ge R(\ga_0)$.   See Figure \ref{fig2}. \end{remark}

\begin{remark}
 Theorem \ref{thm1} is a consequence of the proof of Theorem \ref{thm2} under  a stronger assumption on data. One can refine it by assuming part of the energy norm of data verifies (\ref{6.28.6.18}).
\end{remark}

\begin{theorem}\label{thm_quasi}
Consider (\ref{eqn_1})  which verifies (\ref{potential}) for $n=3$.  There exist   a universal constant $C\ge 1$,   a small constant  $0< \delta_1<\frac{1}{100}$ and a constant  $R(\ga_0,C)\ge 2$, such that, if the initial data set $\phi[0]$ satisfies that
\begin{equation}\label{5.2.4.18}
\begin{split}
& \E_{3,\ga_0, R} \le \delta_1, \quad  \mbox{ with } R\ge R(\ga_0, C)
   \end{split}
\end{equation}
    there  exists a unique solution  in the entire region   $\{ u(M)\le u_0(M)\}$ with  $M= C\delta_1^\f12$. The solution verifies the following energy estimates for $-\ub \le u\le u_0(M),$
  \begin{equation*}
\begin{split}
&E[Z\rp{n} \phi](\H_u^\ub)+E[Z\rp{n} \phi](\Hb_\ub^u)\les \E_{3,\ga_0, R}u_+^{-\ga_0+2\zeta(Z^n)},\\
&\W_1[Z\rp{n} \phi](\H_u^\ub)+\W_1[Z\rp{n} \phi](\Hb_\ub^u)+\W_1[Z\rp{n }\phi](\D_u^\ub)\les \E_{3,\ga_0, R}u_+^{1-\ga_0+2\zeta(Z^n)},
\end{split}
\end{equation*}
where $n\le 3$, $Z\in \{\Omega_{ij}, \p\}$ and  verifies the decay estimates
\begin{equation*}
r^3 |\sn Z\rp{l}\phi|^2+ r^2 u_+|\Lb Z\rp{l} \phi|^2+r^3|LZ\rp{l} \phi|^2\les \E_{3, \ga_0} u_+^{-\ga_0+2\zeta(Z^l)}, l\le 1.
\end{equation*}
\end{theorem}
As an application, we provide the following exterior stability results for Einstein scalar fields under the wave coordinates.
\begin{theorem}\label{eins_thm}
Consider the Einstein scalar field system
\begin{equation}\label{7.12.2.18}
\left\{\begin{array}{lll}
&\bR_{\a\b}(\bg)=\p_\a \phi  \p_\b \phi+\f12 q_0 \, \bg_{\a\b} \phi^2,\\
&\Box_\bg \phi=q_0\phi
\end{array}\right.
\end{equation}
where the constant $q_0\ge 0$. Under the wave coordinate gauge, we set $\bh^1_{\mu\nu}=\bg_{\mu\nu}-\bm_{\mu\nu}-\frac{m_0}{r}\delta_{\mu\nu}$. For constants $1<\ga_0<2$ and $C_0\ge 1$, there exist a small constant $\delta_1>0$ and a constant $R(\ga_0,C_0)\ge 2$, such that,
  if the initial data set $(\bh^1[0], \phi[0])$ \begin{footnote}{We assume they satisfy the constraint equations.  See details of the data construction in \cite[(2.3)-(2.5),Page 1410]{Lindrod2}.}\end{footnote} verifies $$\E_{3, \ga_0, R, 0}(\bh^1[0])+ \E_{3, \ga_0, R, q_0}(\phi[0]) \le C_0 m^2_0, \quad 0<m_0<\delta_1$$
  where $R\ge R(\ga_0, C_0)$,
  then there exists a unique solution  $\Psi=(\bh^1, \phi)$ for (\ref{7.12.2.18})  in the entire region of $\{u(-m)\le u_0(-m)\}$, where $0\le m<\f12 m_0$ is a fixed constant.\begin{footnote}{In our proof, we fix $m=\frac{m_0}{20}$ for convenience.}\end{footnote}  With $u, \ub$ the shorthand notations for $u(-m)$ and $\ub(-m)$, for $-\ub\le u\le u_0(-m)$, there  hold
    \begin{equation*}
\begin{split}
&E[Z\rp{n} \Psi](\H_u^\ub)+E[Z\rp{n} \Psi](\Hb_\ub^u)\les m_0^2 u_+^{-\ga_0+2\zeta(Z^n)},\\
&\W_1[Z\rp{n }\Psi](\H_u^\ub)+\W_1[Z\rp{n }\Psi](\Hb_\ub^u)+\W_1[Z\rp{n }\Psi](\D_u^\ub)\les m_0^2 u_+^{1-\ga_0+2\zeta(Z^n)},
\end{split}
\end{equation*}
where $n\le 3$ and $Z\in \{\Omega_{ij}, \p\}$. There also hold the decay estimates,
\begin{align*}
&r^3 |\sn Z\rp{l}\Psi|^2+ r^2 u_+|\Lb Z\rp{l} \Psi|^2+r^3|LZ\rp{l} \Psi|^2\les m_0^2 u_+^{-\ga_0+2\zeta(Z^l)}, l\le 1,\\
&q_0|r^\frac{5}{4} Z\rp{l}\phi|^2 \les m_0^2u_+^{-\ga_0+2\zeta(Z^l)+\frac{1}{2}},  l\le 1.
\end{align*}
The constants in the above inequalities are independent of $q_0^{-1}$.
\end{theorem}
\begin{figure}[ht]
\centering
\includegraphics[width = 0.48\textwidth]{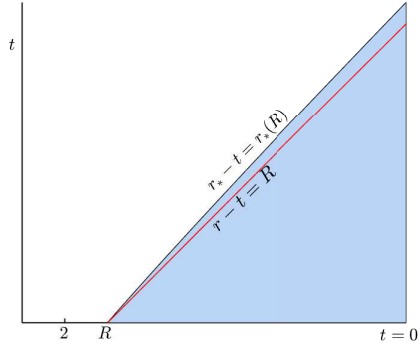}
  \vskip -0.7cm
  \caption{Illustration of the stability zone of Theorem \ref{eins_thm}}\label{fig3}
\end{figure}
\begin{remark}
The energy estimates and pointwise decay in the above four theorems  hold true if $Z$ belongs to the generators of Poincar\'{e} group. If $q\equiv 0$, one can also extend the result to the set of vector fields $\{ \p, x^\mu \p_\mu, \Omega_{\mu\nu}\}$.\begin{footnote}{$\Omega_{\mu\nu}=x_\mu\p_\nu-x_\nu\p_\mu, \quad 0\le  \mu< \nu\le3 \mbox{ where }x_\mu=\bm_{\mu\nu}x^\nu.$}\end{footnote}
\end{remark}
\begin{remark}
We do not rely on the weak null condition of the Einstein scalar field equation (\ref{7.12.2.18}) to prove Theorem \ref{eins_thm}.
If we further use the precise weak null structure of the reduced Einstein equation, the result can be proved upto $R= 2$ under a weak extra smallness assumption of data.
\end{remark}
\section{Preliminary estimates}\label{prel}
 In this section, we adapt the Sobolev inequalities developed for the canonical null hypersurfaces in \cite{mMaxwell} to $u$ and $\ub$ foliations.  With the help of this set of Sobolev inequalities, we provide preliminary estimates in the region of $\{u\le u_0(M_0)\}$ for functions  bounded in terms of the energy norms in (\ref{3.19.1}). Some of the estimates, such as (\ref{5.22.5.18}) and (\ref{6.14.1.18}) in Lemma \ref{6.25.2.18}  are stronger than the known estimates for the free wave. They are crucial for the proof of boundedness of energies.
We also provide estimates on the initial slice in this section.

For ease of exposition, we denote by  $\Omega$  any of the rotation vector fields in $\{\Omega_{ij}, 1\le i<j\le 3\}$ and by $\Omega\rp{k} f$ any of the $k$-th order derivatives by the rotation vector fields. $|\Omega f|^2=\sum_{1\le i<j\le 3}|\Omega_{ij} f|^2$.  $|\Omega\rp{ k}f|^2$ is the sum of all the combinations of $k$-th order derivatives by rotation vector fields. The same convention applies to $|P\Omega\rp{l} f|$  if  $P$ is a linear differential operator.

We adapt from \cite[Section 2.1]{mMaxwell} to obtain the following Sobolev inequalities.
\begin{lemma}[Sobolev inequalities]
For any smooth function $f$ and constants verifying the relation $2\ga=\ga_0'+2\ga_2$, we have, for all $-\ub_1\le -\ub\le  u\le u_0(M_0)$,
\begin{align}
\sup_{S_{u, \ub}}|r^{\ga} f|^4
&\les \sum_{l\leq 1} \int_{S_{u, -u}}|r^\ga \Omega\rp{l} f|^4 r^{-2}+\sum_{k\leq 2}\int_{\H_u^{ \ub_1}}r^{2\ga_2} |\Omega\rp{k} f|^2 r^{-2}\nn\\
&\cdot \sum_{l\leq 1}\int_{\H_u^{ \ub_1}} r^{\ga_0'}|L' \Omega\rp{l}(r^{\ga} f)|^2 r^{-2},\label{6.24.11.18}\\
\int_{S_{u, \ub}}|r^\ga f|^4 r^{-2}
&\les \int_{S_{u, -u}}|r^\ga f|^4 r^{-2} + \int_{\H_u^{ \ub_1}}r^{\ga_0'}|L'(r^\ga f)|^2 r^{-2} \cdot \sum_{l\le 1}\int_{\H_u^{ \ub_1}}r^{2\ga_2}|\Omega\rp{l}  f|^2r^{-2}.\label{6.24.12.18}
\end{align}
The same estimates hold by using the incoming null hypersurface $\Hb_\ub^u$. In this case $L'$ is replaced by $\Lb'$, and the initial sphere is $S_{-\ub, \ub}$.
\end{lemma}
This lemma can be proved similarly as in \cite{mMaxwell} with the help of (\ref{4.30.3.18}) and $|M_0|\ll 1$.

Let (\ref{7.1.1.18}) hold. We first give the following results in the initial slice.
\begin{proposition}\label{6.23.4.18}
Let $1<\ga_0<2$ and the constant $R\ge 2$ be fixed. With $n\le 3$,  there hold on $\Sigma_0\cap \{r\ge R\}$ the following estimates.

 (1)
\begin{align}
&\int_{S_r} r^{1+\ga_0-2\zeta(Z^i)} | Z\rp{i}\phi|^2 d\omega\les \E_{i, \ga_0}, \quad i \le n, \label{6.23.7.18}\\
&\int_{S_r} r^{2+2\ga_0-4\zeta(Z^i)}|Z\rp{i}\phi|^4 d\omega \les \E^2_{i,\ga_0},\quad    i\le n, \label{6.24.4.18}
\end{align}
where  $S_r=\{|x|=r\}$.

(2) Let $u_1, \ub_1$ be a pair of fixed numbers verifying   $-\ub_1\le u_1\le u_0$.
There hold
\begin{align}
&\int_{-u_1}^{\ub_1}(\int_{S_{-\ub, \ub}} r^{2-4\zeta(Z^i)} | Z\rp{i} \phi|^4 d\omega)^\f12 d\ub\les {u_1}_+^{-\ga_0+1}\E_{i, \ga_0}, \quad i \le n\label{6.24.10.18},\\
&\int_{-u_1}^{\ub_1}(\int_{S_{-\ub, \ub}} r^{2 -4\zeta(Z^{i-1})} |\p Z\rp{i-1} \phi|^4 d\omega)^\f12 d\ub\les {u_1}_+^{-\ga_0-1}\E_{i, \ga_0}, \quad i \le n\label{6.24.6.18}.
\end{align}
 The same estimates hold if the domain of integrals are changed to $\int^{u_1}_{-\ub_1} (\int_{S_{u,-u}} \cdot  d\omega)^\f12 du$  for the same integrands.
\begin{align}
&\| r^{-\f12-\zeta(Z^i)}  Z\rp{i} \phi\|^2_{L^2(\Sigma_0^{u_1, \ub_1})}\les {u_1}_+^{-\ga_0+1}\E_{i-1, \ga_0}, \quad i \le n+1 \label{6.25.1.18},\\
&\W_1[Z\rp{i}\phi](\Sigma_0^{u_1, \ub_1})\les {u_1}_+^{-\ga_0+1+2\zeta(Z^i)} \E_{i, \ga_0}, \quad i\le n.\label{6.30.1.18}
\end{align}
(3)
For $i\le n$, there holds on $\Sigma_0\cap \{r\ge R\}$ that
\begin{align}
&r^2|Z\rp{i-1}\phi(u,-u,\omega)|^2\les u_+^{-\ga_0+1+2\zeta(Z^{i-1})}\E_{i, \ga_0} \label{6.24.9.18}.
\end{align}
\end{proposition}
\begin{proof}
Note that for $i\le n$, due to $\E_{n, \ga_0, R}<\infty$,
\begin{equation}
\lim\inf_{r\rightarrow \infty} \int_{S_r} r^{\ga_0-2\zeta(Z^i)} (| Z\rp{i} \phi|^2+r^2|\p Z\rp{i} \phi|^2)d\omega =0. \label{6.23.2.18}
\end{equation}
By using the Sobolev embedding on ${\mathbb S}^2$, we have for any scalar function $F$,
\begin{align*}
(\int_{S_r} r^4 |F|^4 d\omega)^\f12\les \|\Omega F\|^2_{L^2(S_r)}+\|F\|^2_{L^2(S_r)}.
\end{align*}
Thus, with $F= Z\rp{i}\phi$, $i\le n$, by using (\ref{6.23.2.18}), we have
\begin{equation}\label{6.23.9.18}
\lim\inf_{r\rightarrow \infty} \int_{S_r} r^{2\ga_0-4\zeta(Z^{i})} | Z\rp{i} \phi|^4 d\omega =0.
\end{equation}
Now consider (\ref{6.23.7.18}) with the help of (\ref{6.23.2.18}). By integrating back from the spacelike infinity,
\begin{align*}
\int_{S_r} r^{1+\ga_0-2\zeta(Z^i)} | Z\rp{i}\phi|^2 d\omega&\le  r^{1+\ga_0-2\zeta(Z^i)}\int_{r}^\infty |\p_r Z\rp{i}\phi| |Z\rp{i} \phi|d\omega dr'\nn\\
&\les \| \p Z\rp{i} \phi\c r^{1+\frac{\ga_0}{2}-\zeta(Z^i)}\|_{L^2_r L_\omega^2}\| Z\rp{i} \phi\c r^{\frac{\ga_0}{2}-\zeta(Z^i)}\|_{L^2_r L_\omega^2}\nn\\
&\les \E_{i,\ga_0},
\end{align*}
which gives (\ref{6.23.7.18}).

 By using (\ref{6.23.9.18}) and H\"{o}lder's inequality, we derive
\begin{align}
\int_{S_r}  | Z\rp{i}\phi|^4 d\omega&\le  \int_{r}^\infty |\p_r Z\rp{i}\phi| | Z\rp{i} \phi|^3d\omega dr'\nn\\
&\les  \| \p Z\rp{i} \phi \|_{L^2_r L_\omega^2\{ r'\ge r\}}\|| Z\rp{i} \phi|^3 \|_{L^2_r L_\omega^2\{ r'\ge r\}}.\label{6.24.3.18}
\end{align}
Note that  by using \cite[Page 58 (3.2.4a)]{CK}
\begin{equation*}
\int_r^\infty |F|^6 d\omega dr' \les \int_r^\infty \int_{{\mathbb S}^2}|F|^4 d\omega \int_{{\mathbb S}^2} (|F|^2+|r\sn F|^2 )d\omega  dr'.
\end{equation*}
We then combine the above inequality with (\ref{6.24.3.18}) to obtain
\begin{align}
\sup_{r'\ge r} (\int_{S_{r'}}  |  Z\rp{i}\phi|^4 d\omega)^\f12 &\les  \| \p Z\rp{i} \phi \|_{L^2_r L_\omega^2\{ r'\ge r\}}\|| Z\rp{i}\phi|+|r\sn Z\rp{i} \phi| \|_{L^2_r L_\omega^2\{ r'\ge r\}}\nn\\
&\les r^{-(1-2\zeta(Z^{i})+\ga_0)} \E_{i,\ga_0}.\nn
\end{align}
This gives  (\ref{6.24.4.18}).

  Note that $\ub_1\ge \ub\ge -u_1$, and (\ref{6.28.3.18}) implies $r(S_{-\ub,\ub})\ge  \f12 \ub $. By integrating  (\ref{6.24.4.18}), we can obtain (\ref{6.24.10.18}), which implies (\ref{6.24.6.18}) immediately.   (\ref{6.25.1.18}) follows as a direct integration of (\ref{6.23.7.18}).

(\ref{6.25.1.18}) can be derived  by using  the first identity in (\ref{6.28.3.18}) and the definition (\ref{3.21.4.18}). (\ref{6.30.1.18}) is a consequence of (\ref{6.25.1.18}).

Next, we prove (\ref{6.24.9.18}).  Let $r_1\ge R$.  We adapt (\ref{6.24.11.18}) to $\Sigma_0\cap\{r\ge R\}$ with $\ga_0'=\ga=1$ and  $\ga_2=\f12$. This implies for $r\ge r_1$,
\begin{align*}
\sup_{S_r} |r Z\rp{i-1}\phi|^4&\les  \lim\inf_{r\rightarrow \infty}\int_{S_r} |r^\ga \Omega\rp{\le 1} Z\rp{i-1} \phi|^4 r^{-2} +\int_{\Sigma_0\cap\{ r\ge r_1\}} r^{2\ga_2}|\Omega\rp{\le 2} Z\rp{i-1} \phi|^2 r^{-2}\\
&\cdot \int_{\Sigma_0\cap \{r\ge r_1 \}} r^{\ga_0'}|\p_r\Omega\rp{\le 1} (r^\ga Z\rp{i-1}\phi)|^2 r^{-2}\les r_1^{-2\ga_0+2+4\zeta(Z^{i-1})}\E^2_{i,\ga_0},
\end{align*}
where due to (\ref{6.24.4.18}) and $\ga_0>1$,  the first term on the right vanished, and we also used the fact that $|\Omega f |\les r|\p f|$ to treat the term of $\Omega\rp{\le 2} Z\rp{i-1} \phi$. Thus, in view of  (\ref{6.28.3.18}), (\ref{6.24.9.18}) is proved.

\end{proof}

 The energy or weighted energy norms in (\ref{3.19.1}) not only give  control on $\p \varphi$,  they also control  $\varphi$ itself, which is given in the following result.
\begin{lemma}
Let   $-\ub\le u\le u_0(M_0)$ and $\a>0$ be fixed. There hold the following estimates
\begin{align}
\int_{\Hb_\ub^u} r\varphi^2 du' d\omega&+\int_{\D_u^\ub} \{r^2 (L\varphi)^2+\a(\frac{u_+}{r})^\a\varphi^2\} du' d\ub' d\omega\nn\\
&\les \W_1[\varphi](\D_u^\ub)+M\int_{-\ub}^{u} {u'}_+^{-1}E[\varphi](\H_{u'}^\ub) du'+\int_{\Sigma^{u,\ub}_0} r^{-1} \varphi^2  dx,\label{3.23.2.18}\\
\int_{\H_u^\ub } r|L'(r\varphi)|^2 d\omega d\ub'&\les \W_1[\varphi](\H_u^\ub )+M E[\varphi](\H_u^\ub)+M^2\int_{\H_u^{\ub}} r^{-1} | \varphi|^2 d\omega d\ub'  \label{3.23.4.18},\\
\int_{S_{u,\ub}} r \varphi^2 d\omega&+ \int_{\Hb_\ub^u} \varphi^2 du' d\omega\les \int_{S_{-\ub,\ub}} r \varphi^2 d\omega +E[\varphi](\Hb_\ub^u),\label{4.29.6.18}
\end{align}
\begin{equation}\label{5.1.1.18}
\begin{split}
\int_{\H_u^\ub} \varphi^2 d\ub' d\omega &\les \int_{S_{u, -u}}r \varphi^2 d \omega +\int_{S_{-\ub, \ub}} r \varphi^2 d\omega +E[\varphi](\H_u^\ub)+E[\varphi](\Hb_\ub^u).
\end{split}
\end{equation}
\end{lemma}
\begin{proof}
We first prove
\begin{align}
\int_{\Hb_\ub^u }r \varphi^2 du' d\omega&+\int_{\D_u^\ub} \{r^2 (L'\varphi)^2+\a(\frac{u_+}{r})^\a\varphi^2\} du' d\ub' d\omega\nn\\
&\les \int_{\D_u^\ub} |L'(r\varphi)|^2 du'd\ub' d\omega+\int_{\Sigma^{u,\ub}_0} r \varphi^2  du' d\omega\label{3.23.3.18}.
\end{align}
  Due to $L'r=1+h$, (\ref{4.30.3.18}) and (\ref{7.19.3.18}),  by directly computing   $|L'(r\varphi)|^2$, we obtain
\begin{align*}
 \frac{(L'(r\varphi))^2 }{2(1-h)(1+h)}& =\frac{r^2(L' \varphi)^2 }{2(1-h)(1+h)}+ r \p_\ub(\varphi^2)+\frac{1+h}{2(1-h)}\varphi^2\\
 &=\frac{r^2(L' \varphi)^2 }{2(1-h)(1+h)}+ \p_\ub (r \varphi^2).
\end{align*}
Integrating the above identity in $\D_u^{\ub}$  and  using the  smallness of $|h|$ imply
\begin{align}
\int_{\Hb_{\ub}^u} r \varphi^2 d\omega du' +\int_{\D_u^{\ub}} r^2 (L' \varphi)^2 du' d\ub' d\omega \les \int_{\Sigma_0^{u, \ub}} r \varphi^2 du' d\omega+\int_{\D_u^\ub} |L'(r\varphi)|^2 du' d\ub' d\omega. \label{7.19.1.18}
\end{align}
By   using $r\approx \ub$ in (\ref{7.3.1.18}),
\begin{equation*}
\int_{\D_\ub^u }r^{-
\a} \varphi^2 d\omega du'd\ub'\les\a^{-1} u_+^{-\a} \sup_{-u\le \ub'\le \ub} \int_{\Hb_{\ub'}^u} r \varphi^2 d\omega du'.
\end{equation*}
Since (\ref{7.19.1.18}) holds for any $-u\le \ub'\le \ub $, we can conclude  (\ref{3.23.3.18}).

We can prove (\ref{3.23.2.18}) by using (\ref{3.23.3.18}).   Note that $L'(r\varphi) =L(r \varphi)+h \Lb(r \varphi)$.
We have
\begin{equation*}
\int_{\D_u^\ub} |L'(r\varphi)-L(r\varphi)|^2 du' d\ub' d\omega \les \int_{\D_u^\ub} |rh\Lb \varphi-h\varphi|^2 du' d\ub' d\omega.
\end{equation*}
Due to (\ref{7.3.1.18}), we have
\begin{align*}
 \int_{\D_u^\ub} r^2 h^2 (\Lb \varphi)^2 du'd\ub' d\omega\les M^2\int_{\D_u^\ub} r^{-1} (\Lb \varphi)^2 r  du' d\ub' d\omega\les M\int^u_{-\ub} {u'}_+^{-1} E[\varphi](\H^\ub_{u'}) du'.
\end{align*}
Due to (\ref{7.1.1.18}),
$
\int h^2 \varphi^2 d u d\ub d\omega\les \int M^2 {\ub'}^{-3} r\varphi^2 du' d\ub' d\omega.
$
This term can  absorbed by the first term on the left of  (\ref{3.23.3.18}) by using Gronwall's inequality.
Thus can obtain (\ref{3.23.2.18}) by using (\ref{3.23.3.18}).

(\ref{3.23.4.18}) is by direct calculation.


Next we prove (\ref{4.29.6.18}). By using (\ref{7.19.3.18}) we can derive on $\Hb_\ub^u$,
$
\p_u (r^\f12 \varphi)\c  r^\f12 \varphi=r\varphi\p_u \varphi-\frac{1}{4} \bb\varphi^2.
$
By integrating  the above identity along $\H_u^\ub$ with area element $ d u' d\omega$ and by using (\ref{4.30.3.18}), we can obtain
\begin{align*}
\int_{S_{u, \ub}} r\varphi^2 d\omega+\frac{1}{2}\int_{\Hb_\ub^u} \bb\varphi^2  d\omega du'&=\int_{S_{-\ub,\ub}} r \varphi^2d\omega+ \int_{\Hb_\ub^u} \frac{r^2}{(r-M_0)}\Lb' \varphi\c   \varphi d\omega d u'.
\end{align*}
By using Cauchy Schwartz inequality and $\bb>\frac{3}{4}$ due to the smallness of $|h|$,
\begin{align*}
\int_{S_{u, \ub}} r\varphi^2 d\omega + \int_{\Hb_\ub^u} \varphi^2 d\omega du' \les \int_{S_{-\ub,\ub}} r \varphi^2d\omega+ E[\varphi](\Hb_\ub^u).
\end{align*}
This proves (\ref{4.29.6.18}). Next, we prove (\ref{5.1.1.18}) in the same fashion.  It is direct to derive along $\H_u$
\begin{equation*}
\p_\ub (r^\f12 \varphi) r^\f12 \varphi=r \p_\ub\varphi\c \varphi+ \frac{1}{4} \bb \varphi^2.
\end{equation*}
In view of the above identity, by using  (\ref{4.30.3.18}), we integrate along $\H_u^\ub$ with the area element $ d\ub' d\omega$ to derive
\begin{align*}
\frac{1}{2}\int_{\H_u^\ub} \bb\varphi^2 d\omega d\ub'&= \int_{S_{u,\ub}} r \varphi^2 d\omega-\int_{S_{u, -u}} r\varphi^2 d\omega-\int_{\H_u^\ub} \frac{r^2}{r-M_0}  L' \varphi \c \varphi d\ub' d\omega .
\end{align*}
We then combine the estimate of (\ref{4.29.6.18}) and Cauchy Schwartz inequality  to derive
\begin{align*}
\int_{\H_u^\ub} \varphi^2 d\omega d\ub'&\les \int_{S_{-\ub, \ub}} r \varphi^2 d\omega+\int_{S_{u,-u}} r\varphi^2 d\omega+E[\varphi](\H_u^\ub)+E[\varphi](\Hb_\ub^u)
\end{align*}
as desired in  (\ref{5.1.1.18}).
\end{proof}

\subsection{Simple facts of vector fields}
Before proceeding further,  we give basic facts about the vector fields.

 Let $\sn$ be the covariant derivative on $S_{u,\ub}$. Its component under the Cartesian frame $\p_i, i=1,2,3$  takes the form
 of  $ \sn_i=\p_i-\omega^i \p_r, \, \omega^i=x^i/r.$
 We set
$\bar \p=(\sn, L), \,  \ud \p=(\sn, \Lb).$ For smooth functions $f$, we have
\begin{enumerate}
\item[(1)]
By direct calculation, there hold
\begin{align}
&[\p_\rho, \Omega_{\mu\nu}]=\bm_{\rho\mu}\p_\nu-\bm_{\rho \nu}\p_\mu, \quad  0\le \mu<\nu\le 3,\label{8.1.1.18}\\
&[\Omega_{ij}, \sn_l ]f=-\delta_l^i \sn_j f+\delta_l^j \sn_i f,\label{3.24.3.18}\\
&[L,\Omega_{ij}]=0=[\Lb, \Omega_{ij}]=[\p_r, \Omega_{ij}],\quad [L, \Lb]=0,\label{3.24.1.18}\\
&[L, \sn] f=-r^{-1} \sn f, \quad [\Lb, \sn ]f =r^{-1} \sn f, \quad [\p_r, \sn] f= -r^{-1} \sn f.\label{3.24.5.18}
\end{align}
\item[(2)]
For $X=L, \Lb, L', \Lb'$,  due to $X\omega^i=0$, there hold
\begin{equation}\label{3.24.14.18}
|X(r Lf)|+|X(r\Lb f)|
\les |X(r \p f)|; \qquad |X L f|+|X\Lb f|\les |X \p f|.
\end{equation}
\item[(3)]
\begin{align}
  &|\sn \Lb f|+|\sn L f|+ |\sn \p_r f| \les |\sn \p f|+r^{-1} |\sn f|,\label{3.24.2.18}\\
 &|\p f|^2 +|\Omega \p_\mu f|^2\approx  |\p f|^2+|\p_\mu \Omega f|^2,\, \, \mu=0,\cdots, 3,\label{6.2.9.18}\\
&\Omega\rp{\ell} \D_* f=\D_*\Omega\rp{\le \ell} f,\, \quad \mbox{ if }\D_*\in \{\sn, \ud \p, \bar \p\} \, \quad \ell\in {\mathbb N},\label{3.24.4.18}
\end{align}
where  $\Omega$  means one of $\{\Omega_{ij}, 1\le i<j\le 3\}$.
\end{enumerate}
Indeed, it is direct to check
\begin{equation}\label{7.20.1.18}
\sn_l \omega^i=r^{-1}(\delta_l^i -\omega^l \omega^i);\qquad  \Omega_{lj}\omega^i =r^{-1}( x^l \delta_j ^i-x^j \delta_l^i ),\qquad 1 \le  l<j\le 3.
\end{equation}
  (\ref{3.24.2.18}) follows by using the first identity.  To see (\ref{6.2.9.18}), in view of (\ref{8.1.1.18})
  $
|  [\Omega, \p_i]f|\les |\p f|,
  $
by also using that   $\Omega\p_0=\p_0 \Omega$,  (\ref{6.2.9.18}) is proved.

By using (\ref{3.24.3.18}), we can obtain
\begin{align*}
\Omega_{ij}\Omega_{mn} \sn_l f&=\sn_l \Omega_{ij} \Omega_{mn} f+[\Omega_{ij}, \sn_l]\Omega_{mn} f+\Omega_{ij}[\Omega_{mn}, \sn_l]f\\
&=\sn_l \Omega_{ij} \Omega_{mn} f+ \sn \Omega_{mn} f+\sn \Omega_{ij}\rp{\le 1} f.
\end{align*}
Higher order calculation can be done by induction. This implies (\ref{3.24.4.18}) with $\D_*=\sn$ for  $\ell\in {\mathbb N}$. We then combine this calculation with (\ref{3.24.1.18}) to conclude (\ref{3.24.4.18}) if $\D_*$ is one of the derivatives $\sn, \ud \p, \bar \p$.

\subsection{$L^4$ type estimates}
In order to  give the product estimates for the nonlinear terms of (\ref{eqn_1}), we will rely on $L^4$ type estimates. They can be  derived by using the Sobolev inequality (\ref{6.24.12.18}) and the energies in (\ref{3.19.1}).
\begin{lemma}\label{6.25.2.18}
Let $-\ub_1\le u_1\le u_0(M_0)$  and $\a>0$ be fixed.  For smooth functions $F$ and $\psi$, there hold the following estimates,
\begin{align}
\|r^\f12 F\|^2_{L_u^2 L_\omega^4(\Hb_{\ub_1}^{u_1})}&\les \W_1[ F](\Hb_{\ub_1}^{u_1})+\W_1[F](\D_{u_1}^{\ub_1})+\| r^{-\f12} F \|^2_{L^2(\Sigma_0^{u_1, \ub_1})}\nn\\
&\quad +M \int_{-\ub_1}^{u_1} u _+^{-1} E [F](\H_u^{\ub_1}) du,  \label{6.3.9.18}\\
\a\|(\frac{{u_1}_+}{r})^\a F\|^2_{L_u^2 L_\ub^2 L_\omega^4(\D_{u_1}^{\ub_1})}&\les \W_1[F](\D_{u_1}^{\ub_1})+\|r^{-\f12} F\|^2_{L^2(\Sigma_0^{u_1,\ub_1})}\nn\\
&+M\int_{-\ub_1}^{u_1} u_+^{-1}E[F](\H_u^{\ub_1}) du,\label{6.14.7.18}\\
\| r^\f12\p \psi\|^2_{L_u^2 L_\ub^2 L_\omega^4(\D_{u_1}^{\ub_1})}&\les M^{-1} \int_{-\ub_1}^{u_1}  E[\Omega\rp{\le 1}  \psi] (\H_u^{\ub_1}) du, \label{6.14.1.18}\\\
\int_{S_{u_1,\ub_1}} r^2 |F|^4 d\omega &\les \int_{S_{u_1, -u_1}}r^2 |F|^4 d\omega +(E[F](\H_{u_1}^{\ub_1})+\int_{\H_{u_1}^{\ub_1}} |F|^2 d\omega d\ub)^2,
\label{5.22.2.18}\\
\|r^\f12 \p \psi\|^2_{L_\ub^2 L_u^\infty L_\omega^4(\D_{u_1}^{\ub_1})}&\les \int_{-u_1}^{\ub_1} (\int_{S_{-\ub , \ub }} r^2 |\p\psi|^4 d\omega)^\f12  d\ub+ M^{-1/2}(\int_{-\ub_1}^{u_1}  E[\Omega\rp{\le 1} \psi](\H_u^{\ub_1}) du)^\f12\nn \\
&\c(M^{-1} \int_{-\ub_1}^{u_1}  E[\p \psi](\H_u^{\ub_1}) du +{u_1}_+^{-2}(\sup_{-u_1\le \ub <\ub_1} E[\psi](\Hb_\ub^{u_1})\nn\\
&+\sup_{-\ub_1\le u\le u_1}E[\psi](\H_u^{\ub_1})))^\f12 \label{5.22.5.18},\\
\int_{S_{u_1,\ub_1}}|r \ud\p \psi|^4 d\omega &\les \int_{S_{-\ub_1, \ub_1}}|r \ud\p \psi|^4 d\omega  +\big(E[\p \psi](\Hb_{\ub_1}^{u_1})+{u_1}_+^{-2} E[\psi](\Hb_{\ub_1}^{u_1})\big)\nn\\
&\qquad\c E[\Omega\rp{\le 1}\psi](\Hb_{\ub_1}^{u_1}),\label{6.27.1.18}\displaybreak[0]\\
\int_{S_{u_1,\ub_1}}|r L\psi|^4 d\omega &\les \int_{S_{u_1, -u_1}} |rL\psi|^4 d\omega+(E[\p \psi](\H_{u_1}^{\ub_1})+{u_1}_+^{-2} E[\psi](\H_{u_1}^{\ub_1}))\nn\\
&\qquad\c E[\Omega\rp{\le 1}\psi](\H_{u_1}^{\ub_1}),\label{6.28.1.18}\\
\| r \ud\p \psi\|_{L_u^2 L_\omega^4(\Hb_{\ub_1}^{u_1})}^2&\les E[\Omega\rp{\le 1} \psi](\Hb_{\ub_1}^{u_1}),\label{6.13.1.18}\\
\|r L \psi\|^2_{L_u^2 L_\omega^4(\Hb_{\ub_1}^{u_1})}&\les \int_{-\ub_1}^{u_1} (\int_{S_{u, -u}} |r L\psi|^4 d\omega)^\f12 du+ \big(\int_{-\ub_1}^{u_1} E[\Omega\rp{\le 1}\psi](\H_{u}^{\ub_1}) du \big)^\f12\nn\\
&\c \big(\int_{-\ub_1}^{u_1}( E[\p \psi](\H_{u}^{\ub_1})+{u}_+^{-2} E[\psi](\H_{u}^{\ub_1})) du \big)^\f12.\label{6.28.7.18}
\end{align}
\end{lemma}
\begin{proof}
We first derive directly from the Sobolev inequality on unit sphere for $-\ub\le u\le u_0$,
\begin{align*}
\|r^\f12 F\|_{L_\omega^4(S_{u,\ub})}&\les \|r^\f12\sn F\|_{L^2(S_{u,\ub})}+\|r^{-1/2} F\|_{L^2(S_{u,\ub})}.
\end{align*}
Integrating the above inequality in $u$ variable along $\Hb_{\ub_1}$ implies
\begin{equation*}
\|r^\f12 F\|_{L_u^2 L_\omega^4(\Hb_{\ub_1}^{u_1})}\les \W_1[F]^\f12(\Hb_{\ub_1}^{u_1})+\|r^{-\f12} F\|_{L^2(\Hb_{\ub_1}^{u_1})}.
\end{equation*}
We then use (\ref{3.23.2.18}) to obtain (\ref{6.3.9.18}).

By the Sobolev embedding on the unit sphere
\begin{equation*}
\|F\|_{L_u^2 L_\ub^2 L_\omega^4(\D_{u_1}^{\ub_1})}\les \|\sn F\|_{L^2(\D_{u_1}^{\ub_1})}+\|r^{-1} F\|_{L^2(\D_{u_1}^{\ub_1})}
\end{equation*}
and (\ref{3.23.2.18}), we can obtain (\ref{6.14.7.18}).

Similarly, by using the Sobolev inequality that
\begin{equation*}
\| r^\f12\p \psi\|_{L_u^2 L_\ub^2 L_\omega^4(\D_{u_1}^{\ub_1})}\les \|r^\f12 \sn \p \psi\|_{L^2(\D_{u_1}^{\ub_1})}+\|r^{-\f12}\p \psi \|_{L^2(\D_{u_1}^{\ub_1})},
\end{equation*}
 $|\sn f|\approx  r^{-1}|\Omega f| $ and (\ref{6.2.9.18}), we can obtain (\ref{6.14.1.18}).

To prove (\ref{5.22.2.18}), we set  in (\ref{6.24.12.18}) $\ga_0'=1, \ga=\f12, \ga_2=0$.
\begin{align}
&\int_{S_{u_1,\ub_1}} r^2 |F|^4 d\omega  \label{5.22.4.18}\\
&\les \int_{S_{u_1,-u_1}} r^2 |F|^4 d\omega +\int_{\H_{u_1}^{\ub_1}} r|L'(r^\f12 F)|^2 d\ub d\omega\c \int_{\H_{u_1}^{\ub_1}}(|F|^2+r^2 |\sn F|^2 ) d\ub d\omega.\nn
\end{align}
Note that
$
r^\f12 L'(r^\f12 F)=r L'F+\f12(1+h)F.
$
We can derive in view of  the smallness of $|h|$ in (\ref{7.1.1.18}) that
\begin{equation*}
\int_{\H_{u_1}^{\ub_1}} |r^\f12 L'(r^\f12 F)|^2 d\ub d\omega \les E[F](\H_{u_1}^{\ub_1})+\int_{\H_{u_1}^{\ub_1}} |F|^2 d\omega d\ub.
\end{equation*}
Substituting the above inequality  into (\ref{5.22.4.18}) implies (\ref{5.22.2.18}).

Now we prove (\ref{5.22.5.18}).  Note that by taking $\ga_0'=0 $ and $\ga_2=\f12=\ga$ in (\ref{6.24.12.18}), we can derive for any smooth scalar function $F$ and $-u_1\le \ub\le \ub_1$ that
\begin{align*}
\|r^\f12 F\|^2_{L_u^\infty L_\omega^4(\Hb_\ub^{u_1})} &\le (\int_{S_{-\ub, \ub}} r^2 |F|^4 d\omega)^\f12+(\int_{\Hb_\ub^{u_1} } |\Lb' (r^\f12 F)|^2 du d\omega)^\f12 \nn\\
&\c (\int_{\Hb_\ub^{u_1} } (|F|^2 +|\Omega F|^2 ) r  du d\omega)^\f12.
\end{align*}
We then apply the above inequality to $F=\p \psi$, followed with integrating in $\ub$ variable.
\begin{align}
\|r^\f12 \p \psi\|^2_{L_\ub^2 L_u^\infty L_\omega^4(\D_{u_1}^{\ub_1})}&\les \int^{\ub_1}_{-u_1} (\int_{S_{-\ub , \ub}} r^2 |\p\psi|^4 d\omega)^\f12  d\ub\nn \\
&+(\int_{-u_1}^{\ub_1}  \int_{\Hb_\ub^{u_1} } |\Lb' (r^\f12 \p \psi )|^2 du d\omega d\ub)^\f12\nn \\
&\c (\int_{-u_1}^{\ub_1}\int_{\Hb_\ub^{u_1} } (|\p \psi|^2 +|\Omega \p \psi|^2 ) r  du d\omega d\ub)^\f12. \label{6.2.8.18}
\end{align}
Note also that
$|\Lb' (r^\f12 \p \psi)|\les r^\f12|\Lb' \p\psi|+r^{-\f12} (1+h)|\p \psi|$ and the smallness of $|h|$ imply
\begin{equation}\label{6.2.6.18}
\begin{split}
\int_{\D_{u_1}^{\ub_1}} |\Lb'(r^\f12 \p \psi)|^2 d u d\omega d\ub&\les M^{-1} \int^{u_1}_{-\ub_1} E[\p \psi](\H_u^{\ub_1}) du\\
& +{u_1}_+^{-2}\sup_{-u_1\le \ub\le \ub_1} \big(E[\psi](\Hb_\ub^{u_1})+\sup_{-\ub_1\le u\le u_1} E[\psi](\H_u^{\ub_1})\big),
\end{split}
\end{equation}
where the last line is the bound for $\int_{\D_{u_1}^{\ub_1}} r^{-3}|\p \psi|^2$. It is achieved  by using  $|\p \psi|\les |\ud \p \psi|+|\bar \p \psi|$, (\ref{3.19.1}) and  (\ref{7.3.1.18}),  followed with integrating in $u$ or $\ub$ variable.
We then substitute (\ref{6.2.9.18}) and (\ref{6.2.6.18}) to (\ref{6.2.8.18}), which implies (\ref{5.22.5.18}).

Noting $|\Lb'(r\ud \p \psi)|\les |r \Lb' \ud \p\psi|+|\ud \p\psi|$, (\ref{6.27.1.18}) can be proved by using (\ref{6.24.12.18}) for $\ud\p \psi$ along $\Hb^{u_1}_{\ub_1}$ with $\ga'_0=0, \ga_2=\ga=1$, with the help of (\ref{3.24.4.18}).

Note (\ref{3.24.14.18}) and the smallness of $|h|$ imply $|L'(r L\psi)|\les |L\psi|+r |L \p \psi|$. Applying (\ref{6.24.12.18}) to $L \psi$ along $\H_{u_1}^{\ub_1}$ with the same combination of weight exponents, we can similarly obtain (\ref{6.28.1.18}) with the help of (\ref{3.24.4.18}).

 The Sobolev embedding on ${\mathbb S}^2$ gives
\begin{align*}
\| r\ud\p \psi\|_{L_u^2 L_\omega^4(\Hb_{\ub_1}^{u_1})}&\les\| r\Omega\ud\p \psi\|_{L_u^2 L_\omega^2(\Hb_{\ub_1}^{u_1})}+ \| r \ud\p \psi\|_{L_u^2 L_\omega^2(\Hb_{\ub_1}^{u_1})}.
\end{align*}
 Then with a direct substitution of (\ref{3.24.4.18}), we can obtain (\ref{6.13.1.18}).

 (\ref{6.28.7.18}) follows by  a direct integration of (\ref{6.28.1.18}) in $-\ub_1\le u\le u_1$.

\end{proof}

\section{Decay estimates}\label{decay}
In this section, we provide in Proposition \ref{5.20.1.18} and Proposition \ref{6.17.5.18} the decay properties for a smooth function $\phi\in {\mathbb R}^{3+1}$ with bounded energies.

To be more precise,  let  $n=2$ or $3$ be fixed. We  suppose $\phi$ verifies $\E_{n,\ga_0, R}(\phi[0])<\infty$ for a fixed $1<\ga_0<2$ and
 the following assumptions:
\begin{itemize}
 \item Let   $\ub_*> -u_0(M_0)$  be a fixed number and $\dn>0$ be a fixed small constant. There hold with   $0\le l\le n$ that
\begin{equation}\label{7.1.2.18}\tag{$\BA{n}$}
\begin{split}
&E[Z\rp{l} \phi](\H_u^\ub)+E[Z\rp{l} \phi](\Hb_\ub^u)\le 2\dn u_+^{-\ga_0+2\zeta(Z^l)}, \\
&\W_1[Z\rp{l} \phi](\H_u^\ub)+\W_1[Z\rp{l} \phi](\Hb_\ub^u)+\W_1[Z\rp{l}\phi](\D_u^\ub)\le 2\dn u_+^{-\ga_0+1+2\zeta(Z^l)},
\end{split}
\end{equation}
for all $-\ub_* \le -\ub\le u\le u_0$.
\end{itemize}
We first derive a set of estimates, including the pointwise decay estimates, integrated decay estimates and the improved Hardy's inequality.

\begin{proposition}\label{5.20.1.18}
Let $n=2$ or $3$ be fixed.
Under the assumption of $(\BA{n})$,

(1) For any point $(u, \ub, \omega),\omega \in {\mathbb S}^2$ with $-\ub_*\le -\ub\le u\le u_0$, there hold for $l\le n-2$,
\begin{align}
&r^3 |\sn Z\rp{l}\phi|^2+ r^2 u_+|\Lb Z\rp{l} \phi|^2+r^3|LZ\rp{l} \phi|^2\les (\E_{l+2, \ga_0}+\dn) u_+^{-\ga_0+2\zeta(Z^l)},\label{3.24.18.18}\\
& |r Z\rp{l}\phi(u,v, \omega)|^2\les (\E_{l+2,\ga_0}+\dn) u_+^{-\ga_0+1+2\zeta(Z^l)} .\label{3.24.11.18}
\end{align}
(2) Let $(u_1, \ub_1) $ be any pair of numbers verifying $-\ub_*\le -\ub_1\le u_1\le u_0$. With  $a\le n$ and $l\le n-2$,  there hold
\begin{align}
&\|r^\f12 \Lb Z\rp{l}\phi, r^\f12 \p Z\rp{l}\phi\|^2_{L_\ub^2 L^\infty (\D_{u_1}^{\ub_1})}\les(\E_{l+2,\ga_0}+ \dn M^{-1}) {u_1}_+^{-\ga_0+2\zeta(Z^l)},\label{4.2.1.18}\\
&\| r^{-\f12} Z\rp{a}\phi\|_{L^2(\H_{u_1}^{\ub_1})}^2\les {u_1}_+^{-\ga_0+2+2\zeta(Z^a)}(\dn M^{-1}+\E_{a,\ga_0}).\label{6.17.1.18}
\end{align}
With $p>-\frac{\ga_0}{2}+\frac{3}{2}$, there holds for $a\le n$ that
\begin{equation}\label{6.17.2.18}
\|u_+^{-p} r^{-\f12} Z\rp{a} \phi\|_{L^2(\D_{u_1}^{\ub_1})}\les {u_1}_+^{-\frac{\ga_0}{2}+\frac{3}{2}-p+\zeta(Z^a)}(\dn^\f12 M^{-\f12}+\E^\f12_{a,\ga_0}).
\end{equation}
\end{proposition}
\begin{remark}
 The estimates  with $M^{-1}$ appeared in the bound are  stronger than the standard estimates for the free wave in the region $\{r\ge t+R\}$.
  (\ref{4.2.1.18})-(\ref{6.17.2.18}) are not used for the proofs for Theorem \ref{thm2} and \ref{thm1}. They will be  used in Section \ref{quasi}.  In particular, (\ref{6.17.2.18}) is an improved Hardy's inequality, which is proved by using the sharp improved estimate (\ref{6.17.1.18}). The estimate (\ref{6.17.2.18}) takes a weight of $r^\f12 $ upto the top order derivative, which is much stronger than the standard Hardy's type inequality. Such type of estimates will be crucial for the result for the quasilinear equations.
\end{remark}
\begin{remark}
We emphasize that  we only need $(\BA2)$ and $\E_{2, \ga_0}(\phi[0])$ to be  bounded  in order to obtain the above results with $n=2$. There involves neither the third order control from $(\BA3)$ nor the  bound  of $\E_{3,\ga_0}$.
\end{remark}
\begin{remark}
We can also prove
$$\| (\frac{r}{u_+})^\f12 L (rZ\rp{l}\phi)\|^2_{L_\ub^2 L^\infty(\D_{u_1}^{\ub_1})} \les (\E_{l+2,\ga_0}+\dn)  {u_1}_+^{-\ga_0+2\zeta(Z^l)}, l\le n-2 $$ which is not used for the proofs in this paper.
\end{remark}
\begin{proof}
We first consider the  inequality for $\Lb Z\rp{l}\phi$ in (\ref{3.24.18.18}).  With $\ga_2=1=\ga$, $\ga_0'=0$ in (\ref{6.24.11.18}) and  $f=\Lb Z\rp{l}\phi$,  we have
\begin{align*}
\sup_{S_{u,\ub}} |r^\ga f|^4 &\les  \int_{S_{-\ub ,\ub}} |r^\ga \Omega\rp{\le 1} f|^4 r^{-2} +\int_{\Hb_\ub^u} r^{2\ga_2}|\Omega\rp{\le 2} f|^2 r^{-2}\\
&\cdot \int_{\Hb_\ub^u } r^{\ga_0'}|\Lb' \Omega\rp{\le 1} (r^\ga f)|^2 r^{-2}.
\end{align*}
Note that $\Lb' \Omega\rp{m} (r^\a \Lb f)=\Lb' (r^\a\Lb \Omega\rp{m} f)=-\a(1+h) r^{\a-1}\Lb \Omega\rp{m} f+r^\a \Lb' \Lb\Omega\rp{m} f $ holds for any smooth function $f$. Due to the smallness of $|h|$ and (\ref{3.24.14.18}),
\begin{equation}\label{6.26.1.18}
|\Lb' \Omega\rp{m}(r^\a \Lb f)|\les r^{\a-1}|\Lb \Omega\rp{m} f|+r^\a|\Lb \p \Omega\rp{m} f|, \qquad \a \ge 0.
\end{equation}
 With the help of (\ref{3.24.1.18}), we can  derive  in view of (\ref{6.26.1.18}) with $\a=1$,
 \begin{align*}
 \sup_{S_{u,\ub}} |r \Lb Z\rp{l} \phi|^4&\les \int_{S_{-\ub, \ub}} |r \Omega\rp{\le 1} \Lb Z\rp{l} \phi|^4 r^{-2} +\int_{\Hb_\ub^u }|\Lb \Omega\rp{\le 2}Z\rp{l}\phi|^2\\
 &\cdot \int_{\Hb_\ub^u}\{|\Lb \Omega\rp{\le 1} Z\rp{l}\phi|^2+|r\Lb \p \Omega\rp{\le 1} Z\rp{l}\phi|^2 \}r^{-2} \\
 &\les \int_{S_{-\ub, \ub}} |r \Omega\rp{\le 1} \Lb Z\rp{l}\phi|^4 r^{-2}+E[\Omega\rp{\le 2}Z\rp{l} \phi](\Hb_\ub^u)\\
 &\c \big(u_+^{-2}E[\Omega \rp{\le 1} Z\rp{l}\phi](\Hb_\ub^u)+E[\p \Omega\rp{\le 1} Z\rp{l}\phi](\Hb_\ub^u)\big)\\
 &\les (\E^2_{l+2,\ga_0}+\dn^2) u_+^{-2-2\ga_0+2\zeta(Z^l)},
 \end{align*}
 where we employed (\ref{6.24.4.18}) with $Z^i= \p \Omega^{\le 1}_{ij} Z^l$ for $l\le n-2$, (\ref{7.3.1.18}) and $(\BA{n})$.

 Applying (\ref{6.24.11.18}) along $\Hb_\ub^u$ with $\ga_2=\ga=\frac{3}{2}$ and $\ga_0'=0$ to $f= \sn Z\rp{l}\phi$ implies
 \begin{equation}\label{3.24.6.18}
 \begin{split}
 \sup_{S_{u,\ub}}|r^\frac{3}{2} \sn Z\rp{l}\phi|^4&\les \int_{S_{-\ub, \ub}}|r^\frac{3}{2} \Omega\rp{\le 1}\sn Z\rp{l} \phi|^4r^{-2} +\int_{\Hb_\ub^u} r^3 |\Omega\rp{\le 2} \sn Z\rp{l}\phi|^2 r^{-2}\\
 & \c \int_{\Hb_\ub^u} |\Lb'(r^{\frac{3}{2}}\Omega\rp{\le 1}\sn Z\rp{l} \phi)|^2 r^{-2}.
\end{split}
 \end{equation}
 By using (\ref{3.24.3.18}) and  (\ref{3.24.5.18}), symbolically,
for $m=0,1$, in view of $\Lb'(r)=-1-h$,
\begin{align*}
|\Lb' (r^\frac{3}{2} \Omega\rp{m}\sn f)|&\les r^\frac{1}{2} |\sn \Omega\rp{\le m} f|+r^\frac{3}{2}|\Lb'(\sn \Omega\rp{m} f+[\Omega\rp{m}, \sn ]f)|\nn\\
&\les r^\f12|\sn \Omega\rp{\le m} f|+r^\frac{3}{2} |\sn \Lb \Omega\rp{\le m}f|+r^\frac{3}{2} |h||L \sn \Omega\rp{\le m}_{ij} f|\nn\\
&\les r^\f12 |\sn \Omega\rp{\le m} f|+r^\frac{3}{2} |\sn \p\Omega\rp{\le m} f|,
\end{align*}
where we used the smallness of $|h|$, (\ref{3.24.5.18}) and (\ref{3.24.2.18}).  By using the above calculation for $f=Z\rp{l}\phi$ and  (\ref{3.24.4.18}), we deduce from  (\ref{3.24.6.18}) and (\ref{6.24.4.18}) that
 \begin{align*}
 \sup_{S_{u,\ub}}|r^\frac{3}{2} \sn Z\rp{l}\phi|^4&\les \int_{S_{-\ub, \ub}}|r^\frac{3}{2} \sn\Omega\rp{\le 1}Z\rp{l}\phi|^4r^{-2} + \W_1[\Omega\rp{\le 2}  Z\rp{l}\phi](\Hb_\ub^u)\nn \\
 & \c \big( \W_1[\p \Omega\rp{\le 1} Z\rp{l}\phi](\Hb_\ub^u)+u_+^{-1} E[\Omega\rp{\le 1} Z\rp{l}\phi](\Hb_\ub^u)\big)\\
 &\les u_+^{-2\ga_0+2\zeta(Z^l)}(\dn^2+\E^2_{l+2,\ga_0}),
  \end{align*}
  where we also used (\ref{7.3.1.18}). Thus the first two estimates in (\ref{3.24.18.18}) are proved.

   To prove the third one, we first need to prove
  \begin{equation}\label{7.20.3.18}
  \sup_{S_{u,\ub}}r| L(r Z\rp{l}\phi)|^2\les (\E_{l+2,\ga_0}+\dn) u_+^{-\ga_0+2\zeta(Z^l)},\,  l\le n-2,
  \end{equation}
    since so far we  can not bound $\|r^\f12 L Z\rp{l}\phi\|_{L^2(\H_u^\ub)}$  without loss in $r$. Instead of applying the Sobolev inequality (\ref{6.24.11.18}) to $LZ\rp{l}\phi$, we apply it to $f=r^{-1} L(r Z\rp{l}\varphi)$ with $\ga_2=\ga=\frac{3}{2}$ and $\ga_0'=0$.
\begin{align}
\sup_{S_{u,\ub}}|r^\frac{3}{2}\c r^{-1}L(rZ\rp{l}\phi) |^4&\les \int_{S_{u,-u}} |r^\f12\Omega\rp{\le 1} L(r Z\rp{l} \phi)|^4 r^{-2} +\int_{\H_u^\ub}  r|\Omega\rp{\le 2}\big( r^{-1}L(r Z\rp{l}\phi)\big)|^2 \nn\\
&\cdot \int_{\H_u^\ub} |L'\Omega\rp{\le 1} (r^\frac{1}{2}L(r Z\rp{l}\phi))|^2 r^{-2}. \label{3.24.16.18}
\end{align}
Note that due to (\ref{3.24.1.18}) and the smallness of $|h|$,
\begin{equation}\label{3.24.15.18}
|L'(\Omega_{ij}(r^\f12 L(r Z\rp{l}\phi)))| \les r^{-\f12} |L(r\Omega_{ij} Z\rp{l}\phi)|+r^\f12 |L'\big(L(r\Omega_{ij}Z\rp{l}\phi)\big)|.
\end{equation}
By using (\ref{3.24.14.18}), we have
\begin{align*}
 |L(\Omega_{ij}L (rZ\rp{l}\phi))|&\les  |L\Omega_{ij}Z\rp{l}\phi|+|L(rL\Omega_{ij}Z\rp{l}\phi)|\les |L(r \p \Omega_{ij} Z\rp{l}\phi)|+|L\Omega_{ij} Z\rp{l}\phi|,
\end{align*}
and due to $[L, \Lb]=0$ and (\ref{3.24.14.18})
\begin{equation*}
|h| |\Lb  L(r \Omega_{ij}Z\rp{l}\phi)|\le |h|(|L \Omega_{ij} Z\rp{l}\phi|+|L( r\Lb\Omega_{ij} Z\rp{l}\phi)|)\les |h|(|L (r \p \Omega_{ij} Z\rp{l}\phi)|+|L \Omega_{ij}Z\rp{l}\phi|).
\end{equation*}
In view of the above two inequalities, the smallness of $|h|$,  (\ref{3.24.15.18}) and (\ref{7.3.1.18}) we have
\begin{align*}
\int_{\H_u^{\ub}} |L'\Omega\rp{\le 1} (r^\frac{1}{2}L(r Z\rp{l}\phi))|^2 r^{-2}&\les \W_1[\p \Omega Z\rp{l} \phi](\H_u^\ub) +E[\Omega Z\rp{l}\phi](\H_u^\ub)u_+^{-1}\\
&\les \dn u_+^{-\ga_0-1+2\zeta(Z^l)}.
\end{align*}
Thus we derive from (\ref{3.24.16.18}) and (\ref{6.24.4.18})   that
\begin{align*}
&\sup_{S_{u,\ub}}|r^\frac{3}{2}\c r^{-1}L(r Z\rp{l}\phi) |^4\\
&\les \int_{S_{u,-u}} |r^\f12 \Omega\rp{\le 1}L(r Z\rp{l}\phi)|^4 r^{-2}+\W_1[\Omega\rp{\le 2} Z\rp{l}\phi](\H_u^\ub)\c \dn u_+^{-\ga_0-1+2\zeta(Z^l)}\\
&\les (\E^2_{l+2,\ga_0}+\dn^2) u_+^{-2\ga_0+4\zeta(Z^l)},
\end{align*}
as desired in (\ref{7.20.3.18}).

Next we prove   (\ref{3.24.11.18}).  For any fixed point $(u,\ub,\omega)$, we integrate the estimate of $\Lb Z\rp{l}\phi$ in (\ref{3.24.18.18}) along the ingoing integral curve of $\p_{u'}$ along $\Hb_\ub^u$ from $t=0$.  For $l\le n-2$, we derive in view of (\ref{4.30.3.18}) that
\begin{align*}
\big|Z\rp{l}\phi(u, \ub,\omega)&-Z\rp{l}\phi(-\ub, \ub, \omega)\big|\les \int_{-\ub}^u \f12 (1-h)^{-1}|\Lb' Z\rp{l} \phi|(u', \ub, \omega) du'\\
&\les  \int_{-\ub}^u (|\Lb Z\rp{l} \phi|(u',\ub, \omega)+r^{-1} |h|(|L(r Z\rp{l}\phi)|+|Z\rp{l}\phi|)(u',\ub,\omega) ) du'.\\
\end{align*}
 The last term on the right can be treated by using Gronwall's inequality, $|h|\le M r^{-1}$ , (\ref{7.3.1.18}) and $-u\le \ub$.
The integration of the first and second term are bounded by $(\dn^\f12+\E^\f12_{l+2, \ga_0}) \ub^{-1} u_+^{\frac{-\ga_0+1}{2}+\zeta(Z^l)}$ by using $\ub\approx r$, the second estimate of  (\ref{3.24.18.18}) and (\ref{7.20.3.18}).  Thus by using $\ub\approx r$ again, we can derive
\begin{align*}
|rZ\rp{l} \phi(u,\ub,\omega)|&\les \ub| Z\rp{l}\phi(-\ub, \ub, \omega)|+ (\dn^\f12+\E^\f12_{l+2,\ga_0}) u_+^{\frac{-\ga_0+1}{2}+\zeta(Z^l)}.
\end{align*}
We then bound the first term on the right with the help of (\ref{6.24.9.18}) and $r\approx \ub$, which implies
\begin{equation*}
|r Z\rp{l}\phi|\les (\E^\f12_{l+2,\ga_0}+\dn^\f12) u_+^{-\frac{\ga_0-1}{2}+\zeta(Z^l)}.
\end{equation*}
The proof of  (\ref{3.24.11.18}) is then completed.
 The last estimate in  (\ref{3.24.18.18}) can be obtained by using (\ref{3.24.11.18}) and (\ref{7.20.3.18}).

Next, we prove (\ref{4.2.1.18}). It suffices to consider the estimate for $\Lb Z\rp{l}\phi$, other estimates follow from integrating the better estimates in  (\ref{3.24.18.18}).
We first apply the Sobolev inequality (\ref{6.24.11.18}) along $\Hb_\ub^u$ to $\Lb Z\rp{l} \phi$ with $\ga=\f12=\ga_2, \ga_0' =0$, which yields
\begin{align*}
\sup_{S_{u,\ub}}r^2 |\Lb Z\rp{l}\phi|^4&\les \int_{S_{-\ub, \ub}} r^2 |\Omega\rp{\le 1} \Lb Z\rp{l} \phi|^4 d\omega+\int_{\Hb_\ub^u} |\Omega\rp{\le 2}(\Lb Z\rp{l}\phi)|^2 r du'  d\omega\\
& \c\int_{\Hb_\ub^u} |\Lb' \Omega\rp{\le 1} (r^\f12 \Lb Z\rp{l} \phi)|^2 du' d\omega.
\end{align*}
We then derive
\begin{equation}\label{7.20.4.18}
\begin{split}
&\int_{-u_1}^ {\ub_1} \big(\sup_{-\ub \le u\le u_1} \sup_{S_{u,\ub}}r |\Lb Z\rp{l} \phi|^2 \big)d\ub \\
&\les\int_{-u_1}^ {\ub_1}  (\int_{S_{-\ub, \ub}} r^2 |\Omega\rp{\le 1} \Lb Z\rp{l}\phi|^4 d\omega )^\f12 d\ub+(\int_{\D_{u_1}^{\ub_1}} |\Omega\rp{\le 2}(\Lb Z\rp{l}\phi)|^2 r du d\omega d\ub)^\f12 \\
&\c(\int_{\D_{u_1}^{\ub_1}} |\Lb' \Omega\rp{\le 1} (r^\f12 \Lb Z\rp{l}\phi)|^2 du d\omega d\ub)^\f12.
\end{split}
\end{equation}
We apply (\ref{6.24.6.18}) to $\Lb\Omega\rp{\le 1}Z\rp{l}\phi$, which then bounds the first term on the righthand side of the inequality by ${u_1}_+^{-\ga_0-1+2\zeta(Z^l)}\E_{l+2,\ga_0}$. For the second term, due to (\ref{3.24.1.18}) we derive
\begin{align*}
\int_{\D_{u_1}^{\ub_1}} |\Omega\rp{\le 2}\Lb Z\rp{l} \phi|^2 r d\ub du d\omega&\les M^{-1} \int_{-\ub_1}^{u_1}\int_{\H_u^{\ub_1}} |\Lb \Omega \rp{\le 2} Z\rp{l}\phi|^2 r d\ub d\omega du  \\
&\les M^{-1} \int_{-\ub_1}^{u_1} E[\Omega\rp{\le 2}Z\rp{l} \phi](\H_u^{\ub_1}) du \\
&\les M^{-1} \dn { u_1}_+^{-\ga_0+1+2\zeta(Z^l)}.
\end{align*}
 By using (\ref{3.24.14.18}) and in view of (\ref{6.26.1.18}) with $\a=\f12,$ $f=Z\rp{l}\phi$ , we have
\begin{align*}
\int_{\D_{u_1}^{\ub_1}}& |\Lb' \Omega\rp{\le 1} (r^\f12 \Lb Z\rp{l}\phi)|^2 du d\ub d\omega \\
&\les \int_{\D_{u_1}^{\ub_1}} \{r^{-1}|\Lb \Omega\rp{\le 1}Z\rp{l}\phi|^2 + r|\Lb \p \Omega\rp{\le 1}Z\rp{l} \phi|^2 \} du d\ub d\omega\\
&\les \int_{-\ub_1}^{u_1} M^{-1} E[\p \Omega\rp{\le 1} Z\rp{l}\phi](\H_u^{\ub_1}) du+{u_1}_+^{-2} \sup_{-u_1\le \ub\le \ub_1} E[\Omega\rp{\le 1}Z\rp{l}\phi](\Hb_\ub^{u_1})\\
&\les ({u_1}_+^{-1}+M^{-1}) \dn {u_1}_+^{-\ga_0-1+2\zeta(Z^l)}.
\end{align*}
By substituting the above two estimates to (\ref{7.20.4.18}),   (\ref{4.2.1.18}) is proved.

Next, we   prove (\ref{6.17.1.18}). (\ref{6.17.2.18}) follows as an immediate consequence by integrating in $u$-variable.

By applying   $
\Lb'(r f^2)=(-1-h)f^2 +r \Lb'(f^2)
$  to $f=Z\rp{a}\phi$,
and in view of  $1+h>0$ due to the smallness of $h$,  we integrate the result in $\D_{u_1}^{\ub_1}$ with the area element $\f12 (1-h)^{-1} d\ub du d\omega$, which yields,
\begin{align}
&\int_{\H_{u_1}^{\ub_1}}r |Z\rp{a}\phi|^2 (u_1, \ub, \omega) d\ub d\omega \label{7.4.1.18}\\
&\les \int_{-u_1}^{\ub_1}  \int_{{\mathbb S}^2} r |Z\rp{a}\phi|^2(-\ub, \ub, \omega) d\ub d\omega+\int_{-\ub_1}^{u_1} \|r^{-\f12} \Lb' Z\rp{a}\phi\|_{L^2(\H_u^{\ub_1}) } \|r^{-\f12 }Z\rp{a}\phi\|_{L^2(\H_u^{\ub_1})} du\nn,
\end{align}
where we dropped the integral of $(1+h) |Z\rp{a}\phi|^2$ due to the positivity. Note that
\begin{equation}\label{7.20.5.18}
 \|r^{-\f12} \Lb' Z\rp{a}\phi\|_{L^2(\H_u^{\ub_1}) }\les M^{-\f12} E[Z\rp{a}\phi]^\f12(\H_u^{\ub_1})\les u_+^{-\frac{\ga_0}{2}+\zeta(Z^a)}M^{-\f12}\dn^\f12,
\qquad a\le n
\end{equation}
and the first term on the right of (\ref{7.4.1.18}) can be bounded by ${u_1}_+^{-\ga_0+1+2\zeta(Z^a)}\E_{a,\ga_0}$ by using (\ref{6.25.1.18}).
 By  multiplying both sides by ${u_1}_+^p$ with $p\ge 1$, followed with applying Gronwall's inequality,
\begin{align*}
{u_1}_+^\frac{p}{2}&(\int_{\H_{u_1}^{\ub_1}}r |Z\rp{a}\phi|^2 (u_1, \ub, \omega) d\omega d\ub)^\f12\\
&\les {u_1}_+^{\frac{p-\ga_0+1}{2}+\zeta(Z^a)}\E^\f12_{n,\ga_0}+{u_1}_+^{p}\int_{-\ub_1}^{u_1} \|r^{-\f12} \Lb'(Z\rp{a}\phi)\|_{L^2(\H_u^{\ub_1})} u_+^{-\frac{p}{2}} du.
\end{align*}
 We can obtain (\ref{6.17.1.18}) for $a\le n$ by using (\ref{7.20.5.18}).

\end{proof}

With the help of the assumption of ($\BA{n}$) with $n=2, 3$,   we can derive the following  estimates of mixed $L^p$ norms.
\begin{proposition}\label{6.17.5.18}
Let $-\ub_*\le -\ub_1\le u_1\le u_0$,  $b\le n-1$,  $a\le n$, $\ga>\f12$ and $\a>0$. There hold the following estimates
\begin{align}
&\| r^{-\f12} \p Z\rp{a} \phi\|_{L^2(\D_{u_1}^{\ub_1})}\les  {u_1}_+^{-\frac{\ga_0}{2}+\zeta(Z^a)+\f12}\dn^\f12 M^{-\f12},\label{7.5.4.18}\\
&\|r^{-\ga} \p Z\rp{a}\phi\|_{L^2(\D_{u_1}^{\ub_1})}\les {u_1}_+^{-\ga+\f12-\f12\ga_0+\zeta(Z^a)}\dn^\f12,\label{7.19.2.18}\\
&\|r^\f12 Z\rp{a}\phi\|^2_{L_\ub^\infty L^2_u L_\omega^2(\D_{u_1}^{\ub_1})}+\a\|(\frac{{u_1}_+}{r})^\a r^{-1}Z\rp{a}\phi\|^2_{L^2(\D_{u_1}^{\ub_1})}\nn\\
&\qquad\qquad\qquad\qquad\qquad\les {u_1}_+^{-\ga_0+1+2\zeta(Z^a)}(\dn+\E_{a,\ga_0}),\label{7.21.2.18}\\
&\a\|(\frac{{u_1}_+}{r})^\a Z\rp{a} \phi\|^2_{L_u^2 L_\ub^2 L_\omega^4(\D_{u_1}^{\ub_1})}\les {u_1}_+^{- \ga_0+1+2\zeta(Z^a)} (\dn+\E_{a,\ga_0}), \label{6.3.11.18}\\
&\|r^\f12 \p Z\rp{b} \phi\|_{L^4(S_{u_1,\ub_1})}\les {u_1}_+^{-\f12\ga_0-\f12+\zeta(Z^b)}(\dn^\f12+\E^\f12_{b+1,\ga_0}),\label{6.17.3.18}\\
&\| r \p Z\rp{b}\phi\|_{L_\ub^\infty L_u^2 L_\omega^4(\D_{u_1}^{\ub_1})}\les {u_1}_+^{-\frac{\ga_0}{2}+\zeta(Z^b)}(\dn^\f12+\E^\f12_{b+1,\ga_0}),\label{7.5.2.18}\\
&\|r \p Z\rp{b}\phi\|_{L_u^2 L_\ub^\infty L_\omega^4(\D_{u_1}^{\ub_1})}\les  {u_1}_+^{-\frac{\ga_0}{2} +\zeta(Z^b)}(\dn^\f12+\E^\f12_{b+1,\ga_0}),\label{6.17.11.18}\displaybreak[0]\\
&\| r^\f12 \p Z\rp{b} \phi\|_{L_\ub^2 L_u^\infty L_\omega^4(\D_{u_1}^{\ub_1})}\les {u_1}_+^{-\f12\ga_0+\zeta(Z^b)} ( \dn^\f12 M^{-\f12}+\E^\f12_{b+1, \ga_0} {u_1}_+^{-\f12}) \label{6.14.3.18},\\
&\| r^\f12 \p Z\rp{b}\phi\|_{L_\ub^2 L_u^2 L_\omega^4(\D_{u_1}^{\ub_1})}\les {u_1}_+^{-\f12\ga_0+\f12+\zeta(Z^b)} \dn^\f12  M^{-\f12}.\label{7.5.3.18}
\end{align}
If the constant $p>-\frac{\ga_0}{2}+\frac{3}{2}$, there hold
\begin{align}
&\|u_+^{-p} r^\f12 Z\rp{b}\phi \|_{L^2_u L_\ub^2 L_\omega^4(\D_{u_1}^{\ub_1})}\les {u_1}_+^{\zeta(Z^{b})-\f12\ga_0+\frac{3}{2}-p}(\dn^{\f12}M^{-\f12}+\E^\f12_{b+1,\ga_0}), \label{6.5.7.18}\\
&\|r^\f12 Z\rp{a} \phi\|_{ L_u^2 L_\omega^4(\Hb^{u_1}_{\ub_1})}\les {u_1}_+^{\zeta(Z^a)-\f12 \ga_0+\f12}(\dn^\f12+\E^\f12_{a,\ga_0}). \label{6.14.10.18}
\end{align}
\end{proposition}
\begin{proof}
We first derive  by using $(\BA{n})$ that
\begin{align*}
\| r^{-\f12} \p Z\rp{a} \phi\|_{L^2(\D_{u_1}^{\ub_1})}&\les M^{-\f12} (\int_{-\ub_1}^{u_1} E[Z^a \phi](\H_u^{\ub_1}) du)^\f12\les M^{-\f12} {u_1}_+^{-\frac{\ga_0}{2}+\zeta(Z^a)+\f12}\dn^\f12, \quad a\le n.
\end{align*}

By using $(\BA{n})$, (\ref{7.19.2.18}) is a direct consequence of $|\p Z\rp{a}\phi|\les |\ud \p Z\rp{a}\phi|+|\bar \p Z\rp{a}\phi|$;
(\ref{7.21.2.18}) is a consequence of (\ref{3.23.2.18}) and (\ref{6.25.1.18}).

Note that by using (\ref{6.14.7.18}), (\ref{6.25.1.18}) and ($\BA{n}$), we can derive
\begin{align*}
\a\|(\frac{{u_1}_+}{r})^\a Z\rp{a} \phi\|^2_{L_u^2 L_\ub^2 L_\omega^4(\D_{u_1}^{\ub_1})}&\les \W_1[Z\rp{a}\phi](\D_{u_1}^{\ub_1})+\|r^{-\f12} Z\rp{a}\phi\|^2_{L^2(\Sigma_0^{u_1,\ub_1})}\\
&+M\int_{-\ub_1}^{u_1} {u'}_+^{-1}E[Z\rp{a}\phi](\H_{u'}^{\ub_1}) du'\\
&\les {u_1}_+^{- \ga_0+1+2\zeta(Z^a)} (\dn+\E_{a,\ga_0}),
\end{align*}
as desired in (\ref{6.3.11.18}).

Next we consider (\ref{6.17.3.18}).
For the estimate of $L Z\rp{b}\phi$, we apply (\ref{6.28.1.18}) to  $\psi= Z\rp{b}\phi$, which yields  for
$\I=\{(u,\ub):-\ub_1\le -\ub\le u\le u_1\}$  that
\begin{align*}
\int_{S_{u,\ub}}|r LZ\rp{b}\phi|^4 d\omega &\les \int_{S_{u, -u}} |rL Z\rp{b}\phi|^4 d\omega+(E[\p Z\rp{b} \phi](\H_{u}^{\ub})+{u}_+^{-2} E[Z\rp{b}\phi](\H_{u}^{\ub}))\nn\\
&\qquad\c E[\Omega\rp{\le 1}Z\rp{b}\phi](\H_{u}^{\ub}).
\end{align*}
The first term on the right can be bounded by applying (\ref{6.24.4.18}) to $\p Z\rp{b}\phi$, which is bounded by $u_+^{-2-2\ga_0+4\zeta(Z^b)}\E^2_{b+1, \ga_0}$. We then use  ($\BA{n}$) to obtain
\begin{equation*}
\int_{S_{u,\ub}}|r LZ\rp{b}\phi|^4 d\omega\les u_+^{-2-2\ga_0+4\zeta(Z^b)}(\E^2_{b+1, \ga_0}+\dn^2).
\end{equation*}
By repeating the same procedure for $\ud\p\psi$ in view of (\ref{6.27.1.18}), we can get the same estimate with $L$ replaced by $\ud \p$.
 This implies (\ref{6.17.3.18}).

 Integrating (\ref{6.17.3.18}) along $\Hb^{u_1}_\ub$ for any $-u_1\le \ub\le \ub_1$ implies  (\ref{7.5.2.18}) directly.

  We note also that (\ref{6.17.3.18}) is independent of $\ub$. For $(u,\ub)\in \I$, we can take supremum of (\ref{6.17.3.18}) for $-u \le \ub\le \ub_1$, followed with integrating $u$ from $-\ub_1$ to $u_1$. Thus (\ref{6.17.11.18})  is proved.

Next, we apply (\ref{5.22.5.18}) to $\psi=Z\rp{b}\phi$ for $b\le n-1$ and also by using ($\BA{n}$),
\begin{equation*}
\| r^\f12 \p Z\rp{b} \phi\|^2_{L_\ub^2 L_u^\infty L_\omega^4(\D_{u_1}^{\ub_1})}\les\int_{-u_1}^{\ub_1} (\int_{S_{-\ub, \ub}} r^2 |\p Z\rp{b}\phi|^4 d\omega)^\f12 d\ub +M^{-1}\dn {u_1}_+^{2\zeta(Z^b)-\ga_0}.
\end{equation*}
We then apply (\ref{6.24.10.18}) to $\p Z\rp{b}\phi$, which bounds the first term on the right by $\E_{b+1, \ga_0} {u_1}_+^{-\ga_0-1+2\zeta(Z^b)}$.  The result of (\ref{6.14.3.18}) follows by a direct substitution.

(\ref{7.5.3.18}) follows by applying (\ref{6.14.1.18}) to $\psi=Z\rp{b}\phi$ and also using $(\BA{n})$.

We can obtain (\ref{6.5.7.18}) by  applying  the sobolev embedding on unit sphere
\begin{equation*}
\|r^\f12 u_+^{-p} Z\rp{b} \phi\|_{L_u^2 L_\ub^2 L_\omega^4(\D_{u_1}^{\ub_1})} \les \| r^{-\f12} u_+^{-p} \Omega\rp{\le 1} Z\rp{b} \phi\|_{L^2(\D_{u_1}^{\ub_1})},
\end{equation*}
followed by applying (\ref{6.17.2.18}) to $\Omega\rp{\le 1} Z\rp{b}\phi$ with $p>-\frac{\ga_0}{2}+\frac{3}{2}$.

We deduce by applying (\ref{6.3.9.18}) to $F=Z\rp{a}\phi$ followed with using (\ref{6.25.1.18}) and ($\BA{n}$) that
\begin{align*}
\|r^\f12 Z\rp{a}\phi\|_{L_u^2 L_\omega^4(\Hb_{\ub_1} ^{u_1})}&\les \W_1^\f12[Z\rp{a}\phi](\Hb_{\ub_1}^{u_1})+\W_1^\f12[Z\rp{a}\phi](\D_{u_1}^{\ub_1})+\|r^{-\f12}Z\rp{a}\phi\|_{L^2(\Sigma_0^{u_1,\ub_1})}\\
&+M^\f12 (\int_{-\ub_1}^ {u_1} u_+^{-1} E[Z\rp{a}\phi](\H_u^{\ub_1}) du)^\f12\les {u_1}_+^{\zeta(Z^a)-\f12\ga_0+\f12}(\dn^\f12+\E^\f12_{a,\ga_0}),
\end{align*}
as desired in (\ref{6.14.10.18}).
\end{proof}
\section{Semilinear wave equations}\label{semi}

In this section we consider the equation of (\ref{3.18.1.18}) in ${\mathbb R}^{3+1}$, i.e.
\begin{equation*}
\Box_\bm \phi=\N^{\a\b}(\phi) \p_\a\phi\p_\b \phi+q(x) \phi
\end{equation*}
where $0\le q(x)\le 1$ satisfies (\ref{potential}) with $n=2$. We will prove Theorem \ref{thm2} and Theorem \ref{thm1}.

In Section \ref{prel} and \ref{decay}, for functions with (\ref{3.21.4.18}) bounded for $k=2$ or $3$, under the assumption of $(\BA{k})$, we have obtained decay properties and a set of  Sobolev inequalities.  In this section, by a bootstrap argument, we will prove ($\BA{2}$) with $\dn$ comparable with $\E_{2,\ga_0,R}$, provided that the latter is sufficiently small. The analysis in Sections \ref{prel} and \ref{decay} will play a crucial role for achieving the boundedness of the sets of  energies.

We will give the fundamental standard and weighted energy inequalities for the  $u$ and $\ub$ foliations in Proposition \ref{5.16.5.18} and Proposition \ref{6.12.2.18}. The goal is  to justify the energy norms given in (\ref{3.19.1}) can be  achieved along the foliations $\H_u$, $\H_\ub$ and $\D_u^\ub$  provided that  $(\Box_\bm -q)f$  can be bounded  as desired.

Throughout this section, $M>0$ is a fixed small number, with the upper bound determined during the proofs of Theorem \ref{thm2} and \ref{thm1}. Let
 $$h=M/r,\qquad  u_0=u_0(M, R),$$
 where $R\ge 2$, whose lower bound  will be finally determined at the end of the proof of Theorem \ref{thm2}.

For  the wave equation $\Box_\bm \varphi-q\varphi=F$, we define the energy momentum tensor,
\begin{equation}\label{5.7.5.18}
Q_{\a\b}[\varphi]= \p_\a  \varphi \p_\b \varphi-\frac{1}{2} \bm_{\a\b} (\p^\a \varphi \p_\a \varphi+q \varphi^2).
\end{equation}
Recall  $\N, \Nb$ from (\ref{7.11.3.18}).  It is straightforward to obtain  the energy densities on $\H_u$ and $\Hb_\ub$
  \begin{equation}\label{7.4.2.18}
  \begin{split}
&Q[\varphi](\p_t, \N)= \f12 (1+h)^{-1} \{(L\varphi)^2+h (\Lb \varphi)^2+(1+h)( |\sn \varphi|^2+q \varphi^2)\},\\
&Q[\varphi](\p_t, \Nb)= \f12 (1+h)^{-1} \{(\Lb\varphi)^2+h (L \varphi)^2+(1+h)( |\sn \varphi|^2+q \varphi^2)\}.
\end{split}
\end{equation}

Next, we give the fundamental energy inequality in $\{u\le u_0\}$.
\begin{lemma}\label{3.19.2.18}
Consider the equation $\Box_\bm \varphi-q(x) \varphi =F$. There holds for a smooth vectorfield $X$ that
\begin{equation}\label{7.20.6.18}
\p^\a (Q_{\a\b} [\varphi]X^\b)= Q_{\a\b} {}\rp{X}\pi^{\a\b}+(\Box_\bm \varphi-q\varphi) X\varphi-\f12 Xq\c \varphi^2,
\end{equation}
where ${}\rp{X} \pi_{\a\b}=\f12(\l \p_\a X,\p_\b\r +\l \p_\b X, \p_\a\r)$.

 There also holds the following energy estimate,
 \begin{equation}\label{7.20.7.18}
 E[\varphi](\H_u^\ub)+E[\varphi](\Hb_\ub^u) \les |\int_{\D_u^{\ub}} F\c \p_t \varphi |+E[\varphi](\Sigma^{u,\ub}_0)
 \end{equation}
 for all $-\ub\le u\le u_0$.
\end{lemma}
\begin{proof}
For  $Q_{\a\b}[\varphi]$ in (\ref{5.7.5.18}), it is straightforward to derive
\begin{equation*}
\p^\a Q_{\a\b}[\varphi]= (\Box_\bm\varphi-q \varphi) \p_\b \varphi -\f12 \p_\b q \c \varphi^2.
\end{equation*}
(\ref{7.20.6.18}) follows as a consequence.

Applying (\ref{7.20.6.18}) with $X=\p_t$ gives  $\p^\a( Q_{\a\b}\p_t^\b)=(\Box_\bm\varphi-q \varphi) \p_t \varphi$.
We then integrate the identity in $\D_u^\ub$ to  obtain
\begin{equation*}
 \int_{\H_u^\ub}Q(\N, \p_t) d\mu_\H +\int_{\Hb_\ub^u} Q(\Nb, \p_t) d\mu_{\Hb}+\int_{\D_u^\ub} F\c \p_t \varphi dx dt=\int_{\Sigma_0^{u, \ub}}Q(\p_t, \p_t) dx.
\end{equation*}
(\ref{7.20.7.18}) is proved in view of (\ref{7.4.2.18})
\end{proof}
As a consequence, we derive the following energy estimate.
\begin{proposition}\label{5.16.5.18}
Let $\I=\{(u_1,\ub_1): -\ub_2\le -\ub_1 \le u_1\le  u_2\le u_0\}$, where $u_2$ and $ \ub_2$ are fixed. There holds the following energy inequality for  $(u,\ub)\in \I$,
\begin{align*}
 {u}_+^{-2p+\ga_0}&(E[\varphi](\H_{u}^{\ub})+ E[\varphi](\Hb_{\ub}^{u})) \\
&\quad\les \sup_{(u_1,\ub_1)\in \I}\{{u_1}_+^{-2p+\ga_0} E[\varphi](\Sigma_0^{u_1,\ub_1})+{u_1}_+^{-2p+\ga_0}(\|rF^\flat\|_{L_\ub^1 L_u^2 L_\omega^2(\D_{u_1}^{\ub_1})}^2\\
& \quad +  \| r F^\flat \|_{L_u^1 L_\ub^2 L_\omega^2(\D_{u_1}^{\ub_1})}^2)+{u_1}_+^{1-2p+\ga_0} M^{-1} \|r^\f12 F^\sharp \|^2_{L^2(\D_{u_1}^{\ub_1})}\},
\end{align*}
 where $p\le 0$ is any constant \begin{footnote}{The inequality holds uniformly for any $p\le 0$.}\end{footnote}, $F=F^\flat+F^\sharp$, and the two smooth functions $F^\flat$ and $F^\sharp$ are  in the corresponding normed spaces.
\end{proposition}
\begin{proof}
We apply  Lemma \ref{3.19.2.18}  with  $(u,\ub)\in \I$. This implies
\begin{align}
{u}_+^{-2p+\ga_0}&\big(E[\varphi](\H_{u}^{\ub})+E[\varphi](\Hb_{\ub}^{u})\big)\nn \\
&\les\sup_{(u_1,\ub_1)\in \I }\{ {u_1}_+^{-2p+\ga_0}|\int_{\D_{u_1}^{\ub_1}} F\c \p_t \varphi dx dt|+{u_1}_+^{-2p+\ga_0}E[\varphi](\Sigma_0^{u_1,\ub_1})\}.\label{5.16.2.18}
\end{align}
To control the first term on the right, we first consider the $\F^\flat$ term in  the decomposition $F= F^\sharp+F^\flat$.
\begin{align*}
|\int_{\D_{u_1}^{\ub_1}} F^\flat\c \p_t \varphi dx dt|\les \int_{\D_{u_1}^{\ub_1}} |F^\flat \c L \varphi| dx dt+\int_{\D_{u_1}^{\ub_1}} |F^\flat \c \Lb \varphi| dx dt=I+II.
\end{align*}
We estimate $I$ and $II$ in the following way,
\begin{align*}
I&\les \| {u}_+^{p-\f12\ga_0} rF^\flat \|_{L_u^1 L_\ub^2 L_\omega^2(\D_{u_1}^{\ub_1})}\|{u}_+^{-p+\f12 \ga_0} r L \varphi\|_{L_u^\infty L_\ub^2 L_\omega^2(\D_{u_1}^{\ub_1})}  \\
&\les  {u_1}_+^{p-\f12\ga_0} \|rF^\flat \|_{L_u^1 L_\ub^2 L_\omega^2(\D_{u_1}^{\ub_1})}\sup_{-\ub_1\le u\le u_1}  {u}_+^{-p+\f12 \ga_0} E[\varphi]^\f12(\H_{u}^{\ub_1}),\\
II&\les \| rF^\flat \|_{L_\ub^1 L_u^2 L_\omega^2(\D_{u_1}^{\ub_1})}\| r \Lb  \varphi\|_{L_\ub^\infty L_u^2 L_\omega^2(\D_{u_1}^{\ub_1})}\\
&\les  {u_1}_+^{p-\f12\ga_0}\| rF^\flat \|_{L_\ub^1 L_u^2 L_\omega^2(\D_{u_1}^{\ub_1})}  \sup_{-u_1\le \ub\le \ub_1}{u_1}_+^{-p+\f12 \ga_0} E[\varphi]^\f12(\Hb_{\ub}^{u_1}).
\end{align*}
Similarly,  we can derive
\begin{align*}
\int_{\D_{u_1}^{\ub_1}} |F^\sharp \c \p_t \varphi| dx dt
&\le  M^{-\f12}\|r^\f12 F^\sharp \|_{L^2(\D_{u_1}^{\ub_1})}(\int_{-\ub_1}^{u_1} E  [\varphi](\H_{u}^{\ub_1}) du)^\f12\\
  &\les M^{-\f12} \|r^\f12 F^\sharp \|_{L^2(\D_{u_1}^{\ub_1})} {u_1}_+^{-\frac{\ga_0}{2}+\f12+p}\sup_{-\ub_1\le u \le u_1}  {u}_+^{\frac{\ga_0}{2}-p} E^\f12 [\varphi](\H_{u}^{\ub_1}).
\end{align*}
By multiplying the weight ${u_1}_+^{-2p+\ga_0}$ to the three inequalities, followed by using Cauchy-Schwartz inequality, we can derive
\begin{align}
{u_1}_+^{-2p+\ga_0}&|\int_{\D_{u_1}^{\ub_1}} F \c \p_t \varphi dx dt|\nn\\
&\les \ep_1^{-1} \big({u_1}_+^{-2p+\ga_0}\|rF^\flat \|_{L_\ub^1 L_u^2 L_\omega^2(\D_{u_1}^{\ub_1})}^2 + {u_1}_+^{-2p+\ga_0} \|r F^\flat \|_{L_u^1 L_\ub^2 L_\omega^2(\D_{u_1}^{\ub_1})}^2\nn\\
&\quad \quad +\|r^\f12 F^\sharp \|^2_{L^2(\D_{u_1}^{\ub_1})}M^{-1} {u_1}_+^{\ga_0-2p+1}\big)\nn\\
&+\ep_1\big( \sup_{-\ub_1\le u\le u_1}  {u}_+^{-2p+ \ga_0} E[\varphi](\H_{u}^{\ub_1})+ \sup_{-u_1\le \ub\le \ub_1}{u_1}_+^{-2p+\ga_0} E[\varphi](\Hb_{\ub}^{u_1})\big).\label{5.16.3.18}
\end{align}
By substituting the above inequality to (\ref{5.16.2.18}), followed with taking supremum for $(u,\ub)\in \I $, the last line of (\ref{5.16.3.18}) can be absorbed. Thus Proposition \ref{5.16.5.18} is proved.
\end{proof}
Next, we give the weighted energy estimate.
\begin{lemma}\label{5.7.7.18}
For any $\ub \le u\le u_0$,  there holds the following inequality,
\begin{align*}
& \int_{\H_u^\ub} r (L(r\varphi))^2+h r^3( |\sn \varphi|^2+q\varphi^2) d\ub' d\omega+\int_{\Hb_\ub^u} r^{3}(h (L\varphi)^2+ |\sn \varphi|^2+q\varphi^2) du' d\omega\\
&+\f12\int_{\D_u^\ub}\{(L(r\varphi))^2+r^2 |\sn \varphi|^2\} du' d\ub' d\omega \\
&\les \W_1[\varphi](\Sigma_0^{u,\ub})+\int_{\D_u^\ub}| F\c L(r\varphi)| dx dt+\int_{-\ub}^u  E[\varphi] (\H_{u'}^\ub) du'+E[\varphi](\Hb_\ub^u)\nn\\
&+\|r^{-1/2}\varphi\|^2_{L^2(\Sigma_0^{u, \ub})}+\int_{S_{-\ub, \ub}} r\varphi^2 d\omega,\nn
\end{align*}
where $q(x)\ge 0$ verifies (\ref{potential})  with $n=2$.
\end{lemma}
\begin{proof}
We define
\begin{equation*}
{}\rp{X}\sP_\a= Q_{\a\b}X^\b +\f12 \p_\a(\varphi^2)+Y_\a,
\end{equation*}
where
\begin{equation*}
X=rL, \quad Y= \f12 r^{-1} \varphi^2 L.
\end{equation*}
Similar to the calculation in \cite{mMaxwell}, we have
\begin{align}\label{5.7.6.18}
\p^\a {}\rp{X} \sP_\a= (\Box_\bm \varphi-q\varphi) (X\varphi+\varphi)+\f12 \big(r^{-2} (L(r\varphi))^2 +|\sn \varphi|^2\big) - \f12 (Xq +q)\varphi^2.
\end{align}

Indeed, by using (\ref{7.20.6.18}), we derive
\begin{equation}\label{7.20.8.18}
\p^\a {}\rp{X} \sP_\a = (\Box_\bm \varphi-q\varphi) X\varphi-\f12 Xq\c \varphi^2+Q_{\a\b}{}\rp{X}\pi^{\a\b}+\f12 \p^\a \p_\a (\varphi^2)+\p^\a Y_\a.
\end{equation}
We can check the nonvanishing components of ${}\rp{X} \pi$ for $X=r L$ are
\begin{equation*}
{}\rp{X}\pi_{\Lb \Lb}=2, \,\quad {}\rp{X}\pi_{L\Lb}=-1,\quad {}\rp{X}\pi_{e_Ae_B}=\delta_{AB}, \, \, A,B=1,2
\end{equation*}
and
\begin{align*}
&Q_{L L}=(L\varphi)^2, \qquad Q_{\Lb \Lb}=(\Lb \varphi)^2,  \qquad Q_{L\Lb}=|\sn \varphi|^2+q\varphi^2\\
&Q_{AB}=\sn_A \varphi \sn_B \varphi-\f12 \bm_{AB}(-L \varphi \Lb \varphi+|\sn \varphi|^2+q\varphi^2 ).
\end{align*}
By combining the calculations of $Q_{\a\b}$ and ${}\rp{X}\pi^{\a\b}$, we derive
\begin{align*}
&Q_{LL}{}\rp{X} \pi^{LL}=\f12 (L \varphi)^2, \quad \quad Q_{L \Lb}{}\rp{X} \pi^{L \Lb}=-\frac{1}{4} (|\sn \varphi|^2+q\varphi^2)\\
&Q_{AB} {}\rp{X}\pi^{AB}=\Lb \varphi L \varphi-q \varphi^2.
\end{align*}
Thus
\begin{equation*}
Q_{\a\b}{}\rp{X}\pi^{\a\b}=\f12\big((L \varphi)^2 -|\sn\varphi|^2)+\Lb \varphi L \varphi-\frac{3}{2}q\varphi^2.
\end{equation*}
It is straightforward to have
\begin{align*}
\f12 \p^\a\p_\a (\varphi^2)&=  \p^\a\varphi \p_\a \varphi+\Box_\bm \varphi \c \varphi=-L \varphi \Lb \varphi+|\sn \varphi|^2+\Box_\bm \varphi \c \varphi.
\end{align*}
Note that $\p^\a L_\a=2/r$. We have
\begin{align*}
\p^\a Y_\a &= \f12 \p^\a(r^{-1} \varphi^2L_\a )=\f12 \{\varphi^2 L(r^{-1})+r^{-1}L(\varphi^2)+r^{-1} \varphi^2\p^\a L_\a \}\\
&=\f12 (r^{-1} L (\varphi^2)+r^{-2} \varphi^2).
\end{align*}
Combining the above calculations with (\ref{7.20.8.18}) implies (\ref{5.7.6.18}).

On the other hand, by divergence theorem, we have
\begin{align}
\int_{\H_u^\ub}& \N^\a {}\rp{X}\sP_\a  d\mu_\H  +\int_{\Hb_\ub^u} \Nb^\a{}\rp{X} \sP_\a  d\mu_{\Hb}= \int_{\Sigma_0^{u, \ub}} {}\rp{X}\sP_\a (\p_t)^\a dx-\int_{\D_u^\ub} \p^\a{}\rp{X} \sP_\a dx dt\label{5.7.3.18}
\end{align}
where the area elements $d\mu_\H$ and $d\mu_{\Hb}$ can be found in (\ref{7.4.3.18}). A direct substitution implies
\begin{equation}\label{5.7.4.18}
\begin{split}
&r^2 (1+h) \N^\a {}\rp{X}\sP_\a= r \{(L (r\varphi))^2+r^2 h (|\sn \varphi|^2+q\varphi^2)-r^{-1}\f12 L' (r^2 \varphi^2)\},\\
&r^2 (1+h)\Nb^\a{}\rp{X} \sP_\a=r^3\big(h (L\varphi)^2+|\sn \varphi|^2+q\varphi^2\big)+\f12 \Lb' (r^2 \varphi^2)+rh\big(r L(\varphi^2)+\varphi^2\big),\\
&r^2 \p_t^\a {}\rp{X}\sP_\a=\f12 r^3 (r^{-2}|L(r\varphi)|^2+|\sn \varphi|^2  +q\varphi^2 )-\f12\p_r(r^2 \varphi^2).
\end{split}
\end{equation}

Note that  for any $-\ub_1\le u_1\le u_0$, there holds
\begin{align}
\int_{\H_{u_1}^{\ub_1}}\frac{d}{d\ub} (r^2 \varphi^2) d\omega d\ub-\int_{\Hb_{\ub_1}^{u_1}} \frac{d}{du} (r^2 \varphi^2) d\omega du-\int_{\Sigma_0^{u_1, \ub_1}} \p_r (r^2 \varphi^2) d\omega dr =0.\label{6.21.1.18}
\end{align}
By adding this identity to (\ref{5.7.3.18}), in view of (\ref{5.7.4.18}) and (\ref{4.30.3.18}), we can obtain
\begin{equation*}
\begin{split}
& \int_{\H_u^\ub} \big(r (L(r\varphi))^2+h r^3 (|\sn \varphi|^2+q\varphi^2)\big) (1+h)^{-1}\bb r^2  d\ub' d\omega+\int_{\Hb_\ub^u} r^{3}(h|L \varphi|^2\\
&+|\sn \varphi|^2+q\varphi^2) (1+h)^{-1}\bb du' d\omega+\f12\int_{\D_u^\ub }\{(L(r\varphi))^2+r^2 |\sn \varphi|^2\} \bb du' d\ub' d\omega \\
&\les\W_1[\varphi](\Sigma_0^{u, \ub})+\int_{\D_u^\ub} \{|Xq| \varphi^2+ q\varphi^2 \}+\int_{\D_u^\ub}|  F\c L(r\varphi) |\\
&+\int_{\H^u_\ub} r|h| \big|\varphi^2+rL (\varphi^2)\big| du'd\omega.
\end{split}
\end{equation*}
The last term can be treated by using (\ref{4.29.6.18}) and  Cauchy-Schwartz inequality,
\begin{align*}
\int_{\Hb_\ub^u} |h| (r|L(\varphi^2)|+\varphi^2)  du' d\omega &\les \int_{S_{-\ub, \ub}} r\varphi^2 d\omega +E[\varphi](\Hb_\ub^u)+ \int_{\Hb_\ub^u}|h|  |r L\varphi|^2  du'd\omega\\
&\les
\int_{S_{-\ub, \ub}} r\varphi^2 d\omega +E[\varphi](\Hb_\ub^u).
\end{align*}
Note that, due to (\ref{potential}), $|Xq|\les (1+r)^{-2-\eta}$.  We then have by using (\ref{3.23.2.18}) that
\begin{equation}\label{6.23.3.18}
\int_{\D_u^\ub}  |Xq|\varphi^2\les \int_{-u}^\ub (1+\ub')^{-1-\eta} \W_1[\varphi](\D^ {\ub'}_{u}) d\ub'+ \int_{\Sigma_0^{u, \ub}} r^{-1}\varphi^2+M\int_{-\ub}^{u} {u'}_+^{-1}E[\varphi](\H_{u'}^\ub) du'.
\end{equation}
The first term can be absorbed by using Gronwall's inequality, other terms can be derived by direct  integration. Thus, Lemma \ref{5.7.7.18} is proved.
\end{proof}
We then can  derive the following result.
\begin{proposition}\label{6.12.2.18}
 Let $p\le 0$ be any fixed number.  With  $\I=\{(u,\ub): -\ub_1\le u_1\le u_0\}$,  the following energy inequality holds
\begin{align*}
& \W_1[\varphi](\H^{\ub_1}_{u_1})+\W_1[\varphi](\Hb_{\ub_1}^{u_1})+\W_1[\varphi](\D_{u_1}^{\ub_1})\\
&\les\W_1[\varphi](\Sigma_0^{u_1,\ub_1})+\|r^\frac{3}{2}F\|_{L_u^1 L_\ub^2 L_\omega^2(\D_{u_1}^{\ub_1})}^2+ {u_1}_+^{-\ga_0+2p+1}\sup_{-\ub_1\le u\le u_1} u_+^{-2p+\ga_0}E[\varphi](\H_u^{\ub_1})\nn\\
&\qquad+ E[\varphi](\Hb_{\ub_1}^{u_1})+\|r^{-1/2}\varphi\|^2_{L^2(\Sigma_0^{u_1, \ub_1})}+\int_{S_{-\ub_1, \ub_1}} r\varphi^2 d\omega\nn.
\end{align*}
\end{proposition}
\subsection{Preliminaries}
The proof is based on a bootstrap argument, with the assumption of $(\BA2)$ and $\dn=C_1 \E_{2,\ga_0}$ with $C_1>1$ to be determined.
We recast the assumption as follows.

Let $\ub_*>-u_0$ be any fixed large number.
For $0\le n\le 2$, we suppose
\begin{align}
&E[Z\rp{n} \phi](\H_u^\ub)+E[Z\rp{n} \phi](\Hb_\ub^u)\le 2\dn u_+^{-\ga_0+2\zeta(Z^n)},\label{3.21.5.18}\\
&\W_1[Z\rp{n} \phi](\H_u^\ub)+\W_1[Z\rp{n} \phi](\Hb_\ub^u)+\W_1[Z\rp{n}\phi](\D_u^\ub)\le 2\dn u_+^{-\ga_0+1+2\zeta(Z^n)}\label{3.21.6.18}
\end{align}
hold for all $-\ub_*\le -\ub \le u\le u_0$.

The local well-posedness result,  in $\{-\ub(\ep)\le  -\ub < u\le  u_0\}$ with $\ub(\ep)$ finite,   can follow by running a standard iteration argument (see \cite{sogge}), or by using the standard local existence result upto the characteristic boundary, i.e. $\{r\ge t+R\}$. Thus the above assumptions hold for some $\ub_* >-u_0$.  Our task is to show that the estimates in the assumption hold for any $\ub_*>-u_0$, with the bound improved to be $<2\dn$. \begin{footnote}{The same argument is employed for setting up the bootstrap assumptions in Section \ref{quasi} and Section \ref{Eins}, which will not be repeated in later sections.}\end{footnote}

As a direct consequence of the bootstrap assumption, we have
\begin{lemma}
For $Z=\Omega$ or $\p$, there holds
\begin{equation}\label{5.18.1.18}
|Z\phi|\les u_+^{-\f12-\frac{\ga_0}{2}}\dn^\f12 r^{\zeta(Z)}.
\end{equation}
\end{lemma}
\begin{proof}
It  follows by  using (\ref{3.24.18.18}) with $l=0$ that
\begin{equation*}
r |\sn \phi|\les r^{-\f12} u_+^{-\frac{\ga_0}{2}} \dn^\f12,\qquad r|\p \phi|\les u_+^{-\f12-\frac{\ga_0}{2}}\dn^\f12
\end{equation*}
which gives (\ref{5.18.1.18}).
\end{proof}


Let $0\le m\le n$. For the ordered product of vector fields, $Z^n=Z_1\cdots Z_n $, we denote by $Z^m\subset Z^n$ if $Z^m=Z_{k_1}\cdots Z_{k_m}$ with $k_1<k_2\cdots <k_m$.
By $Z^a\sqcup Z^b=Z^n$, we denote a decomposition of  $Z^n$ into  $Z^a$ and $Z^b$. It means,  $Z^a, Z^b\subset Z^n$ with $Z^a=Z_{k_1}\cdots Z_{k_a}$ and $Z^b=Z_{m_1}\cdots Z_{m_b}$, none of the subindices among $(k_1, \cdots,  k_a)$ and  $(m_1, \cdots, m_b)$ are equal, i.e., $k_l\neq m_j$ and $a+b=n$. $Z^{a_1}\sqcup\cdots \sqcup Z^{a_m}$ can be understood inductively.

If $Z_1 Z_2\cdots Z_n$ is regarded as a differential operator, we denote it as $Z\rp{n}$. We set $Z\rp{n\ssm i}=Z_1\cdots Z_{i-1}Z_{i+1}\cdots Z_n$, for $i=1,\cdots,n$.   $Z^{n\ssm i}$ represents the corresponding product of vector fields.
\begin{lemma}\label{comm}
  For each killing vector field $Z$, $[Z,\p_\a]= \tensor{C}{_{Z\a}}^\ga \p_\ga$, where $\tensor{C}{_{Z\a}}^\ga=-\p_\a Z^\ga$ is a $(1,1)$ tensor.  $\tensor{C}{_{Z\a}}^\ga=0$ if $Z=\p$. Due to (\ref{8.1.1.18}),
   the components of $C_Z$ are $1$ or $-1$ if $Z=\Omega_{\mu\nu}$. Thus, symbolically, we may ignore the tensorial feature of $C_Z$, and regard $C_Z$  as constants. The tensor products $C_{Z^m}= C_{Z_1}\cdots C_{Z_m}$ may be regarded as a set of product of constants, since $C_{Z_i}$ is understood as a set of  constants with $|C_{Z_i}|=1$ if $Z_i=\Omega$; and $C_{Z_i}=0$ if $Z_i=\p$. \begin{footnote}{$C_{id}=0$}\end{footnote}
\begin{enumerate}
\item[(1)] For $n=1,2,3$, there holds the symbolic identity
\begin{equation}\label{5.16.1.18}
[\p, Z\rp{n}] f= \sum_{Z^a\sqcup Z^b=Z^n, a\ge 1} C_{Z^a}\p Z\rp{b} f.
\end{equation}
Thus, if $\zeta(Z^n)=-n$, $[\p, Z\rp{n}] f=0$.
\item[(2)]
For $n=1,2,3$
\begin{equation}\label{5.18.11.18}
| Z\rp{n} \p f|\les \sum_{Z^a\sqcup Z^b=Z^n}  r^{\zeta(Z^a)} |\p Z\rp{b} f|.
\end{equation}
\end{enumerate}
\end{lemma}
\begin{remark}
In application, most of the time we will replace $r^{\zeta(Z^a)}$ by $u_+^{\zeta(Z^a)}$ which is a weaker version of the result.
\end{remark}
\begin{proof}
(\ref{5.16.1.18}) follows by direct calculation. It  follows directly from (1) in view of the definition of $\zeta(\cdot)$ that $
|[\p, Z\rp{n}] f|\les \sum_{Z^a\sqcup Z^b=Z^n, a\ge 1}  r^{\zeta(Z^a)} |\p Z\rp{b} f|$. (\ref{5.18.11.18}) follows as its consequence.
\end{proof}
\begin{lemma}\label{5.17.2.18}
Let $\varphi$ be a smooth function and $n=1,2,3$. Under the assumption that $|\N\rp{i}(\varphi)|\le C$ with $i=1\cdots n$, there holds
\begin{equation}\label{5.17.1.18}
|Z\rp{n}(\N(\varphi))|\les |Z\rp{n}\varphi|+\sum_{i=1}^n |Z\rp{n\ssm i} \varphi\c Z_i \varphi|+|\Pi_{i=1}^n Z_i \varphi|
\end{equation}
and consequently
\begin{equation}\label{5.17.3.18}
|Z\rp{n}(\N(\varphi))|\les \sum_{Z^a \sqcup Z^b=Z^n, a\ge 1}|Z\rp{a}\varphi|\dn^{\f12b} u_+^{\zeta(Z^b)}.
\end{equation}
\end{lemma}
\begin{remark}
  Under the bootstrap assumption  $(\BA{2})$,  (\ref{3.24.11.18}) holds, which imply $|\phi|\les r^{-1}(\E^\f12_{2,\ga_0}+\dn^\f12)\les 1$. Since $\N(y)$ is smooth, we can obtain
$
|\N\rp{i}(\phi)|\les 1,\, i\le  k
$
  for any fixed $k\in {\mathbb N}$.  So the assumption holds for $\varphi=\phi$. We also remark that we only used $|Z\varphi|\les \dn^\f12 u_+^{\zeta(Z)}$ to prove the above result.
\end{remark}
\begin{proof}
It is straightforward to derive
\begin{align*}
&Z\rp{1}(\N(\varphi))=\N'(\varphi) Z\rp{1} \varphi\\
&Z\rp{2}(\N(\varphi))=\N'(\varphi) Z\rp{2} \varphi+ \N^{''}(\varphi)Z_2\varphi Z_1\varphi \\
&Z\rp{3}(\N(\varphi))=\N'(\varphi)Z\rp{3} \varphi+\N^{''}(\varphi)\sum_{i=1}^3 Z\rp{n\ssm i} \varphi \c Z_i \varphi+\N^{'''}(\varphi)\Pi_{i=1}^3Z_i\varphi.
\end{align*}
We then can derive (\ref{5.17.1.18}) for $n=1,2,3$. (\ref{5.17.3.18}) follows by using (\ref{5.18.1.18}).
\end{proof}
\subsection {Error estimates}
We will improve the bootstrap assumptions (\ref{3.21.5.18}) and (\ref{3.21.6.18}) by deriving energy estimates, with the help of Proposition \ref{5.16.5.18} and Proposition \ref{6.12.2.18}. For deriving both types of estimates for $Z\rp{i}\phi$, the main task is to obtain the error  estimates on $(\Box_\bm -q)Z\rp{i} \phi$ with $i\le 2$.  We analyze  in the following result these major error terms.


\begin{lemma}\label{6.3.5.18}
Let $\F=\N(\phi) \p \phi \c \p \phi $,  $\str{n}\F=Z\rp{n}\F$ and $\str{0}\F=\F$.
For $n=0, \cdots,  3$,

(1)
\begin{equation}\label{6.3.4.18}
|Z\rp{n} (\p \phi\c \p \phi)|\les \sum_{Z^a\sqcup Z^b\sqcup Z^c=Z^n}  u_+^{\zeta(Z^c)} |\p Z\rp{b} \phi||\p Z\rp{a}\phi|.
\end{equation}
(2)
\begin{equation}\label{6.3.6.18}
|\str{n} \F|\les \str{n} \F_\Q+\str{n}\F_\C
\end{equation}
where the quadratic part and the cubic part  are
\begin{equation*}
\begin{split}
 \str{n}\F_\Q&= \sum_{Z^a\sqcup Z^b\sqcup Z^c=Z^n} u_+^{\zeta(Z^c)} |\p Z\rp{b} \phi||\p Z\rp{a}\phi|\\
\str{n}\F_\C &=\sum_{Z^{a}\sqcup Z^b\sqcup Z^c\sqcup Z^d=Z^n, 1\le a\le n }|Z\rp{a} \phi|  |\p Z\rp{b} \phi| |\p Z\rp{c} \phi| u_+^{\zeta(Z^d)}.
\end{split}
\end{equation*}
\end{lemma}
\begin{remark}
The result with $n=3$ will be used in Section \ref{quasi} and \ref{Eins}.
\end{remark}
\begin{proof}
We first can obtain (\ref{6.3.4.18}) by using (\ref{5.18.11.18}) and
\begin{equation*}
Z\rp{n}(\p \phi \c \p \phi)= \sum_{Z^a\sqcup Z^b =Z^n}Z\rp{a}\p \phi \c Z\rp{b} \p \phi.
\end{equation*}

Next we derive the estimates of $\str{n}\F$ in view of
\begin{equation}\label{6.3.7.18}
\str{n}\F= (\sum_{a\ge 1}+\sum_{a=0})\sum_{Z^a \sqcup Z^b=Z^n}Z\rp{a} \big(\N(\phi)\big)Z\rp{b}(\p \phi \c \p \phi).
\end{equation}
The $a=0$ term can be bounded by using (\ref{6.3.4.18}) directly.
For the terms with $a\ge 1$ in (\ref{6.3.7.18}), we can apply  (\ref{5.17.3.18}) with $n=a$ and  (\ref{6.3.4.18}) with $n=b$ to derive the cubic type of terms.  We then combine the estimates for $0\le a\le n$ to obtain
\begin{equation*}
\begin{split}
|\str{n}\F|&\les \sum_{Z^{a_1}\sqcup Z^b\sqcup Z^c\sqcup Z^d=Z^n}\sum_{ 1\le a_1\le a\le n} \dn^{\f12(a-a_1)}|Z\rp{a_1} \phi| |\p Z\rp{b} \phi| |\p Z\rp{c} \phi| u_+^{\zeta(Z^d)} \\
& +\sum_{Z^a\sqcup Z^b\sqcup Z^c=Z^n} u_+^{\zeta(Z^c)} |\p Z\rp{b} \phi||\p Z\rp{a}\phi|.
\end{split}
\end{equation*}
The second line is the quadratic term $\str{n}\F_\Q$. The first line on the righthand side is a sum of  cubic terms of $\phi$.
By the boundedness of  $\dn^{a-a_1}$,  we can obtain the formula for $\str{n}\F_\C$ in (\ref{6.3.6.18}).
\end{proof}
As an important remark, we can write according to Lemma \ref{6.3.5.18} that
\begin{equation*}
|\str{1}\F|\les (|\p Z\phi|+u_+^{\zeta(Z)}|\p \phi|)|\p \phi|,
\end{equation*}
for which  the cubic term is already controlled by using (\ref{5.18.1.18}). Thus, symbolically,
\begin{equation}\label{7.5.1.18}
|\str{1}\F|\les \str{1}\F_\Q.
\end{equation}

\begin{proposition}\label{6.2.5.18} For $n\le 2$ and $-\ub_*\le -\ub_1\le u_1\le u_0$,  the following estimates hold
\begin{align}
&{u_1}_+^{\f12 \ga_{0}+\f12-\zeta(Z^n)}\| r^\f12 \str{n} \F_\Q\|_{L^2(\D_{u_1}^{\ub_1})}\les {u_1}_+^{-\f12\ga_0+\f12} \dn M^{-\f12} , \label{6.13.2.18}\\
&{u_1}_+^{\f12\ga_0-\f12-\zeta(Z^n)}\| r^{\frac{3}{2}} \str{n}\F_\Q\|_{L_u^1 L_\ub^2 L_\omega^2(\D_{u_1}^{\ub_1})}\les {u_1}_+^{-\frac{\ga_0}{2}}  \dn M^{-\f12},\label{6.14.6.18}\\
&\|r^\f12 \str{2} \F_\C\|_{L^2(\D_{u_1}^{\ub_1})}\les {u_1}_+^{-\frac{3\ga_0}{2}-\f12+\zeta(Z^2)}\dn^\frac{3}{2}M^{-\f12},\label{6.3.8.18}\\
&\|r^\frac{3}{2} \str{2} \F_\C\|_{L_u^1 L_\ub^2 L_\omega^2(\D_{u_1}^{\ub_1})}\les {u_1}_+^{-\frac{3}{2}\ga_0-\f12+\zeta(Z^2)}\dn^\frac{3}{2}.\label{6.3.12.18}
\end{align}
\end{proposition}
\begin{proof}
 We first  decompose $\str{n}\F_\Q$  as below
\begin{equation}\label{6.7.3.18}
\begin{split}
 \str{n}\F_{\Q,1}&=\sum_{Z^a\sqcup Z^b\sqcup Z^c=Z^n, b=0}  u_+^{\zeta(Z^c)} |\p Z\rp{b}\phi| |\p Z\rp{a}\phi|,\\
\str{n}\F_{\Q,2}&=\sum_{Z^a\sqcup Z^b\sqcup Z^c=Z^n, b\ge 1}u_+^{\zeta(Z^c)} |\p Z\rp{b} \phi| |\p Z\rp{a}\phi|,
\end{split}
\end{equation}
where we assume $b\le a$ without loss of generality. We will frequently use (\ref{3.21.5.18}) and (\ref{3.21.6.18}) in the sequel.

Note that with $a\le n$, we can apply  (\ref{3.24.18.18}) and (\ref{7.5.4.18}) to derive
\begin{equation}\label{6.13.3.18}
\begin{split}
\| r^\f12 \str{n}\F_{\Q,1}\|_{L^2(\D_{u_1}^{\ub_1})}&\les\sum_{Z^a\sqcup Z^b\sqcup Z^c=Z^n, b=0}  {u_1}_+^{\zeta(Z^c)}\|r\p Z\rp{b}\phi\|_{L^\infty(\D_{u_1}^{\ub_1})}\|r^{-\f12} \p Z\rp{a}\phi \|_{L^2(\D_{u_1}^{\ub_1})}\\
&\les \dn M^{-\f12}{u_1}_+^{-\ga_0+\zeta(Z^n)}.
\end{split}
\end{equation}
We note that by (\ref{7.5.1.18}) in the case of $n\le 1$, $\str{n} \F_{\Q,2}$ vanishes. If $n=2$, $a=b=1\le n-1$.
In view of (\ref{7.5.2.18}) and (\ref{6.14.3.18}), we deduce  for  $\str{2}\F_{\Q,2}$ that
\begin{align*}
&\|r^\f12 \str{n}\F_{\Q,2}\|_{L^2(\D_{u_1}^{\ub_1})}\\
&\les \sum_{Z^a\sqcup Z^b\sqcup Z^c=Z^n, b\ge 1}u_+^{\zeta(Z^c)}\|r \p Z\rp{b} \phi\|_{L_\ub^\infty L_u^2 L_\omega^4(\D_{u_1}^{\ub_1})}\| r^\f12 \p Z\rp{a} \phi\|_{L_\ub^2 L_u^\infty L_\omega^4(\D_{u_1}^{\ub_1})}\\
&\les  \dn M^{-\f12} {u_1}_+^{-\ga_0+\zeta(Z^n)}.
\end{align*}
(\ref{6.13.2.18}) follows by combining the estimates of $\str{n}\F_{\Q,1}$ and $\str{n}\F_{\Q,2}$.

Next, we  prove  (\ref{6.14.6.18}). Using (\ref{3.24.18.18}) and (\ref{7.5.4.18}) implies
\begin{align}
&\| r^\frac{3}{2}\str{n}\F_{\Q,1}\|_{L_u^1 L_\ub^2 L_\omega^2(\D_{u_1}^{\ub_1})}\nn\\
&\les \sum_{Z^a\sqcup Z^c=Z^n} {u_1}_+^{\zeta(Z^c)} \|r^\f12 \p Z\rp{a} \phi u_+^{-\frac{1+\ga_0}{2}}\|_{L_u^1 L_\ub^2 L_\omega^2(\D_{u_1}^{\ub_1})}\|r \p \phi\c u_+^{\frac{1+\ga_0}{2}}\|_{L^\infty(\D_{u_1}^{\ub_1})}\nn\\
&\les \dn^\f12\sum_{Z^a\sqcup Z^c=Z^n} {u_1}_+^{\zeta(Z^c)}\|r^{-\f12} \p Z\rp{a} \phi \|_{ L^2(\D_{u_1}^{\ub_1})}\|u_+^{-\f12\ga_0-\f12}\|_{L_u^2L^\infty (\D_{u_1}^{\ub_1})}\nn\\
&\les \dn M^{-\f12}  {u_1}_+^{\zeta(Z^n)-\ga_0+\f12}.\nn\
\end{align}
In the case that $1\le b\le a$, again $a\le n-1$. By using  (\ref{6.17.11.18}) and (\ref{7.5.3.18}),  we have
\begin{align*}
&\| r^\frac{3}{2}\str{n}\F_{\Q,2}\|_{L_u^1 L_\ub^2 L_\omega^2(\D_{u_1}^{\ub_1})}\\
&\les \sum_{Z^a\sqcup Z^b\sqcup Z^c=Z^n,  b\ge 1} {u_1}_+^{\zeta(Z^c)}\| r\p Z\rp{b} \phi\|_{L_u^2 L_\ub^\infty L_\omega^4(\D_{u_1}^{\ub_1})}\c \| r^\f12 \p Z\rp{a}\phi\|_{L_u^2 L_\ub^2 L_\omega^4(\D_{u_1}^{\ub_1})}\\
&\les \dn M^{-\frac{1}{2}}  {u_1}_+^{\zeta(Z^n)-\ga_0+\f12}.
\end{align*}
 By combining the above two estimates, we can obtain (\ref{6.14.6.18}).

 Next we consider the estimates of   $\str{2}\F_\C$ by using (\ref{3.24.18.18}), (\ref{6.14.3.18}) and  (\ref{6.14.10.18}),
 \begin{align*}
& \|r^\f12\str{2} \F_\C\|_{L^2(\D_{u_1}^{\ub_1})}\nn\\
 &\les \sum_{Z^a\sqcup Z^b \sqcup Z^c \sqcup Z^d=Z^2, a\ge 1, b\le c}\|Z\rp{a}\phi\|_{L_u^\infty L_u^2 L_\omega^4(\D_{u_1}^{\ub_1})}\|r \p Z\rp{b}\phi\|_{L^\infty(\D_{u_1}^{\ub_1})}\|r^\f12 \p Z\rp{c}\phi\|_{L_\ub^2 L_u^\infty L_\omega^4(\D_{u_1}^{\ub_1})}\\
 &\les \dn^\frac{3}{2}M^{-\f12} {u_1}_+^{-\frac{3\ga_0}{2}-\f12+\zeta(Z^2)}.
 \end{align*}
  Here we assumed $b\le c$ without loss of generality, which implies $b=0$. By using (\ref{3.24.18.18}), (\ref{6.17.11.18}) and (\ref{6.3.11.18}),  we have
 \begin{align*}
 \|r^\frac{3}{2}&\str{n}\F_\C\|_{L_u^1 L_\ub^2 L_\omega^2(\D_{u_1}^{\ub_1})}\\
 &\les\sum_{Z^a\sqcup Z^b \sqcup Z^c \sqcup Z^d=Z^2, a\ge 1, b\le c}{u_1}_+^{\zeta(Z^d)}\|r^{-\f12}Z\rp{a}\phi\|_{L_u^2 L_\ub^2 L_\omega^4(\D_{u_1}^{\ub_1})}\| r \p Z\rp{b}\phi\|_{L^\infty(\D_{u_1}^{\ub_1})}\\
&\c \|r \p Z\rp{c}\phi\|_{L_u^2 L_\ub^\infty L_\omega^4(\D_{u_1}^{\ub_1})}\\
 &\les \dn^\frac{3}{2} {u_1}_+^{\zeta(Z^2)-\frac{3}{2}\ga_0-\f12}.
 \end{align*}
  Thus (\ref{6.3.8.18}) and (\ref{6.3.12.18}) are both proved.
\end{proof}

\begin{lemma}\label{5.25.2.18}
Under the assumption of (\ref{potential}), there hold for $n\le 3$ and $-\ub_*\le -\ub_1\le u_1\le u_0$ that
\begin{align*}
{u_1}_+^{-\zeta(Z^n)+\frac{\ga_0}{2}} &(\|r [Z\rp{n}, q] \phi \|_{L_u^1 L_\ub^2 L_\omega^2(\D_{u_1}^{\ub_1})} +\|r [Z\rp{n}, q]\phi \|_{L_\ub^1 L_u^2 L_\omega^2(\D_{u_1}^{\ub_1})})\\
&\les \E^\f12_{n,\ga_0}+\sup_{-u_1\le \ub\le \ub_1}\sum_{Z^m\subsetneq Z^n}\big( E[Z\rp{m} \phi]^\f12(\Hb_{\ub}^{u_1}) {u_1}_+^{\frac{\ga_0}{2}-\zeta(Z^{m})}\big),\\
{u_1}_+^{-\zeta(Z^n)+\f12\ga_0-\f12}&\|r^\frac{3}{2} [Z\rp{n}, q]\phi\|_{L_u^1 L_\ub^2 L_\omega^2(\D_{u_1}^{\ub_1})}\\
& \les \E^\f12_{n,\ga_0}+ \sum_{Z^m\subsetneq Z^n} {u_1}_+^{-\zeta(Z^m)+\frac{\ga_0-1}{2}} \big(\W_1[Z\rp{m} \phi](\D_{u_1}^{\ub_1})\\
&+M\int_{-\ub_1}^{u_1} u_+^{-1}E[Z\rp{m}\phi](\H_u^{\ub_1}) du\big)^\f12.
\end{align*}
\end{lemma}
\begin{proof}
It is direct to obtain
\begin{equation*}
[Z^n, q]\phi= \sum_{i=1}^n Z\rp{i}q Z\rp{n-i}\phi,
\end{equation*}
where all $Z^{n-i}\subsetneq Z^n$.
Note that the assumption (\ref{potential}) on $q$ implies
\begin{equation}\label{7.5.5.18}
\|r^{\frac{3}{2}+\f12\eta} Z\rp{i} q\|_{L_u^2 L^\infty(\D_{u_1}^{\ub_1})}+\|r^{\frac{3}{2}+\f12\eta} Z\rp{i}q\|_{L_\ub^2 L^\infty(\D_{u_1}^{\ub_1})}\les {u_1}_+^{\zeta(Z^i)}.
\end{equation}
Thus
\begin{align*}
\|r [Z\rp{n}, q] \phi \|_{L_u^1 L_\ub^2 L_\omega^2(\D_{u_1}^{\ub_1})} &+\|r [Z\rp{n}, q]\phi \|_{L_\ub^1 L_u^2 L_\omega^2(\D_{u_1}^{\ub_1})}\\
&\les \sum_{ 1\le i\le n}{u_1}_+^{\zeta(Z^i)}\|r^{-\frac{3}{2}-\f12\eta}Z\rp{n-i}\phi\|_{L^2(\D_{u_1}^{\ub_1})}.
\end{align*}
By using  (\ref{4.29.6.18}), we can derive
\begin{equation*}
\begin{split}
\|r^{-\frac{3}{2}-\f12\eta}Z\rp{n-i} \phi\|_{L^2(\D_{u_1}^{\ub_1})}& \les \sup_{-u_1\le \ub \le \ub_1 } \|r^{-1}Z\rp{n-i} \phi\|_{L^2(\Hb_{\ub}^{u_1})} \\
&\les \sup_{-u_1\le \ub\le \ub_1}\big( (\int_{S_{-\ub,\ub}} r (Z\rp{n-i}\phi)^2 d\omega)^\f12+ E[Z\rp{n-i} \phi]^\f12(\Hb_{\ub}^{u_1})\big).
\end{split}
\end{equation*}
We can apply (\ref{6.23.7.18}) to $Z\rp{n-i}\phi$ to control the first term. The first inequality of Lemma \ref{5.25.2.18} holds by combining the above two estimates.

To see the second inequality, we have by using (\ref{7.5.5.18}) that
\begin{align*}
\|r^\frac{3}{2} [Z\rp{n}, q] \phi \|_{L_u^1 L_\ub^2 L_\omega^2(\D_{u_1}^{\ub_1})}\les \sum_{ 1\le i\le n}{u_1}_+^{\zeta(Z^i)}\|r^{-1-\f12\eta}Z\rp{n-i}\phi\|_{L^2(\D_{u_1}^{\ub_1})}.
\end{align*}
It follows by using  (\ref{3.23.2.18}) that
\begin{align*}
\|r^{-1-\frac{\eta}{2}} Z\rp{n-i}\phi\|^2_{L^2(\D_{u_1}^{\ub_1})}&\les \W_1[Z\rp{n-i} \phi](\D_{u_1}^{\ub_1})+M\int_{-\ub_1}^{u_1} u_+^{-1}E[Z\rp{n-i}\phi](\H_u^{\ub_1}) du\\
&+\int_{\Sigma_0^{u_1,\ub_1}} r (Z\rp{n-i} \phi)^2 du d\omega.
\end{align*}
By combining the above two inequalities and applying (\ref{6.25.1.18}) to $Z\rp{n-i}\phi$, we can obtain Lemma \ref{5.25.2.18}.
\end{proof}
\subsection{Boundedness of energies}
Next, we will use the fact
\begin{equation}\label{5.25.3.18}
\Box_\bm Z\rp{n} \phi-q Z\rp{n} \phi=\str{n} \F+[Z\rp{n}, q]\phi
\end{equation}
Proposition \ref{6.2.5.18}, Lemma \ref{5.25.2.18}, Proposition \ref{5.16.5.18} and Proposition \ref{6.12.2.18} to prove the boundedness of energies in (\ref{6.28.5.18}).

\begin{proposition}\label{3.25.3.18}
Let $\I=\{(u,\ub),-\ub_*\le  -\ub \le u\le u_0\}$ and $n\le 2$.  For $(u, \ub)\in \I$ with $u\le u_1\le u_0$, there hold
\begin{align}
u_+^{\ga_0-2\zeta(Z^n)}&( E[Z\rp{n} \phi](\H_u^\ub)+E[Z\rp{n} \phi](\Hb_\ub^u))\les\E_{n,\ga_0}+M^{-2} \dn^2 {u_1}_+^{1-\ga_0},  \label{6.14.8.18}\\
 u_+^{-2\zeta(Z^n)-1+\ga_0}&\left(\W_1[Z\rp{n} \phi](\H_{u}^{\ub})+\W_1[Z\rp{n} \phi](\Hb_{\ub}^u)+\W_1[Z\rp{n} \phi](\D_u^{\ub})\right)\nn\\
&\les \E_{n, \ga_0} +M^{-2}\dn^2 {u_1}_+^{1-\ga_0}.\nn
\end{align}
\end{proposition}
\begin{proof}
In view of (\ref{5.25.3.18}), we will use Proposition \ref{5.16.5.18} with $ \F^\sharp=\str{n}\F_\Q+\str{n}\F_\C$ and $\F^\flat=[Z\rp{n},q]\phi$. By using   (\ref{6.13.2.18}), (\ref{6.3.8.18}) and Lemma \ref{5.25.2.18}, we have
\begin{align*}
 u_+^{\ga_0-2\zeta(Z^n)}&( E[Z\rp{n} \phi](\H_u^\ub)+E[Z\rp{n} \phi](\Hb_\ub^u))\\
&\les \E_{n,\ga_0}+M^{-2}\dn^2 {u_1}_+^{1-\ga_0 }+ \sup_{-\ub_*\le -\ub\le u_1}\sum_{i=1}^n  E[Z\rp{n-i} \phi](\Hb_\ub^{u_1}) {u_1}_+^{\ga_0-2\zeta(Z^{n-i})},
\end{align*}
where the last term vanishes when $n=0$. This implies the first estimate in Proposition \ref{3.25.3.18} by induction.

In view of (\ref{6.14.6.18}) and (\ref{6.3.12.18}) in Proposition \ref{6.2.5.18}  for $\str{n} \F_\Q$ and $\str{n} \F_\C$, we can derive  for $(u_1, \ub_1)\in \I$,
\begin{align*}
{u_1}_+^{-\zeta(Z^n)+ \ga_0-\f12} \|r^\frac{3}{2}\str{n}\F\|_{L_u^1 L_\ub^2 L_\omega^2(\D_{u_1}^{\ub_1})}&\les \dn M^{-\f12}.
\end{align*}
By using the above estimate, Proposition \ref{6.12.2.18}, (\ref{6.23.7.18}), (\ref{6.25.1.18}), (\ref{6.30.1.18}),  the second estimate in Lemma \ref{5.25.2.18} and  (\ref{6.14.8.18}), we  can derive for $(u_1, \ub_1)\in \I$,
\begin{align*}
&{u_1}_+^{-2\zeta(Z^n)-1+\ga_0}\left(\W_1[Z\rp{n} \phi](\H_{u_1}^{\ub_1})+\W_1[Z\rp{n} \phi](\Hb_{\ub_1}^{u_1})+\W_1[Z\rp{n} \phi](\D_{u_1}^{\ub_1})\right)\\
&\les \E_{n, \ga_0}+\dn^2 {u_1}_+^{-\ga_0} (M^{-1}+M^{-2}{u_1}_+)+\sum_{i=1}^n {u_1}_+^{-2\zeta(Z^{n-i})-1+\ga_0} \W_1[Z\rp{n-i} \phi](\D_{u_1}^{\ub_1}).
\end{align*}
Note that when $n=0$ the last term on the right vanishes. The weighted energy estimate can then be derived by induction.
\end{proof}

{\bf Improvement on the bootstrap assumption}
Let us denote the universal constant in $\les $ in the estimates of Proposition \ref{3.25.3.18}  as $C_3>0$. We need to show that
\begin{equation*}
C_3(\E_{2,\ga_0}+M^{-2}\dn^2 u_+^{1-\ga_0})<2\dn, \forall u\le u_0
\end{equation*}
Recall that $\dn=C_1 \E_{2,\ga_0}$  with $C_1$ to be chosen, and $\E_{2,\ga_0}\le C M^2$ with $C\ge 1$  a fixed  constant. Thus we need to choose $C_1$ such that
  $$
  C_3(C_1^{-1} +M^{-2} \dn u_+^{1-\ga_0}) <2.
  $$
  Let $C_1=4C_3$.  It is reduced to  $M^{-2}\dn u_+(R)^{1-\ga_0}< \frac{7}{4C_3}$. Since $\delta_1$ in Theorem \ref{thm2} can be  sufficiently small, $M<\frac{1}{10}$ can be achieved. Thus $u_+(R)>\f12 R$ can be guaranteed. Thus we need
  \begin{equation}\label{5.28.1.18}
  (\frac{R}{2})^{1-\ga_0}<\frac{7}{16 C_3^2  C}.
  \end{equation}
  $R$  is fixed to satisfy the  inequality of (\ref{5.28.1.18}). Thus the proof of Theorem \ref{thm2} is completed. If $R=2$ but $C$ is allowed to be chosen, with $C<\delta_0$ such that $\frac{16}{7} C_3^2 \delta_0< (\frac{R}{2})^{\ga_0-1}=1$, (\ref{5.28.1.18})  can also be achieved. This proves Theorem \ref{thm1}.

\section{Quasilinear equations}\label{quasi}
In this section, we consider the general quasilinear equations (\ref{eqn_1}) in ${\mathbb R}^{3+1}$ which verifies (\ref{potential}) for $n=3$.  $\bg(\phi, \p\phi)$ and $\N^{\a\b}(\phi)$ are both smooth functions of their arguments,  $\bg(0, \mathbf{0})=\bm$. For convenience, we set $\Phi=(\phi, \p \phi)$, which is a $5$- vector valued function, and $|\Phi|=|\phi|+\sum_{\mu=0}^3|\p^\mu\phi| $.  We will prove Theorem \ref{thm_quasi} in this section.

\subsection{Bootstrap assumptions}
Due to the influence of the metric $\bg^{\a\b}$,
the bootstrap assumptions are more delicate than (\ref{3.21.5.18}) and (\ref{3.21.6.18}) in Section \ref{semi}. We first need to fix the constant $M_0$ since the region where the stability result holds is determined by $u_0(M_0)$.   By  using $H^{\a\b}(0, \mathbf{0})=0$ and the fact that $H^{\a\b}$ are smooth functions of the arguments, we can derive
\begin{equation*}
\begin{split}
\sup_{0\le \a, \b\le n } |H^{\a\b}(\phi, \p \phi)|\les|\phi|+|\p \phi|  \mbox{ on } \Sigma_0\cap \{r\ge R\}.
\end{split}
\end{equation*}
By using the above estimate and (\ref{6.24.9.18}), there holds
\begin{equation}\label{7.23.4.18}
\sup_{\a, \b}r |H^{\a\b}|\le C u_+^{-\frac{\ga_0}{2}+\f12} \E^\f12_{2,\ga_0, R} \quad \mbox{ on } \Sigma_0\cap \{r\ge R\}.
\end{equation}
 We can choose $M_0= 3 C \delta^\f12_1$ in the definition of (\ref{metric}) and $h=\frac{M_0}{r}$.   By this choice and (\ref{6.24.9.18}) on $\Sigma_0\cap \{r\ge R\}$ there holds
 \begin{equation}\label{6.30.4.18}
 \sup_{\a,\b}r (h-H^{\a\b})> C(3 \delta^{\f12}_1-  u_+^{-\frac{\ga_0}{2}+\f12} \E^\f12_{2,\ga_0,R}).
 \end{equation}
Since $ u_+(R)\ge \f12 R$, with $(\frac{R}{2})^{-\frac{\ga_0}{2}+\f12}<\f12$, on $\Sigma_0\cap \{r\ge R\}$,
\begin{equation*}
\sup_{\a,\b}r (h-H^{\a\b})> \frac{5}{2}C \delta^{\f12}_1=\frac{5}{6}M, \qquad M=M_0.
\end{equation*}

The bootstrap assumption for proving Theorem \ref{thm_quasi} consists of the control of the metric and  the boundedness of energies.

Let  $\ub_*> -u_0$ be a fixed number and let $\I=\{(u,\ub): -\ub_*\le -\ub \le u\le u_0 \}$. We suppose there hold on $\I$ that
\begin{equation}\label{6.10.4.18}\tag{A1}
r(h-H^{\Lb \Lb })\ge\frac{M}{3}, \quad r(h-H^{LL})\ge\frac{M}{3},
\end{equation}
and   for any $(\ub_1,  u_1)\in \I$, $n\le 3$,
\begin{align}
&E[Z\rp{n} \phi](\H_{u_1}^{\ub_1})+E[Z\rp{n} \phi](\Hb_{\ub_1}^{u_1})\le 2\dn {u_1}_+^{-\ga_0+2\zeta(Z^n)},\label{4.25.1.18}\\
&\W_1[Z\rp{n} \phi](\H_{u_1}^{\ub_1})+\W_1[Z\rp{n} \phi](\Hb_{\ub_1}^{u_1})+\W_1[Z\rp{n}\phi](\D_{u_1}^{\ub_1})\le 2\dn {u_1}_+^{-\ga_0+1+2\zeta(Z^n)},\label{4.25.2.18}
\end{align}
where $\dn \le C_1 M^2$ with $C_1>1$ to be chosen later and $Z\in\{\Omega_{ij}, \p\}$.

As a consequence of the above assumptions, all the estimates in Section \ref{decay} hold.  We can first summarize some of  the decay estimates that will be frequently used in this section.
\begin{proposition}[Decay estimates]\label{decay_quasi}
There hold the following decay estimates \begin{footnote}{$Z\rp{n}H$ is understood as $Z\rp{n}(H(\Phi)).$ }\end{footnote} in $\I$
\begin{align}
&r^2|\p\rp{l} Z\rp{n} \phi |^2 \les \dn u_+^{-\ga_0+1-2l+2\zeta(Z^n)}, \quad l , n\le 1\label{6.7.1.18}\\
&r^2|Z\rp{n}H|^2\les \dn u_+^{-\ga_0+1+2\zeta(Z^n)},\quad n\le 1, \label{4.25.3.18} \\
&r^3 |\bar \p H|^2+ r^2 u_+|\Lb H|^2\les \dn u_+^{-\ga_0},\label{7.23.2.18} \\
& \|r^\f12 \p H\|^2_{L_\ub^2 L^\infty (\D_{u_1}^{\ub_1})}\les  \dn M^{-1} {u_1}_+^{-\ga_0},\label{7.23.3.18}
\end{align}
where $(u_1, \ub_1)\in \I$.
\end{proposition}
\begin{proof}
If $l=1$, (\ref{6.7.1.18}) is a consequence of (\ref{3.24.18.18}); if $l=0$, it is  the estimate of (\ref{3.24.11.18}).
 The result of (\ref{6.7.1.18}) with $n=0$ , $H(0, \mathbf{0})=0$ and  the fact that $H$ is smooth  imply there hold for all $(u, \ub)\in \I$ that
 \begin{equation}\label{5.2.5.18}
   | H(\Phi)|\les |\phi|+|\p \phi|; \quad |(D\rp{i} H)(\Phi)|\les 1, \, i\ge 1.
   \end{equation}
The $n=0$ case in (\ref{4.25.3.18}) can then be derived by using (\ref{6.7.1.18}) with $(n,l)=(0,1)$, $(0,0)$.
Due to (\ref{5.2.5.18}),  $|Z\rp{1}\big(H(\Phi)\big)|\les |Z\rp{1}\phi|+|Z\rp{1}\p \phi|$, by also using  (\ref{5.18.11.18}) and (\ref{6.7.1.18}) we have
\begin{equation*}
r|Z\rp{1} \p \phi|\les u_+^{\zeta(Z^1)-\f12-\frac{\ga_0}{2}}\dn^\f12.
\end{equation*}
The estimate for $Z\rp{1}\phi$ follows from (\ref{6.7.1.18}). Thus $n=1$ case in (\ref{4.25.3.18}) is treated.
(\ref{7.23.2.18}) follows by using (\ref{5.2.5.18}) and (\ref{3.24.18.18}); similarly, (\ref{7.23.3.18})  can be proved by using (\ref{4.2.1.18}).
\end{proof}
\subsection{Energy and weighted energy inequalities}\label{eng}
In this subsection, we derive the fundamental energy estimate and $r$-weighted energy estimate for (\ref{eqn_1}).

 We will always use the Minkowski metric to lift and lower the indices. For example $H_\a^\b:=\bm_{\a\a'}H^{\a'\b}$. We define the following  $(0,2)$ tensor, which is not necessarily symmetric,
$$
\ti \sQ_{\a\b}[\varphi]=\sQ_{\a\b}[\varphi]+H_\a^\ga \p_\ga \varphi \p_\b \varphi,
$$
where
\begin{equation*}
\sQ_{\a\b}[\varphi]=\p_\a\varphi\p_\b \varphi-\f12 \bm_{\a\b} (\bg^{\rho \sigma} \p_\rho \varphi \p_\sigma \varphi+q\varphi^2).
\end{equation*}
\begin{lemma}\label{6.10.7.18}
Let $X$ be a smooth vector field. There holds
\begin{align}
\p^\a (\ti\sQ_{\a\b}X^\b)&=X^\b\{(\widetilde{\Box}_\bg \varphi-q \varphi+\p^\a H_\a^\ga \p_\ga \varphi)\c \p_\b \varphi-\f12 \p_\b q \c\varphi^2 \nn\\
&-\f12\p_\b H_\a^\ga \p_\ga \varphi \c \p^\a \varphi\}+\sQ_{\a\b}{}\rp{X}\pi^{\a\b}+H_\a^\ga \p_\ga \varphi \p_\b \varphi \p^\a X^\b.\label{4.10.1.18}
\end{align}
\end{lemma}
\begin{proof}
We calculate $\p^\a \ti \sQ_{\a\b}$ below,
\begin{align*}
\p^\a(\ti \sQ_{\a\b} X^\b)&=\p^\a \ti \sQ_{\a\b} X^\b+\sQ_{\a\b} {}\rp{X}\pi^{\a\b}+H_\a^\ga \p_\ga \varphi\p_\b \varphi\p^\a X^\b.
\end{align*}
For the first term, we proceed as follows
\begin{align*}
\p^\a \ti \sQ_{\a\b}X^\b&=(\Box_\bm \varphi\c \p_\b \varphi-\f12 \p_\b(H^{\rho\sigma} \p_\rho \varphi \p_\sigma\varphi)+\p^\a (H_\a^\ga \p_\ga \varphi \p_\b \varphi))X^\b\\
&=X^\b\{((\Box_\bm -q)\varphi+H_\a^\ga \p^\a\p_\ga \varphi+\p^\a H_\a^\ga \p_\ga \varphi)\p_\b \varphi-\f12 \p_\b H_\a^\ga \p_\ga \varphi \p^\a \varphi-\f12 \p_\b q \varphi^2\}
\end{align*}
as desired.
\end{proof}

We now give the energy density on $\H_u^\ub$, $\Hb_\ub^u$ and $\Sigma_t$.
\begin{lemma}\label{6.10.6.18}
There hold the following identities for energy densities,
\begin{equation}\label{4.10.2.18}
\begin{split}
\ti\sQ_{\a\b}\p_t^\b (L^\a+h\Lb^\a)&=\f12(\Lb \varphi)^2 \big( h +(h-1) H^{\Lb \Lb } -2h \tensor{H}{^{L\Lb}}\big)\\
&+\f12 \left((L \varphi)^2+(1+h)(|\sn \varphi|^2+q\varphi^2) \right)+(H+h H) \p \varphi \bar \p \varphi,
\end{split}
\end{equation}
\begin{equation}\label{4.24.1.18}
\begin{split}
 \ti \sQ_{\a\b} \p_t^\b (\Lb^\a+h L^\a)&=\f12(L \varphi)^2\big(h+(h-1)H^{LL}-2 h H^{\Lb L}\big)\\
&+\f12\big((\Lb \varphi)^2+(1+h)(|\sn \varphi|^2+q\varphi^2)\big)+(H+h H) \ud\p\varphi  \p \varphi,
\end{split}
\end{equation}
and
\begin{equation}\label{4.24.2.18}
\ti \sQ_{\a\b}\p_t^\a \p_t^\b= \f12\{-\bg^{00}(\p_t\varphi)^2+\bg^{ij}\p_i \varphi\p_j \varphi+q\varphi^2\},
\end{equation}
where $\bar \p =(L,  \sn)$, $\ud\p= (\Lb ,  \sn)$.
\end{lemma}
\begin{proof}
For the energy density on $\H_u^\ub$, we derive
\begin{align*}
&\ti \sQ_{\a\b} \p_t^\b (L^\a+h\Lb^\a)\\
&=\f12 (L \varphi+h\Lb\varphi)(L \varphi+\Lb \varphi)+\f12 (h+1)\big(-\Lb \varphi L\varphi+|\sn \varphi|^2+q\varphi^2+H^{\Lb \Lb} (\Lb \varphi)^2\\
&\quad+H \p \varphi \c \bar\p\varphi\big)+(-2H^{\Lb \ga} \p_\ga \varphi \c \p_t \varphi+h H_\Lb ^\ga\p_\ga \varphi \p_t\varphi)\\
&=\f12( (L \varphi)^2+h (\Lb\varphi)^2)+\f12 (h+1) \big(|\sn \varphi|^2+q\varphi^2+H^{\Lb \Lb } (\Lb \varphi)^2+H\p \varphi \bar \p \varphi\big)\\
&-\big(H^{\Lb \Lb }(\Lb \varphi)^2 +H \p \varphi \bar \p \varphi\big)+ h (\f12\tensor{H}{_\Lb^\Lb} (\p_\Lb \varphi)^2+H \bar\p \phi \p \phi).
\end{align*}
Thus,  (\ref{4.10.2.18}) is proved. The energy density   (\ref{4.24.1.18}) on $\Hb_\ub^u$ can be derived by directly swapping $L$ and $\Lb$ in the above calculation.

The energy density on  $\Sigma_t$ can be derived by
\begin{equation*}
\begin{split}
\ti \sQ_{\a\b}\p_t^\a \p_t^\b&= (\p_t \varphi)^2+\f12q \varphi^2+\f12 \bg^{\rho \sigma}\p_\rho \varphi \p_\sigma\varphi+H_0^\ga\p_\ga \varphi \p_t \varphi\\
&=\f12\{-\bg^{00}(\p_t\varphi)^2+\bg^{ij}\p_i \varphi\p_j \varphi+q\varphi^2\},
\end{split}
\end{equation*}
where all other terms have been cancelled. This gives (\ref{4.24.2.18}).
\end{proof}
With the help of Lemma \ref{6.10.6.18}, we give the fundamental energy estimates.
\begin{proposition}[Energy inequality]\label{4.29.4.18}
Suppose there hold  on $\I$  the assumptions (\ref{6.10.4.18}), (\ref{4.25.1.18}), (\ref{4.25.2.18})  and
\begin{equation}\label{6.10.3.18}\tag{A2}
\ti C M^{-1}\dn^\f12 (\frac{R}{2})^{\f12-\frac{\ga_0}{2}}\le \frac{1}{6}
\end{equation}
with the universal constant $\ti C\ge 1$ specified in the proof.

 Let  $(u_2, \ub_2)\in \I$. For $ (u, \ub)\in \I$ with $ u\le u_2$,
there holds the following energy inequality for any constant $p\le 0$
\begin{align*}
 u_+^{-2p+\ga_0}&(E[\varphi](\H_u^{\ub})+ E[\varphi](\Hb_\ub^{u})) \\
&\quad\quad\les \sup_{(u_1,\ub_1)\in \I, u_1\le u_2}\{{u_1}_+^{-2p+\ga_0} E[\varphi](\Sigma_0^{u_1,\ub_1})+{u_1}_+^{-2p+\ga_0}(\|r\sF^\flat\|_{L_\ub^1 L_u^2 L_\omega^2(\D_{u_1}^{\ub_1})}^2+\\
& \quad \quad+\|r\sF^\flat\|^2_{L_u^1 L_\ub^2 L_\omega^2(\D_{u_1}^{\ub_1})})+ {u_1}_+^{1-2p+\ga_0} M^{-1}\| r^\f12 \sF^\sharp\|_{L^2(\D_{u_1}^{\ub_1})}^2\},
\end{align*}
where $\widetilde{\Box}_\bg \varphi-q\varphi=\sF^\flat+\sF^\sharp$.
\end{proposition}
\begin{proof} For convenience, we denote $\sF= \widetilde{\Box}_\bg \varphi-q\varphi$.
We first show that, in $\D_u^{\ub}$ with $(u, \ub)\in \I$ and $u\le u_2$, there holds
\begin{align}
E[\varphi](\H_{u}^{\ub})+E[\varphi](\Hb_{\ub}^{u})&\les E[\varphi](\Sigma_0^{u, \ub})+\int_{\D_{u}^{\ub}}|(\sF+\p^\a\tensor{H}{_\a^\ga} \p_\ga \varphi) \p_t \varphi\nn\\
&-\f12 \p_t \tensor{H}{_\a^\ga}\p_\ga \varphi \p^\a \varphi| .\label{6.10.2.18}
\end{align}
If $X=\p_t$ in  (\ref{4.10.1.18}), the last two terms  vanish due to  $\p^\a {\p_t}^\b=0 $. By divergence theorem, we have
\begin{align*}
\int_{\D_u^\ub} \p^\a (\ti\sQ_{\a\b}\p_t^\b)=\int_{\Sigma_0^{u, \ub}} \ti\sQ_{\a\b}\p_t^\b \p_t^\a-(\int_{\H_u^\ub} \ti\sQ_{\a\b}\p_t^\b \N^\a+\int_{\Hb_{\ub}^u}\ti \sQ_{\a\b}\p_t^\b \Nb^\a).
\end{align*}
Here we recall that the area element on $\D_u^\ub$ is $\f12\bb du d\ub d\omega$, $d\mu_\H=\f12\bb d\ub d\omega$ and $d\mu_{\Hb}=\f12 \bb du d\omega$ are the standard area elements on $\H_u$ and $\Hb_\ub$.  We may drop the standard area elements for the integral on the corresponding hypersurfaces or domains for convenience.

 Combining the above identity with Lemma \ref{6.10.7.18} and Lemma \ref{6.10.6.18} implies
\begin{align*}
&\int_{\Hb_\ub^u}\f12\{(\Lb\varphi)^2+(L \varphi)^2(h-H^{LL})+(1+h)(|\sn \varphi|^2+q\varphi^2)\\
&+(H+hH) \ud\p\varphi \p \varphi\}(1+h)^{-1}d\mu_{\Hb}\\
&+\int_{\H_u^\ub} \f12 \{(L \varphi)^2+(\Lb \varphi)^2(h-H^{\Lb \Lb})+(1+h)(|\sn \varphi|^2+q\varphi^2)\\
&+(H+hH) \p\varphi \bar \p \varphi\}(1+h)^{-1}d\mu_\H\\
&=\f12\int_{\Sigma_0^{u,\ub}}(-\bg^{00}(\p_t\varphi)^2+\bg^{ij} \p_i \varphi \p_j \varphi+q\varphi^2) dx-\int_{\D_u^\ub} \{(\sF +\p^\a \tensor{H}{_\a^\ga} \p_\ga \varphi)\p_\b \varphi\\
&-\f12 \p_\b\tensor{H}{_\a^\ga} \p_\ga \varphi \p^\a \varphi\}\p_t^\b.
\end{align*}

We note that by using (\ref{4.25.3.18}) and the smallness of $|h|$,
\begin{align*}
&\int_{\Hb_\ub^u}|(H+hH)\ud \p \varphi \p\varphi| (1+h)^{-1}d\mu_{\Hb}\les\int_{\Hb_\ub^u} |H |(|\ud\p\varphi\bar\p \varphi|+|\ud\p \varphi|^2) d\mu_{\Hb} \\
&\quad \quad \les\dn^\f12\|\ud \p \varphi\|_{L^2(\Hb_\ub^u)} (M^{-\f12} u_+^{-\frac{\ga_0}{2}}\|M^\f12 r^{-\f12} \bar\p \varphi\|_{L^2(\Hb_\ub^u)}+ u_+^{-\frac{\ga_0}{2}-\f12} \|\ud \p\varphi\|_{L^2(\Hb_\ub^u)}) \\
&\quad \quad \les \dn^\f12 (M^{-\f12}+u_+^{-\f12} \dn^\f12)u_+^{-\frac{\ga_0}{2}} E[\varphi](\Hb_\ub^u).
\end{align*}
Similarly,
\begin{equation*}
\int_{\H_u^\ub}|(H+hH) \p \varphi \bar \p \varphi| (1+h)^{-1} d\mu_\H\les\dn^\f12 (M^{-\f12}+u_+^{-\f12} \dn^\f12) u_+^{-\frac{\ga_0}{2}} E[\varphi](\H_u^\ub).
\end{equation*}
Since $\dn^\f12 M^{-\f12}\les M^\f12 $, the coefficient can be sufficiently small. Due to (\ref{6.10.4.18}), this pair of error  terms will be absorbed by the leading positive terms on the lefthand side.

With the help of  (\ref{7.23.4.18}),  we can derive
\begin{align*}
\int_{\Sigma_0^{u, \ub}}(-\bg^{00} (\p_t \varphi)^2+\bg^{ij} \p_i \varphi \p_j \varphi+q\varphi^2) dx\les E[\varphi](\Sigma_0^{u, \ub}).
\end{align*}

Thus
\begin{equation}\label{7.11.1.18}
\begin{split}
\int_{\H_u^\ub}& \big((L \varphi)^2+|\sn \varphi|^2+q\varphi^2 +\frac{M}{r}(\Lb \varphi)^2\big) +\int_{\Hb_\ub^u} \big((\Lb \varphi)^2 +\frac{M}{r}(L \varphi)^2 +|\sn \varphi|^2+q\varphi^2\big)\\
&\le \C\big( E[\varphi](\Sigma_0^{u, \ub}) +\int_{\D_u^\ub}|(\sF+\p^\a\tensor{H}{_\a^\ga} \p_\ga \varphi) \p_t \varphi-\f12 \p_t \tensor{H}{_\a^\ga}\p_\ga \varphi \p^\a \varphi |\big).
\end{split}
\end{equation}
Thus (\ref{6.10.2.18}) is proved.

We remark that
\begin{equation}\label{6.2.2.18}
\Tr[\varphi]=\p^\a H_\a^\ga \p _\ga \varphi \c \p_t \varphi-\f12 \p_t H_\a^ \ga \p_\ga \varphi \p^\a \varphi= \frac{1}{4} \Lb H^{\Lb \Lb }(\Lb \varphi)^2 +\bar{\Tr}(\p H, \p \varphi, \p \varphi)
\end{equation}
where the trilinear term  $\bar{\Tr}$ means that, in the product of the three terms, at least one of the derivatives $\p$ is  $\bar \p$. \begin{footnote}{In Section \ref{Eins}, we will take advantage of this structure when the wave coordinates condition is available.}\end{footnote}

Symbolically, $\Tr[\varphi]=\p H\c \p\varphi\c \p \varphi$. Let $\a=-p+\frac{\ga_0}{2}$ with the constant $p\le 0$. We calculate for any $(u_2, \ub_2)\in \I$ by using (\ref{7.23.3.18})
\begin{align*}
&{u_2}_+^{2\a} \int_{\D_{u_2}^{\ub_2}} |\p H\c \p \varphi\c \p \varphi| \\
&\les {u_2}_+^{2\a} \|r^\f12 \p H\|_{L_\ub^2 L^\infty(\D_{u_2}^{\ub_2})}(\|r \ud \p \varphi\|_{L_\ub^\infty L_u^2 L_\omega^2(\D_{u_2}^{\ub_2})}+\| r\bar \p \varphi\|_{L_u^\infty L_\ub^2 L_\omega^2(\D_{u_2}^{\ub_2})}) \| r^{-\f12} \p \varphi\|_{L^2(\D_{u_2}^{\ub_2})}\\
&\les {u_2}_+^{\a}\|r^\f12 \p H\|_{L_\ub^2 L^\infty(\D_{u_2}^{\ub_2})}M^{-\f12}(\int_{-\ub_2}^{ u_2} E[\varphi](\H_u^{\ub_2}) du)^\f12(\sup_{-\ub_2\le \ub\le u_2} {u_2}_+^{\a}E[\varphi]^\f12 (\Hb_{\ub}^{u_2})\\
&+\sup_{-\ub_2\le u\le u_2}u_+^\a E[\varphi]^\f12(\H_u^{\ub_2}))\\
&\les  M^{-1}\dn^\f12 {u_2}_+^{\f12-\frac{\ga_0}{2}}(\sup_{-\ub_2\le u\le u_2} u_+^{2\a} E [\varphi](\H_u^{\ub_2})+\sup_{-u_2\le \ub\le \ub_2}E[\varphi](\Hb_\ub^{u_2}) {u_2}_+^{2\a}).
\end{align*}
To treat the integral of $\sF\c \p_t \varphi$ in (\ref{6.10.2.18}), we repeat the derivation in  (\ref{5.16.3.18}).  Thus with $\sF=\sF^\flat+\sF^\sharp$,
\begin{align*}
{u_2}_+^{-2p+\ga_0}&\int_{\D_{u_2}^{\ub_2}} (|\sF\c \p_t \varphi|+|\p H \c\p \varphi\c \p \varphi|) \nn\\
&\le C (\ep_1) {u_2}_+^{-2p+\ga_0}(\|r\sF^\flat\|_{L_\ub^1 L_u^2 L_\omega^2(\D_{u_2}^{\ub_2})}^2 + \| r \sF^\flat\|_{L_u^1 L_\ub^2 L_\omega^2(\D_{u_2}^{\ub_2})}^2\\
&+M^{-1} {u_2}_+ \|r^\f12 F^\sharp\|^2_{L^2(\D_{u_2}^{\ub_2})})+(\ep_1+ \ti \C M^{-1}\dn^\f12 {u_2}_+^{\f12-\frac{\ga_0}{2}})\\
&\cdot\big( \sup_{-\ub_2\le u\le u_2}  u_+^{-2p+ \ga_0} E[\varphi](\H_u^{\ub_2})+ \sup_{-u_2\le \ub\le \ub_2}{u_2}_+^{-2p+\ga_0} E[\varphi](\Hb_\ub^{u_2})\big).
\end{align*}
We fix the constant in (\ref{6.10.3.18})  by $\ti C=\C\c \ti \C$.
By using (\ref{6.10.3.18}) and choosing $\ep_1$ sufficiently small, the last term can be absorbed. Substituting  the above estimate to (\ref{7.11.1.18}) implies   Proposition \ref{4.29.4.18}.
\end{proof}

Next we establish the $r$-weighted energy inequality by first  giving the following result.
\begin{lemma}\label{4.15.1.18}
Let $ -\ub_*\le -\ub_1\le u_1\le u_0$. Under the assumptions of (\ref{4.25.1.18}), (\ref{4.25.2.18}) and (\ref{6.10.4.18}), there  holds the following estimate for any constant $p\le 0$
\begin{equation}\label{5.5.2.18}
\begin{split}
\int_{\D_{u_1}^{\ub_1}}&|(\widetilde{\Box}_\bg \varphi-q\varphi)(X\varphi+\varphi)+\f12 (r^{-2} (L(r\varphi))^2+|\sn \varphi|^2)-(\p^\a {}\rp{X}\sP_\a+\f12(Xq+q) \varphi^2) |\\
&\les\dn ^\f12 M^{-\f12}\{{u_1}_+^{1-\frac{3\ga_0}{2}+2p} (\sup_{-\ub_1\le u\le u_1}u_+^{\ga_0-2p} E[\varphi](\H_u^{\ub_1})+{u_1}_+^{\ga_0-2p}\sup_{-u_1\le \ub\le \ub_1} E[\varphi](\Hb_{\ub}^{u_1}))\\
&+\int_{-\ub_1}^{u_1} u_+^{-\frac{\ga_0+1}{2}} \W_1[\varphi](\H_u^{\ub_1})du\},
\end{split}
\end{equation}
where
the energy current is defined by
\begin{equation}\label{4.14.1.18}
{}\rp{X}\sP_\a =\ti \sQ_{\a\b} X^\b +\f12 \bm_{\a \a'} \bg^{\a'\ga} \p_\ga (\varphi^2)+Y_\a
\end{equation}
with
\begin{equation*}
X= r(L-H^{\Lb\Lb} \Lb), \quad Y=\frac{1}{2} r^{-1} \varphi^2 L.
\end{equation*}
\end{lemma}
The proof of (\ref{5.5.2.18}) relies on an important cancelation thanks to the construction of $\sQ_{\a\b}$ and the choice of the multiplier $X$. The cancelation will be seen in the following proof. For convenience, we denote by $\sP$ the current ${}\rp{X}\sP$ from now on.

\begin{proof}
The proof is based on the following identity  on $\D_{u_1}^{\ub_1}$ .
\begin{equation}\label{7.11.2.18}
\begin{split}
\p^\a  \sP_\a +\f12(Xq+q)\varphi^2&=(\widetilde{\Box}_\bg\varphi-q\varphi)(X\varphi+\varphi)+\f12 (r^{-2} (L(r\varphi))^2+|\sn \varphi|^2)\\
&+I+II+III
\end{split}
\end{equation}
where the error terms $I$, $II$ and $III$ are
\begin{align}
&I=\p^\a \tensor{H}{_\a^\ga}\p_\ga \varphi( X \varphi+\varphi)-\f12 X\tensor{H}{_\a^\ga}\p_\ga \varphi\c \p^\a \varphi+\f12 r L H^{\Lb \Lb}(\Lb \varphi)^2,\nn\\
&II=\f12 \Lb(r H^{\Lb\Lb}) |\sn \varphi|^2+H \p \varphi \bar \p \varphi,\nn\\
&III=\tensor{H}{_\a^\ga} \p_\ga \varphi \p_\b \varphi \p^\a X^\b+(H +r\p H) \big(H^{\rho \sigma}\p_\rho\varphi \p_\sigma\varphi+q\varphi^2).\label{7.27.8.18}
\end{align}
To  show (\ref{7.11.2.18}), by using (\ref{4.10.1.18}), we first derive
\begin{equation}\label{7.23.1.18}
\begin{split}
\p^\a \sP_\a&=X^\b\{(\widetilde{\Box}_\bg \varphi-q\varphi+\p^\a H_\a^\ga \p_\ga \varphi)\c \p_\b \varphi-\f12 \p_\b q\varphi^2\\
&-\f12\p_\b H_\a^\ga \p_\ga \varphi \c \p^\a \varphi\}+H_\a^\ga \p_\ga \varphi \p_\b \varphi \p^\a X^\b+\sE
\end{split}
\end{equation}
where
$$\sE=\sQ_{\a\b}\p^\a X^\b+\p^\a(\f12 \bm_{\a\a'} \bg^{\a'\ga} \p_\ga (\varphi^2)+Y_\a).$$
It is straightforward to check that
\begin{eqnarray*}
{}\rp{X}\pi_{AB}=\delta_{AB}(1+H^{\Lb\Lb})&& {}\rp{X} \pi_{L \Lb} =-1-H^{\Lb \Lb}+r \Lb H^{\Lb \Lb}\\
{}\rp{X}\pi_{\Lb\Lb}=2 && {}\rp{X}\pi_{LL}=2(H^{\Lb \Lb}+r L H^{\Lb \Lb}).
\end{eqnarray*}
It is easy to check in view of the definition of $\sQ_{\a\b}$ that
\begin{align*}
&\sQ_{L L}=(L\varphi)^2 \quad \quad\quad \quad \sQ_{\Lb \Lb}=(\Lb \varphi)^2 \\
&\sQ_{AB}=\sn_A \varphi \sn_B \varphi-\f12 \bm_{AB}(-L \varphi \Lb \varphi+|\sn \varphi|^2+q\varphi^2 +H^{\rho \sigma} \p_\rho \varphi \p_\sigma\varphi)\\
&\sQ_{L\Lb}=|\sn \varphi|^2+q\varphi^2+H^{\rho \sigma}\p_\rho \varphi \p_\sigma\varphi.
\end{align*}
By combining the lists of $\sQ_{\a\b}$ and ${}\rp{X}\pi^{\a\b}$, we derive
\begin{align*}
&\sQ_{LL}{}\rp{X} \pi^{LL}=\f12 (L \varphi)^2 \\
&\sQ_{\Lb \Lb} {}\rp{X} \pi^{\Lb \Lb}= \f12 (\Lb\varphi)^2 (H^{\Lb \Lb}+r L H^{\Lb \Lb})\\
&\sQ_{L \Lb}{}\rp{X} \pi^{L \Lb}=\frac{1}{4} (-1-H^{\Lb \Lb}+r\Lb H^{\Lb \Lb})(|\sn \varphi|^2+q\varphi^2+H^{\rho\sigma} \p_\rho \varphi \p_\sigma\varphi)\\
&\sQ_{AB} {}\rp{X}\pi^{AB}=(1+H^{\Lb \Lb})(\Lb \varphi L \varphi-q \varphi^2-H^{\rho \sigma} \p_\rho \varphi\p_\sigma \varphi).
\end{align*}
Note that
\begin{align*}
\f12 \p^\a(\bm_{\a\a'} \bg^{\a'\ga} \p_\ga (\varphi^2))&=\f12\p_\a H^{\a\ga} \p_\ga (\varphi^2) +\bg^{\a'\ga} \p_{\a'}\varphi \p_\ga \varphi+\widetilde{\Box}_\bg \varphi \c \varphi\\
&=\f12\p_\a H^{\a\ga} \p_\ga (\varphi^2)+ \p^\a\varphi\c \p_\a\varphi+H^{\rho \sigma}\p_\rho \varphi \p_\sigma \varphi+\widetilde{\Box}_\bg \varphi \c \varphi.
\end{align*}
Since $\p^\a L_\a=2/r$ and $\p^\a \Lb_\a=-2/r$, we have
\begin{align*}
\p^\a Y_\a &= \f12 \p^\a(r^{-1} \varphi^2L_\a )=\f12 \{\varphi^2 L(r^{-1})+r^{-1}L(\varphi^2)+r^{-1} \varphi^2\p^\a L_\a \}\\
&=\f12 (r^{-1} L (\varphi^2)+r^{-2} \varphi^2).
\end{align*}
Thus
\begin{align*}
\sE&=\widetilde{\Box}_\bg \varphi \c \varphi+\f12 (r^{-2}(L (r\varphi))^2+|\sn \varphi|^2+H^{\Lb \Lb }(\Lb \varphi)^2)+(-\f12-\frac{3}{2}H^{\Lb\Lb}\\
&+\f12 r \Lb H^{\Lb \Lb})H^{\rho\sigma} \p_\rho \varphi \p_\sigma\varphi +\f12 (\Lb \varphi)^2 r L H^{\Lb\Lb}+\f12\Lb(r H^{\Lb \Lb})|\sn\varphi|^2\\
&+H^{\Lb \Lb} \Lb \varphi L \varphi+q\varphi^2\big(-\frac{3}{2}(1+H^{\Lb \Lb})+\f12 r \Lb H^{\Lb \Lb}\big)+\f12\p_\a H^{\a\ga} \p_\ga (\varphi^2).
\end{align*}
Note that
\begin{equation*}
-H^{\Lb\Lb }(\Lb \varphi)^2+H^{\rho \sigma} \p_\rho \varphi \p_\sigma \varphi=H \bar\p \varphi \p \varphi.
\end{equation*}
Hence we conclude that
\begin{align*}
\sE&=\widetilde{\Box}_\bg \varphi \c \varphi+\f12\p_\a H^{\a\ga} \p_\ga (\varphi^2) +\f12 (r^{-2}(L (r\varphi))^2+|\sn \varphi|^2)+(H+ r \p H)(H^{\rho\sigma} \p_\rho \varphi \p_\sigma\varphi +q\varphi^2)\\
&+\f12  r L H^{\Lb \Lb}(\Lb \varphi)^2+\f12 \Lb(r H^{\Lb \Lb})|\sn\varphi|^2+ H \p\varphi \c\bar \p \varphi-\frac{3}{2}q\varphi^2.
\end{align*}
Substituting the above formula to (\ref{7.23.1.18}) yields (\ref{7.11.2.18}).

 Next we prove (\ref{5.5.2.18}) by controlling the error terms in the identity.  We claim
\begin{equation}\label{6.17.13.18}
\begin{split}
I&=r L H\p \varphi\c \bar\p \varphi+ \p H\c \p\varphi\c L(r\varphi)+r H\c \p H\c (\p \varphi)^2, \\
 II&=r \p H |\bar \p \varphi|^2+ H\p \varphi \bar\p \varphi.
 \end{split}
\end{equation}
Indeed,  the last term of $I$ cancels completely the bad component of the second term, which can be seen below
\begin{equation*}
\begin{split}
-\f12& X \tensor{H}{_\a^\ga} \p_\ga \varphi \p^\a \varphi+\f12 r L H^{\Lb \Lb }(\Lb \varphi)^2\\
&= \f12 r (-L H^{\a\ga} \p_\ga \varphi \p_\a \varphi+L H^{\Lb \Lb } (\Lb \varphi)^2)+\f12 r H^{\Lb \Lb} \Lb \tensor{H}{_\a^\ga} \p_\ga \varphi \p^\a \varphi\\
&=\f12 r L H \p \varphi \bar \p \varphi+\f12 r H \p H (\p \varphi)^2.
\end{split}
\end{equation*}

By direct calculation, we can obtain the symbolic formula for $I$.  The formula of $II$ is a simple recast.
It remains to consider the error term $III$ in (\ref{7.11.2.18}). For the first term, note that
\begin{align*}
\p^\a X^\b=\p^\a r (L -H^{\Lb\Lb} \Lb )^\b+r (\p^\a L^\b-\p^\a H^{\Lb \Lb} \Lb^\b  -H^{\Lb \Lb}\p^\a\Lb^\b)
\end{align*}
and the nontrivial components of  $\p^\a L^\b$ and $\p^\a \Lb^\b$ are $\p_A L_B=r^{-1} \delta_{AB}$ and $\p_A \Lb_B= -r^{-1} \delta_{AB}$.
 By a direct substitution, we can obtain
\begin{align*}
|\tensor{H}{_\a^\ga} \p_\ga \varphi\p_\b \varphi \p^\a X^\b|\les| H \p \varphi \c \bar\p\varphi|+ (r|\p H|+|H|) \c| H| (\p \varphi)^2
\end{align*}
and then
\begin{equation}\label{7.26.3.18}
|III|\les |H\p \varphi \c \bar \p \varphi|+(r|\p H|+|H|) \c( |H|(\p \varphi)^2+q\varphi^2).
\end{equation}
Next we prove the following error estimate
\begin{equation}\label{6.20.1.18}
\begin{split}
\int_{\D_{u_1}^{\ub_1}}|I|+|II|+|III|&\les \dn^\f12 M^{-\f12}{u_1}_+^{-\frac{3\ga_0}{2}+1} (\sup_{-\ub_1\le u\le u_1} E[\varphi](\H_u^{\ub_1}) u_+^{\ga_0}\\
&+ \sup_{-u_1\le \ub \le \ub_1}E[\varphi](\H_{\ub}^{u_1}){u_1}_+^{\ga_0})+ \dn^\f12 M^{-\f12}\int_{-\ub_1}^{u_1} {u}_+^{-\frac{\ga_0+1}{2}}\W_1[\varphi](\H_{u}^{\ub_1}) du.
\end{split}
\end{equation}

Let us consider the error terms $I$ in (\ref{6.17.13.18}). By using (\ref{4.25.3.18}) and (\ref{7.23.3.18}), we have
\begin{align*}
\int_{\D_{u_1}^{\ub_1}}& r |H|| \p H| (\ud\p \varphi)^2 \\
&\les \| r \ud \p \varphi\|_{L_\ub^\infty L_u^2 L_\omega^2(\D_{u_1}^{\ub_1})}\| r^{-\f12} \ud \p \varphi\|_{L^2(\D_{u_1}^{\ub_1})}\|r^\f12 \p H\|_{L_\ub^2 L^\infty(\D_{u_1}^{\ub_1})}\| r H\|_{L^\infty(\D_{u_1}^{\ub_1})}\\
&\les M^{-1} \dn {u_1}_+^{-2\ga_0+\f12} \sup_{-\ub_1\le u\le u_1} E^\f12[\varphi](\H_u^{\ub_1}) u_+^{\frac{\ga_0}{2}} \c \sup_{-u_1\le \ub \le \ub_1}E^\f12[\varphi](\H_{\ub}^{u_1}){u_1}_+^{\frac{\ga_0}{2}}.
\end{align*}
 Similarly, by using (\ref{4.25.3.18}) and (\ref{7.23.2.18}), we can obtain
\begin{equation}\label{7.28.4.18}
\begin{split}
\int_{\D_{u_1}^{\ub_1}}r  (|L H|+|H||\p H|) |\p \varphi \c \bar \p \varphi| &\les \int_{-\ub_1}^{u_1} \| \bar\p \varphi\|_{L^2(\H_u^{\ub_1})}\| r^{-\f12}\p \varphi\|_{L^2(\H_u^{\ub_1})} u_+^{-\frac{\ga_0}{2}}\dn^\f12 d u \\
&\les \dn^\f12 M^{-\f12}{u_1}_+^{-\frac{3\ga_0}{2}+1} \sup_{-\ub_1\le u\le u_1} E[\varphi](\H_u^{\ub_1}) u_+^{\ga_0}.
\end{split}
\end{equation}
By using (\ref{4.25.3.18}), we derive
\begin{align}
\int_{\D_{u_1}^{\ub_1}} | L (r\varphi)\c \p\varphi\c \p H | &\les \dn^\f12 \int_{-\ub_1}^{u_1} {u}_+^{-\frac{\ga_0+1}{2}}\|r^{-\f12}L (r\varphi )\|_{L^2(\H_{u}^{\ub_1})}  \|r^{-\f12}\p \varphi\|_{L^2(\H_{u}^{\ub_1})} du\nn\\
&\les \dn^\f12 M^{-\f12}\int_{-\ub_1}^{u_1} {u}_+^{-\frac{\ga_0+1}{2}}(\W_1[\varphi](\H_{u}^{\ub_1})+ E[\varphi](\H_{u}^{\ub_1})) du.\label{7.27.7.18}
\end{align}
Thus the error terms in $I$ are all  treated.

It again follows by using (\ref{4.25.3.18}) that
\begin{align*}
\int_{\D_{u_1}^{\ub_1}} r|\bar \p \varphi\c \bar \p\varphi\c \p H |  &\les\dn^\f12 \int_{-\ub_1}^{u_1} u_+^{\frac{-\ga_0-1}{2}} \|\bar \p \varphi \|_{L^2(\H_{u}^{\ub_1})}^2 du\\
&\les\dn^\f12\int_{-\ub_1}^{u_1} u_+^{-\frac{\ga_0+1}{2}} E[\varphi](\H_u^{\ub_1}) du \\
&\les \dn^\f12 {u_1}_+^{-\frac{3\ga_0}{2}+\f12} \sup_{-\ub_1\le u\le u_1} u_+^{\ga_0} E[\varphi](\H_u^{\ub_1}),
\end{align*}
\begin{align*}
\int_{\D_{u_1}^{\ub_1}}  |H ||\p \varphi|| \bar \p \varphi|
  &\les \dn^\f12\int_{-\ub_1}^{u_1} u_+^{-\frac{\ga_0}{2}} \|\bar\p \varphi\|_{L^2(\H_u^{\ub_1})}\|r^{-\f12} \p \varphi\|_{L^2(\H_u^{\ub_1})} du\\
&\les \dn^\f12 M^{-\f12} {u_1}_+^{-\frac{3\ga_0}{2}+1}\sup_{-\ub_1\le u\le u_1}u_+^{\ga_0} E[\varphi](\H_u^{\ub_1}).
\end{align*}
Thus the terms of  $II$ are all treated.

Next we estimate $III$ in view of (\ref{7.26.3.18}).
The first term  on the right of (\ref{7.26.3.18}) has been estimated in $II$. Note that  by using (\ref{4.25.3.18}),
\begin{equation*}
r|\p H|+ \frac{r}{u_+}|H| \les \dn^\f12 u_+^{-\frac{\ga_0+1}{2}}, \qquad r ( r|\p H ||H|+ |H|^2  ) \les \dn u_+^{-\ga_0}
\end{equation*}
which, in view of $\ga_0>1$, imply
\begin{align*}
\int_{\D_{u_1}^{\ub_1}} (r|\p H| +|H|) (|H|(\p \varphi)^2+q\varphi^2)&\les \dn \int_{-\ub_1}^{u_1} u_+^{-\ga_0} \|r^{-\f12} \p \varphi\|^2_{L^2(\H_u^{\ub_1})}du +\dn^\f12 \int_{\D_{u_1}^{\ub_1}} u_+^{-\frac{\ga_0+1}{2}} q\varphi^2 \\
&\les (\dn M^{-1} +\dn^\f12)u_+^{\f12-\frac{3}{2}\ga_0}\sup_{-\ub_1\le u\le u_1} E[\varphi](\H_u^{\ub_1}) u_+^{\ga_0}.
\end{align*}
 By combining the estimates for $I$, $II$ and $III$,  (\ref{6.20.1.18}) is proved since $\dn M^{-1}\le (\dn M^{-1})^ \f12$. Thus we can  conclude the inequality in (\ref{5.5.2.18}).
\end{proof}

Next, we give the $r$-weighted energy estimate.
\begin{proposition}\label{4.29.5.18}
 Under the assumptions of (\ref{4.25.1.18}), (\ref{4.25.2.18}) and (\ref{6.10.4.18}), with $-\ub_*\le -\ub_1\le u_1\le u_0, $
  there holds the following weighted energy estimate for $\widetilde{\Box}_\bg \varphi-q \varphi=\sF$,
  \begin{align*}
  &\sup_{-\ub_1 \le u\le u_1} {u}_+^{\ga_0-1-2p}\big(\W_1[\varphi](\D_{u}^{\ub_1})+ \W_1[\varphi](\H_{u}^{\ub_1})+\W_1[\varphi](\Hb_{\ub_1}^{u})\big)\\
&\les \sup_{-\ub_1 \le u\le u_1} {u}_+^{\ga_0-1-2p}\|r^\frac{3}{2} \sF\|_{L_u^1 L_\ub^2L_\omega^2(\D_{u}^{\ub_1})}^2   + \sup_{-\ub_1\le -\ub\le u\le u_1}\{ u_+^{\ga_0-2p}  (E[\varphi](\H_u^\ub)+ E[\varphi](\Hb_\ub^u))\} \\
&+\sup_{-\ub_1\le u\le u_1} u_+^{\ga_0-1-2p}\big(\| r^\f12 |\p \varphi|+r^{-\f12}|\varphi|+q^\f12_0|r^\f12 \varphi|\|^2_{L^2(\Sigma_0^{u,\ub_1})}+\int_{S_{u, -u}} r \varphi^2 d\omega \big),
  \end{align*}
where $p\le 0$ is any fixed constant.
\end{proposition}
\begin{proof}
With the help of  Lemma \ref{4.15.1.18}, we will apply divergence theorem to $\p^\a \sP_\a$ in $\D_{u_1}^{\ub_1}$.  We first  confirm the boundary terms give the desired weighted energy.

 Let us  first compute $(1+h) \sP_\a \N^\a$. Recall $X=r(L-H^{\Lb \Lb}\Lb)$ and $\N$ from (\ref{7.11.3.18}).
\begin{align*}
&(1+h)  \ti \sQ_{\a\b} X^\b \N^\a\\
&=(L+h\Lb) \varphi X\varphi-\f12r (L+h\Lb, L-H^{\Lb \Lb} \Lb)(-\Lb \varphi L\varphi+|\sn \varphi|^2+H^{\rho \sigma} \p_\rho \varphi \p_\sigma \varphi+q\varphi^2)\\
&+r\tensor{H}{_\a^\ga}\p_\ga \varphi \p_\b \varphi (L+h\Lb)^\a (L-H^{\Lb \Lb} \Lb )^\b\\
&=r\{ (L\varphi-H^{\Lb \Lb } \Lb \varphi)\big(L \varphi+h\Lb \varphi-2(H^{\Lb \ga}+h H^{L \ga}) \p_\ga \varphi\big)+(h-H^{\Lb \Lb})\\
&\cdot(-L \varphi\Lb\varphi+|\sn \varphi|^2+q\varphi^2 +H^{\rho \sigma} \p_\rho \varphi \p_\sigma \varphi)\}\\
&=r\{(L\varphi -H^{\Lb \Lb }\Lb \varphi)(L \varphi+(h-2(H^{\Lb \Lb}+h H^{L\Lb}))\Lb \varphi-2(H+h H) \bar\p \varphi)\\
 &+(h-H^{\Lb \Lb })(-\Lb\varphi L \varphi+H^{\Lb \Lb }(\Lb \varphi) ^2)+(h-H^{\Lb \Lb} )(|\sn \varphi|^2+q\varphi^2+ H  \p \varphi \bar \p \varphi)\}.
 \end{align*}
 We can cancel the first term in the last line by noting that
 $$L \varphi+(h-2(H^{\Lb \Lb}+h H^{L\Lb}))\Lb \varphi=(h-H^{\Lb \Lb })\Lb \varphi+L \varphi-H^{\Lb \Lb } \Lb \varphi-2 h H^{L \Lb}\Lb \varphi,$$
 where the first term on the right gives the cancelation after substitution.
 Hence
 \begin{align*}
& (1+h)  \ti \sQ_{\a\b} X^\b \N^\a\\
&=r\{(L\varphi -H^{\Lb \Lb} \Lb \varphi)^2+(h-H^{\Lb \Lb})(|\sn \varphi|^2+q\varphi^2)+(L\varphi-H^{\Lb \Lb }\Lb \varphi)\\
 &\cdot\big(-2hH^{\Lb L} \Lb \varphi+(H+hH) \bar\p\varphi\big)+(h-H^{\Lb \Lb}) H \p \varphi\bar \p \varphi\}.
\end{align*}
We remark that only the first term  on the righthand side is involved with the further cancelations with the next two identities.

Note also that
\begin{align*}
\f12(1+h) \N^\a  \bm_{\a \a'}\bg^{\a'\ga} \p_\ga (\varphi^2)&=\f12 \bm_{\a \a'} (H^{\a'\ga}+\bm^{\a'\ga})\p_\ga (\varphi^2)(L^\a +h\Lb^\a)\\
&=\varphi\{(\tensor{H}{_L^\ga}+h \tensor{H}{_\Lb^\ga}) \p_\ga\varphi+(L \varphi+h \Lb \varphi) \},
\end{align*}
and
\begin{equation*}
(1+h) \N^\a Y_\a=-r^{-1} \varphi^2 h.
\end{equation*}
  In view of the definition of $\sP^\a$ in (\ref{4.14.1.18}) and the above three identities, we can derive that
 \begin{align}
&r^2  (1+h) \N^\a  \sP_\a \label{4.30.2.18}\\
&= r\{(L(r\varphi)-r H^{\Lb \Lb} \Lb \varphi)^2+r^2(h-H^{\Lb \Lb})(|\sn \varphi|^2+q\varphi^2)-r^{-1}\f12 L (r^2 \varphi^2)\nn\\
&+r[-2h H^{L \Lb} \varphi \Lb \varphi-2(H^{\Lb \ga}+h H^{L \ga}) \varphi \bar\p\varphi+r^{-1}h\varphi \Lb (r\varphi)]\nn\\
&+r^2\big[(L \varphi- H^{\Lb \Lb}\Lb \varphi) \big(-2h H^{\Lb \Lb} \Lb \varphi+(H+hH) \bar\p \varphi\big)+(h-H^{\Lb \Lb})H \p \varphi \bar \p \varphi\big]\}\nn.
 \end{align}
 Thus the calculation of the energy density on $\H_u^\ub$ is completed.

Next we consider the weighted energy on $\Hb_{\ub}^u$. We first compute
\begin{align}
&(1+h)\sQ_{\a\b}X^\b \Nb^\a\label{6.20.5.18}\\
&= r\{(\Lb \varphi+hL \varphi)(L \varphi-H^{\Lb \Lb} \Lb \varphi)-\f12 (L-H^{\Lb \Lb}\Lb , \Lb+hL)((H^{\rho\sigma}+\bm^{\rho \sigma})\p_\rho \varphi \p_\sigma\varphi)+q\varphi^2\}\nn\\
&=r\{\Lb \varphi L \varphi (1-hH^{\Lb \Lb})+ h(L \varphi)^2-H^{\Lb \Lb}(\Lb \varphi)^2-\f12(-2+2 hH^{\Lb \Lb})\nn\\
&\cdot\big(H^{LL}(L \varphi)^2+H^{\Lb \Lb} (\Lb \varphi)^2+H\ud\p\varphi \bar \p \varphi-\Lb \varphi L \varphi+|\sn \varphi|^2+q\varphi^2\big)\}\nn\\
&=r\{(L\varphi)^2 (h+(1-H^{\Lb \Lb}h) H^{LL})-(\Lb \varphi)^2 h {H^{\Lb \Lb}}^2+(1-H^{\Lb \Lb} h) (|\sn \varphi|^2+q\varphi^2)\nn\\
&+(1-H^{\Lb \Lb} h)H \ud\p\varphi \bar \p \varphi\}\nn
\end{align}
and
\begin{align}
&(1+h)\tensor{H}{_\a^\ga}\p_\ga \varphi \p_\b \varphi \Nb^\a X^\b\label{6.20.4.18}\\
&= r\{ \Lb \varphi L \varphi(-2H^{L \Lb }-2hH^{\Lb \Lb}+2H^{\Lb \Lb}(H^{LL}+hH^{L\Lb}))+2(\Lb \varphi)^2 H^{\Lb \Lb}(H^{L \Lb}+hH^{\Lb\Lb})\nn\\
&-2(L \varphi)^2(H^{LL}+h H^{\Lb L})+(H+hH) \ud\p \varphi \bar \p \varphi\}.\nn
\end{align}

Note that
\begin{align*}
\f12(1+h) \bm_{\a\a'}\bg^{\a' \ga} \p_\ga (\varphi^2) \Nb^\a&= \f12 \bm_{\a\a'}(H^{\a'\ga}+\bm^{\a'\ga}) \p_\ga (\varphi^2) (\Lb^\a +hL^\a)\\
&=\f12 \{(\Lb +h L)(\varphi^2)+(\tensor{H}{_\Lb^\ga}+h \tensor{H}{_L ^\ga}) \p_\ga(\varphi^2)\}
\end{align*}
and
\begin{equation*}
(1+h) \Nb^\a Y_\a=\f12 r^{-1} \varphi^2(L, \Lb+hL)=-r^{-1} \varphi^2.
\end{equation*}
We then combine the above calculations for the two terms to obtain
\begin{align}
&r^2(1+h) \Nb^\a (\f12 \bm_{\a\a'} \bg^{\a' \ga} \p_\ga (\varphi^2)+Y_\a)\label{6.20.3.18}\\
&=\f12 \{(\Lb +hL)(r^2 \varphi^2)+r^2 (\tensor{H}{_\Lb^\ga}+h \tensor{H}{_L^\ga})\p_\ga(\varphi^2)\}-r h\varphi^2.\nn
\end{align}
Thus, by combining (\ref{6.20.5.18})-(\ref{6.20.3.18}) and in view of the definition (\ref{4.14.1.18}), we can obtain
\begin{align}
&(1+h)r^2\Nb^\a \sP_\a\label{4.30.1.18}\\
&=r^3\{(L\varphi)^2(h-H^{LL}+h H^2+h H)+(\Lb \varphi)^2 (h H^2+HH)+(|\sn \varphi|^2+q\varphi^2)(1-H^{\Lb \Lb} h)\nn\\
&+H\ud \p \varphi \bar \p \varphi\}+\f12 \{(\Lb +hL)(r^2 \varphi^2)+r^2 (H+hH) \p (\varphi)^2\}-rh \varphi^2\nn,
\end{align}
where the term  $ H(1+h+h H )\ud\p \varphi\c \bar\p \varphi$ is simplified to be $H\ud \p \varphi\bar \p \varphi$ due to   $|h|+|H|\le 1$.

In the sequel, we will constantly use the fact that $|h|+|H|\le 1$ to shorten the symbolic formula.
Recall from (\ref{7.4.3.18}) for the area elements $d\mu_\H$ and $d\mu_{\Hb}$. By (\ref{4.30.2.18}), we have
\begin{align}
&I=\int_{\H_{u_1}^{\ub_1}}(1+h) \N^\a \sP_\a  (1+h)^{-1} d\mu_\H \nn\\
&=\int_{\H_{u_1}^{\ub_1}} \big( r (L(r\varphi)-r H^{\Lb \Lb } \Lb \varphi)^2+r^3 (h-H^{\Lb \Lb})( |\sn \varphi|^2+q\varphi^2)\nn\\
&-\f12 (L-h\Lb) (r^2 \varphi^2)\big)\frac{r}{2(r-M_0)} d\omega d\ub\nn+\Er_1,\\
\Er_1&=\int_{\H_{u_1}^{\ub_1}}\{r^2 (L (r\varphi )-r H^{\Lb \Lb} \Lb \varphi)\cdot (H\bar \p \varphi+ hH \Lb \varphi)+r^3 (h-H^{\Lb \Lb}) H \p \varphi \bar \p \varphi \nn\\
&+r^2\varphi H(\bar \p \varphi+h \ud \p \varphi) \}\frac{r}{2(r-M_0)}d\omega d\ub.\label{4.30.6.18}
\end{align}
On $\Hb_{\ub_1}^{u_1}$, by using (\ref{4.30.1.18}),  we have
\begin{align*}
II&=\int_{\Hb_{\ub_1}^{u_1}}(1+h) \Nb^\a \sP_\a  (1+h)^{-1} d\mu_{\Hb} \\
&=\int_{\Hb_{\ub_1}^{u_1}} \big(r^3[(L \varphi)^2(h-H^{LL}+hH)+(q\varphi^2+|\sn \varphi|^2)(1-Hh)+(\Lb \varphi)^2 HH +H \ud \p \varphi \bar \p \varphi]\\
&+\f12[ (\Lb -hL) (r^2 \varphi^2)+r^2 H \p (\varphi^2  )]+h L (r^2\varphi^2)-r h\varphi^2 \big)\frac{r}{2(r-M_0)} d\omega du\\
&=\int_{\Hb_{\ub_1}^{u_1}} r^3\big((L \varphi)^2 (h-H^{LL})+ |\sn \varphi|^2+q\varphi^2\big)+\f12 (\Lb -hL )(r^2 \varphi^2))\frac{r}{2(r-M_0)} d\omega du+\Er_2,
\end{align*}
where
\begin{align}
\Er_2&= \int_{\Hb_{\ub_1}^{u_1}}\{r^2 (H+h) \p (\varphi ^2)+r h \varphi^2+r^3H (h(L \varphi)^2+(\Lb\varphi)^2+\ud \p\varphi\c \bar \p \varphi+qh \varphi^2)\}\frac{r}{2(r-M_0)} d\omega du.\label{7.26.2.18}
\end{align}
In $I$ and $II$, the coefficients of leading terms are precise, while $\Er_1$ and $\Er_2$ are symbolic formulas for the error terms.
Note that applying  divergence theorem  to $\D_{u_1}^{\ub_1}$ implies
\begin{equation*}
I+II-\int_{\Sigma_0^{u_1, \ub_1}} \sP_\a {\p_t}^\a dx =-\int_{\D_{u_1}^{\ub_1}}\p^\a \sP_\a.
\end{equation*}
For convenience, we set
\begin{align*}
&\ti I= I+\int_{\H_{u_1}^{\ub_1}}\f12 (L-h\Lb) (r^2 \varphi^2)\frac{r}{2(r-M_0)}d\omega d\ub;  \\
&\widetilde {II}=II-\int_{\Hb_{\ub_1}^{u_1}} \f12 (\Lb -hL )(r^2 \varphi^2)\frac{r}{2(r-M_0)} d\omega du;\\
&\widetilde{III}=\int_{\Sigma_0^{u_1, \ub_1}} \{\sP_\a \p_t^\a r^2+\f12\p_r (r^2 \varphi^2) \} d\omega dr.
\end{align*}
 Then, in view of (\ref{6.21.1.18}) and (\ref{4.30.3.18}), we have
\begin{equation*}
\ti I+\widetilde{II}-\widetilde{III}=-\int_{\D_{u_1}^{\ub_1}}\p^\a \sP_\a.
\end{equation*}
Now combining (\ref{5.5.2.18}) with the above identity, in view of  the definitions of $\ti I, \widetilde{II}$, we derive
\begin{equation}\label{6.22.1.18}
\begin{split}
&\int_{\D_{u_1}^{\ub_1}} \f12 (r^{-2}| L(r\varphi)|^2+|\sn \varphi|^2)\\
&+ \int_{\H_{u_1}^{\ub_1}} \{r(L(r \varphi)- r H^{\Lb \Lb}\Lb \varphi)^2+r^3 (h-H^{\Lb \Lb}) (|\sn \varphi|^2+q\varphi^2)\} \frac{r}{2(r-M_0)} d\omega d\ub\\
&+\int_{\Hb_{\ub_1}^{u_1}} r^3 \big((L\varphi)^2(h-H^{\Lb \Lb})+|\sn \varphi|^2+q\varphi^2) \frac{r}{2(r-M_0)} d\omega du\\
&\le |\widetilde{III}|+|\Er_1|+|\Er_2|+\int_{\D_{u_1}^{\ub_1}} |\sF (X\varphi+\varphi)| +\frac{q+|Xq|}{2} \varphi^2   \\
&+C \dn ^\f12 M^{-\f12}\{{u_1}_+^{1-\frac{3\ga_0}{2}+2p} (\sup_{-\ub_1\le u\le u_1} E[\varphi](\H_u^{\ub_1})u_+^{\ga_0-2p}+{u_1}_+^{\ga_0-2p}\sup_{-u_1\le \ub\le \ub_1} E[\varphi](\Hb_{\ub}^{u_1}))\\
&\int_{-\ub_1}^{u_1} u_+^{-\frac{\ga_0+1}{2}} \W_1[\varphi](\H_u^{\ub_1})du\}.
\end{split}
\end{equation}
Moreover, by using (\ref{4.25.3.18}), it is straightforward to obtain
\begin{equation*}
\int_{\H_{u_1}^{\ub_1}} r^3 |H|^2 (\Lb  \varphi)^2 d\ub d\omega \les E[\varphi](\H_ {u_1}^{\ub_1}) {u_1}_+^{-\ga_0+1} \dn M^{-1},
\end{equation*}
which leads to
\begin{equation}\label{5.4.1.18}
\begin{split}
|\int_{\H_{u_1}^{\ub_1}}& r (L (r\varphi)-r H^{\Lb \Lb} \Lb \varphi)^2\frac{r}{2(r-M_0)}d\omega d\ub -\int_{\H_{u_1}^{\ub_1}}r |L (r\varphi)|^2\frac{r}{2(r-M_0)} d\omega d\ub|\\
&\les E[\varphi](\H_{u_1}^{\ub_1}) (u_1)_+^{1-\ga_0} \dn M^{-1}.
\end{split}
\end{equation}
Also by using (\ref{6.10.4.18}), we can derive
\begin{equation}\label{6.22.2.18}
\begin{split}
&\int_{\D_{u_1}^{\ub_1}} \f12 (r^{-2}| L(r\varphi)|^2+|\sn \varphi|^2)+ \int_{\H_{u_1}^{\ub_1}}\{ r(L(r \varphi))^2+\frac{1}{3}r^2 M (|\sn \varphi|^2+q \varphi^2)\} \frac{r}{2(r-M_0)} d\omega d\ub\\
&+\int_{\Hb_{\ub_1}^{u_1}}r^3( \frac{M}{3r} (L\varphi)^2  +|\sn \varphi|^2+q\varphi^2) \frac{r}{2(r-M_0)} d\omega du\\
&\le |\widetilde{III}|+|\Er_1|+|\Er_2|+\int_{\D_{u_1}^{\ub_1}} |\sF (X\varphi+\varphi)| +\frac{q+|Xq|}{2} \varphi^2   \\
&+C \dn ^\f12 M^{-\f12}\{{u_1}_+^{1-\frac{3\ga_0}{2}+2p} (\sup_{-\ub_1\le u\le u_1} E[\varphi](\H_u^{\ub_1})u_+^{\ga_0-2p}+{u_1}_+^{\ga_0-2p}\sup_{-u_1\le \ub\le \ub_1} E[\varphi](\Hb_{\ub}^{u_1}))\\
&+\int_{-\ub_1}^{u_1} u_+^{-\frac{\ga_0+1}{2}} \W_1[\varphi](\H_u^{\ub_1})du\}.
\end{split}
\end{equation}
It is straightforward to derive
\begin{equation*}
r^2\sP_\a \p_t^\a+\f12 \p_r(r^2 \phi^2) =\f12\big( r(L(r\varphi))^2+ r^3(|\sn \varphi|^2+q\varphi^2) )+r^3 H\big((\p \varphi)^2+q\varphi^2+\p
\varphi\c \varphi \big).
\end{equation*}
Note that  with the help of (\ref{7.23.4.18}),
\begin{equation}\label{5.4.3.18}
|\widetilde{III}|
 \les \|r^\f12 \p\varphi\|^2_{L^2(\Sigma_0^{u_1, \ub_1})}+\|r^{-\f12}\varphi\|^2_{L^2(\Sigma_0^{u_1, \ub_1})}+q_0\|r^\f12 \varphi\|^2_{L^2(\Sigma_0^{u_1, \ub_1})}.
\end{equation}
It remains to estimate $\Er_1$ and $\Er_2$.
 We first  estimate $\Er_1$ in view of
\begin{align*}
&|r^\frac{3}{2}( H \bar\p \varphi + h H\Lb \varphi)|\les (|r h \Lb \varphi|  +r |\bar \p \varphi|)|r^\f12 H| \\
&r^3|(h-H)H \p \varphi \bar \p \varphi|\les |r \bar\p \varphi||r^\f12 \p \varphi| |r(h-H)||r^\f12 H| \\
&r^2|\varphi H (\bar \p \varphi+h \ud \p \varphi)|\les (|r \bar \p \varphi|+|r h \ud \p \varphi|)|\varphi| r|H|.
\end{align*}
By  H\"{o}lder's inequality, (\ref{4.25.3.18}) and  $r|h|\le M$,  we have
\begin{align*}
&\|r^\frac{3}{2} ( H \bar\p \varphi + h H\Lb \varphi)\|_{L^2_\omega L_\ub^2(\H_{u_1}^{\ub_1})}\les  {u_1}_+^{-\frac{\ga_0}{2}}  \dn^\f12 E[ \varphi]^\f12(\H_{u_1}^{\ub_1}),\\
& \int_{\H_{u_1}^{\ub_1}} r^3|(h-H)H \p \varphi \bar \p \varphi|  d\ub d\omega\les E[\varphi](\H_{u_1}^{\ub_1}) \dn^\f12 {u_1}_+^{-\frac{\ga_0}{2}} (M^{-\f12}\dn^\f12+M^\f12),\\
&\int_{\H_{u_1}^{\ub_1}} r^2 |\varphi H |(\bar \p \varphi+ |h \ud \p \varphi|)|d\ub d\omega\les E[\varphi]^\f12(\H_{u_1}^{\ub_1}) ( E[\varphi](\H_{u_1}^{\ub_1})+E[\varphi](\Hb_{\ub_1}^{u_1})\\
&\qquad\qquad\qquad\qquad\qquad+\int_{S_{u_1,-u_1}} r\varphi^2 d\omega+\int_{S_{-\ub_1, \ub_1}}r\varphi^2 d\omega )^\f12\dn^\f12 {u_1}_+^{-\frac{\ga_0-1}{2}},
\end{align*}
where we employed  (\ref{5.1.1.18}) to derive the last inequality.
Thus, noting that $\frac{r}{r-M_0}\les 1$ and $\dn M^{-1}<1$ we conclude that
\begin{equation}\label{6.22.3.18}
\begin{split}
|\Er_1|&\les {u_1}_+^{-\frac{\ga_0-1}{2}} \dn ^\f12\{({u_1}_+^{-1}\int_{\H_{u_1}^{\ub_1}} r(L(r\varphi)-rH^{\Lb\Lb} \Lb \varphi)^2  du d\omega)^\f12  E[\varphi]^\f12(\H_{u_1}^{\ub_1})\\
&+E[\varphi](\H_{u_1}^{\ub_1})+E[\varphi](\Hb_{\ub_1}^{u_1})+\sup_{-\ub_1\le u\le u_1}\int_{S_{u,-u}} r\varphi^2 d\omega \}\\
&\les {u_1}_+^{-\frac{\ga_0-1}{2}} \dn ^\f12\{{u_1}_+^{-1}\int_{\H_{u_1}^{\ub_1}} r(L(r\varphi))^2  du d\omega+E[\varphi](\H_{u_1}^{\ub_1})+E[\varphi](\H_{\ub_1}^{u_1})\\
&+\sup_{-\ub_1\le u\le u_1}\int_{S_{u, -u}} r\varphi^2 d\omega \},
\end{split}
\end{equation}
where we employed (\ref{5.4.1.18})  to derive the last inequality.  We remark that the term of $L(r\varphi)$ can be absorbed when $\Er_1$ is substituted back to (\ref{6.22.2.18}).

Next we control the error term in $\Er_2$ in a similar fashion. By using (\ref{4.25.3.18}) and $r|h|\le M$,  we derive
\begin{align*}
\int_{\Hb_{\ub_1}^{u_1}} r^3 (L \varphi)^2 |h H| d\omega du \les \dn^{\f12} {u_1}_+^{-\frac{\ga_0-1}{2}}\int_{\Hb_{\ub_1}^{u_1}} M r (L \varphi)^2 d\omega du\les \dn^{\f12} {u_1}_+^{-\frac{\ga_0-1}{2}}E[\varphi](\Hb_{\ub_1}^{u_1}),
\end{align*}
\begin{align*}
\int_{\Hb_{\ub_1}^{u_1}} r^3 |H|(|\Lb \varphi|^2+q\varphi^2)d \omega du &\les\dn ^\f12 {u_1}_+^{-\frac{\ga_0-1}{2}} E[\varphi](\Hb_{\ub_1}^{u_1}),\\
\int_{\Hb_{\ub_1}^{u_1}} r^3 |H \ud \p \varphi \bar \p \varphi| d\omega du&\les \dn^\f12 {u_1}_+^{-\frac{\ga_0-1}{2}} E[\varphi]^\f12 (\Hb_{\ub_1}^{u_1})\big(M^{-\f12}(\int_{\Hb_{\ub_1}^{u_1}}M r^2( L \varphi )^2  du d\omega)^\f12 \\
&\quad \quad \quad +E[\varphi]^\f12(\Hb_{\ub_1}^{u_1})\big),\\
\int_{\Hb_{\ub_1}^{u_1}} r^2 (|H |+|h|)\p (\varphi^2)  d\omega du &\les(\dn^{\f12} {u_1}_+^{-\frac{\ga_0-1}{2}}+M)\|r^{-1}\varphi\|_{L^2(\Hb_{\ub_1}^{u_1})}\\
&\times (\|\ud\p \varphi\|_{L^2(\Hb_{\ub_1}^{u_1})}+M^{-\f12}\|M^\f12 L \varphi\|_{L^2(\Hb_{\ub_1}^{u_1})}), \\
\int_{\Hb_{\ub_1}^{u_1}} r |h| \varphi^2  d\omega du&\le M\|r^{-1}\varphi\|^2_{L^2(\Hb_{\ub_1}^{u_1})}.
\end{align*}
Thus by combining the above estimates, and also by using (\ref{4.29.6.18}) to treat the term of $\|r^{-1}\varphi\|_{L^2(\Hb_{\ub_1}^{u_1})}$  we can derive
\begin{align*}
|\Er_2|& \les\big( \dn^\f12 M^{-\f12}{u_1}_+^{-\frac{\ga_0-1}{2}} +M^\f12\big) (\sup_{-\ub_1\le u\le u_1}\|r^{-\f12} \varphi\|^2_{L^2(S_{u,-u})} +E[\varphi](\Hb_{\ub_1}^{u_1})+\W_1[\varphi](\Hb^{u_1}_{\ub_1}))\\
&\les M^ \f12\big(\W_1[\varphi](\Hb_{\ub_1}^{u_1})+
E[\varphi](\Hb_{\ub_1}^{u_1})+\sup_{-\ub_1\le u\le u_1} \|r^{-\f12} \varphi\|^2_{L^2(S_{u,-u})}\big),
\end{align*}
where the first term on the righthand side of the last inequality will be absorbed due to the smallness of $M$.

We then substitute the estimates of $\Er_1$,  $\Er_2$, (\ref{6.22.2.18}) and (\ref{5.4.3.18}) to derive
\begin{align}
&\W_1[\varphi](\D_{u_1}^{\ub_1})+ \W_1[\varphi](\H_{u_1}^{\ub_1})+\W_1[\varphi](\Hb_{\ub_1}^{u_1})\nn\\
&\les \| |r^\f12 \p \varphi|+r^{-\f12} |\varphi|+q^\f12_0 r^\f12|\varphi|\|^2_{L^2(\Sigma_0^{u_1, \ub_1})}+\int_{\D_{u_1}^{\ub_1}} |\sF (X\varphi+\varphi)| +\frac{q+|Xq|}{2} \varphi^2   \nn\\
&+ M^\f12 \big(E[\varphi](\H_{u_1}^{\ub_1})+E[\varphi](\Hb_{\ub_1}^{u_1})+\sup_{-\ub_1\le u\le u_1} \int_{S_{u, -u}} r \varphi^2 d\omega \big)\nn\\
&+ \dn ^\f12 M^{-\f12}\{{u_1}_+^{1-\frac{3\ga_0}{2}+2p} (\sup_{-\ub_1\le u\le u_1} E[\varphi](\H_u^{\ub_1})u_+^{\ga_0-2p}\label{7.12.1.18}\\
&+{u_1}_+^{\ga_0-2p}\sup_{-u_1\le \ub\le \ub_1} E[\varphi](\Hb_{\ub}^{u_1}))+\int^{u_1}_{-\ub_1} u_+^{-\frac{\ga_0+1}{2}} \W_1[\varphi](\H_u^{\ub_1})du\}.\nn
\end{align}
Note that $\int_{\D_{u_1}^{\ub_1}} q \varphi^2\les \int_{-\ub_1}^{u_1}E[\varphi](\H_u^{\ub_1}) du$. The term $\int_{\D_{u_1}^{\ub_1}}|Xq| \varphi^2  dx dt$ can be treated exactly as in (\ref{6.23.3.18}).
The term of $\|r^{-\f12}(X\varphi+\varphi)\|_{L^2(\H_u^{\ub_1})}$ can treated  by applying (\ref{5.4.1.18}) on each $\H_u^{\ub_1}$, with the bound of error included in the first term in the line of (\ref{7.12.1.18}). By using Gronwall's inequality (see \cite[Section 2.3, Lemma 3]{mMaxwell}), the last term on the righthand side of the inequality and the first term on the right of (\ref{6.23.3.18}) can both be absorbed. We summarize the result after  the above treatments as below
\begin{align*}
&\W_1[\varphi](\D_{u_1}^{\ub_1})+ \W_1[\varphi](\H_{u_1}^{\ub_1})+\W_1[\varphi](\Hb_{\ub_1}^{u_1})\\
&\les C(\ep_1)\|r^\frac{3}{2} \sF\|_{L_u^1 L_u^2L_\omega^2(\D_{u_1}^{\ub_1})}^2 +{u_1}_+^{-\ga_0+2p+1}\ep_1 \sup_{-\ub_1\le u\le u_1} u_+^{\ga_0-2p-1}\W_1[\varphi](\H_u^{\ub_1})\\
&   +{u_1}_+^{1-\ga_0+2p} \sup_{-\ub_1\le u\le u_1}E[\varphi](\H_u^{\ub_1}) u_+^{\ga_0-2p}  +\| r^\f12 |\p \varphi|+r^{-\f12}|\varphi|\|^2_{L^2(\Sigma_0^{u_1, \ub_1})}+q_0\|r^\f12 \varphi\|^2_{L^2(\Sigma_0^{u_1,\ub_1})}\\
&+ (M^\f12+\dn^\f12 M^{-\f12}) \big({u_1}_+^{1-\frac{\ga_0}{2}}\sup_{-u_1\le \ub\le \ub_1}E[\varphi](\Hb_{\ub}^{u_1})+\sup_{-\ub_1\le u\le u_1} \int_{S_{u, -u}} r \varphi^2 d\omega \big),
\end{align*}
where $p\le 0$ is any fixed constant. We then multiply both sides by ${u_1}_+^{\ga_0-1-2p}$ followed with taking supremum on $u_1$ in $  -\ub_1\le u_1\le u_2\le u_0$ for a fixed $u_2$.  In view of the assumption that $\dn \les  M^2$, by choosing the constant $\ep_1>0$ sufficiently small, Proposition \ref{4.29.5.18} can be proved.
\end{proof}

\subsection{Error estimates}\label{error}
The main estimates of this part are the error estimates for controlling  $(\widetilde{\Box}_\bg-q) Z\rp{n}\phi$ in Proposition \ref{6.17.4.18}. Comparing this result with Proposition \ref{6.2.5.18} for the semilinear equation (\ref{3.18.1.18}), one difference lies in that Proposition \ref{6.17.4.18} copes with the term of $[\widetilde{\Box}_\bg, Z\rp{n}] \phi, \mbox{  with } n\le 3.$ Such error terms arise due to the nontrivial influence of the metric $\bg(\phi, \p \phi)$ and vanish in the semilinear case.  The treatment needs a sharp decay property for  the term $ZH$, particularly for $Z=\Omega_{ij}$, which requires us to bound energies for $Z\rp{3} H$.
Therefore in Proposition \ref{6.17.4.18} we treat $\str{n}\F$ \begin{footnote}{See the definition of $\str{n}\F$ in Lemma \ref{6.3.5.18}.}\end{footnote}   with $n\le 3$ for the solution of  quasilinear equation (\ref{eqn_1}), while in   Proposition \ref{6.2.5.18} we only need to bound the terms of $Z\rp{n}(\Box  \phi -q\phi)$ with $n\le 2$.

\begin{lemma}
Let $\Phi=(\phi, \p\phi)$. For $n\le 3$,
\begin{equation}\label{6.2.1.18}
|Z\rp{n}(H(\Phi))|\les \sum_{Z^a \sqcup Z^b=Z^n}|Z\rp{a}\Phi| u_+^{\zeta(Z^b)}.
\end{equation}
\end{lemma}
Note that  due to (\ref{6.7.1.18}) and (\ref{5.18.11.18}), (\ref{5.18.1.18}) holds for $\p \phi$ as well.  Thus, with the help of (\ref{5.2.5.18}), we can repeat the proof of  (\ref{5.17.3.18}) with $\N(\varphi)$ replaced by $H(\Phi)$ to obtain (\ref{6.2.1.18}).
\begin{proposition}\label{6.3.1.18}
\begin{equation}\label{7.6.1.18}
\begin{split}
|[\widetilde \Box_\bg, Z\rp{n}]\phi|&\les \sum_{Z^{n-1} \sqcup Z^1=Z^n}\big( u_+^{\zeta(Z^1)}|H|+| Z\rp{1} H|\big)  |\p^2 Z\rp{n-1} \phi|\\
&+\sum_{Z^a\sqcup Z^b\sqcup Z^c=Z^n}^{ b\le n-2}u_+^{\zeta(Z^c)} |Z\rp{a} \Phi|| \p^2 Z\rp{b}\phi|,\qquad 1\le n\le 3
\end{split}
\end{equation}
where the last term on the righthand side vanishes if $n=1$.
\end{proposition}
\begin{proof}
Since $Z\in \{\Omega, \p\}$ are killing vector fields,  $[\widetilde\Box_\bg, Z]=[H^{\mu\nu}\p_\mu\p_\nu, Z]$. We can derive
\begin{equation}\label{4.29.1.18}
[\widetilde{\Box}_{\bg}, Z] \psi=-Z H^{\a\b} \p_\a \p_\b \psi+2 H^{\a\b} \tensor{C}{_{Z \a}^\ga} \p_\ga \p_\b \psi,
\end{equation}
since
\begin{equation*}
H^{\a\b} [Z, \p_\a]  \p_\b \psi =H^{\a\b}\p_\a [Z, \p_\b] \psi= H^{\a\b}\tensor{C}{_Z_\a^\ga} \p_\ga \p_\b \psi,
\end{equation*}
where $\tensor{C}{_{Z \a}^\ga}$ has been defined in Lemma \ref{comm}. For convenience we will drop the coefficient $2$ in the calculation, and we will adopt
 the convention in Lemma \ref{comm}.

Similarly,
\begin{align*}
\widetilde{\Box}_\bg Z\rp{2} \psi
  &=Z\rp{2} \widetilde{\Box}_\bg \psi +Z_2[Z_1, \widetilde\Box_\bg] \psi+[Z_2, \widetilde{\Box}_{\bg}]Z_1 \psi\\
  &=Z\rp{2} \widetilde{\Box}_\bg \psi+Z_2[Z_1, H^{\a\b} \p_\a\p_\b ]\psi+[Z_2, H^{\a\b} \p_\a \p_\b]Z_1 \psi.
\end{align*}
By using (\ref{4.29.1.18}), we also give the following commutation identities
\begin{equation}\label{4.29.2.18}
\begin{split}
\widetilde\Box_\bg Z\rp{2} \psi&=Z\rp{2} \widetilde\Box_\bg \psi+\sum_{X\sqcup Y=Z^2}\{Z\rp{2} H^{\mu\nu} \p_\mu \p_\nu \psi+ X H^{\mu\nu} \p_\mu\p_\nu Y \psi+X H^{\mu\nu} (C_Y\c \p^2 \psi)_{\mu\nu}\\
&+H^{\mu\nu} \tensor{C}{_{X\mu}^\ga}\p_\ga\p_\nu Y\psi+H^{\mu\nu}( C_X C_Y \p^2\psi)_{\mu\nu}\},
\end{split}
\end{equation}
\begin{equation}\label{6.1.1.18}
\begin{split}
\widetilde{\Box}_\bg Z\rp{3}\psi&=Z\rp{3} \widetilde\Box_\bg\psi+ Z\rp{3} H^{\mu\nu} \p^2_{\mu\nu}\psi+ \sum_{a=1}^3 Z\rp{3\setminus a} H^{\mu\nu}\big(\p_\mu \p_\nu Z_a \phi+ (C_{Z_a} \c \p^2 \psi)_{\mu\nu}\big)\\
&+\sum_{a=1}^3 Z_a H^{\mu\nu} \left(\p_\mu \p_\nu Z\rp{3\setminus a} \psi +\sum_{X\sqcup Y=Z\rp{3\setminus a}}(C_X \c \p^2 Y\psi)_{\mu\nu}+(C_{Z\rp{3\setminus a}}\c \p^2 \psi)_{\mu\nu}\right)\\
&+H^{\mu\nu}\{\sum_{a=1}^3\big(\tensor{C}{_{Z_a}_\mu^\ga}  \p_\ga\p_\nu Z\rp{3\setminus a} \psi+ (C_{Z\rp{3\setminus a}} \p^2 Z_a \psi)_{\mu\nu}\big) +(C_{Z^3} \p^2 \psi)_{\mu\nu}\}.
\end{split}
\end{equation}
We then  summarize  the terms in (\ref{4.29.1.18})-(\ref{6.1.1.18}) into
\begin{equation}\label{6.3.2.18}
\begin{split}
|[\widetilde \Box_\bg, Z\rp{n}]\phi|&\les \sum_{Z^1\sqcup Z^{n-1}=Z^n} (u_+^{\zeta(Z^1 )}|H| + | Z\rp{1} H|) |\p^2 Z\rp{n-1}\phi|\\
&+\sum_{Z^a\sqcup Z^b\sqcup Z^c=Z^n, b\le n-2}u_+^{\zeta(Z^c)} |Z\rp{a} H|| \p^2 Z\rp{b}\phi|,
\end{split}
\end{equation}
where the last term on the righthand side vanishes if $n=1$.
By using (\ref{6.2.1.18}) to treat the term $Z\rp{a}H$,  the last term on the righthand side of (\ref{6.3.2.18}) can be bounded by
$$\sum_{Z^{a_1}\sqcup Z^b\sqcup Z^{c_1}=Z^n, b\le n-2}u_+^{\zeta(Z^{c_1})} |Z\rp{a_1} \Phi|| \p^2 Z\rp{b}\phi|.$$

We then combine the above estimates to
 conclude Proposition \ref{6.3.1.18}.
\end{proof}

In view of Proposition \ref{6.3.1.18} and (2) in  Lemma \ref{6.3.5.18}, we will prove the following result.
\begin{proposition}\label{6.17.4.18}
For $0\le n\le 3$, there hold for $(u_1, \ub_1)\in \I$ that
\begin{align}
&{u_1}_+^{-\zeta(Z^n)+\f12\ga_0+\f12}\|r^\f12 [\widetilde{\Box}_\bg, Z\rp{n}] \phi\|_{L^2(\D_{u_1}^{\ub_1})}\les \dn M^{-\f12}{u_1}_+^{-\f12\ga_0+\f12},\label{6.5.8.18}\\
&{u_1}_+^{-\zeta(Z^n)+\f12\ga_0}\|r^\frac{3}{2} [\widetilde{\Box}_\bg, Z\rp{n}]\phi\|_{L_u^1 L_\ub^2 L_\omega^2(\D_{u_1}^{\ub_1})} \les \dn M^{-\f12} {u_1}_+^{-\f12\ga_0+\f12},\label{6.9.3.18}\\
&{u_1}_+^{-\zeta(Z^n)+\f12 \ga_0+\f12}\|r^\f12 \str{n}\F_\Q\|_{L^2(\D_{u_1}^{\ub_1})}\les \dn M^{-\f12} {u_1}_+^{-\f12\ga_0+\f12},\label{6.8.2.18}\\
&{u_1}_+^{-\zeta(Z^n)+\f12\ga_0}\|r^\frac{3}{2} \str{n}\F_\Q\|_{L_u^1 L_\ub^2 L_\omega^2(\D_{u_1}^{\ub_1})}\les \dn M^{-\frac{1}{2}} {u_1}_+^{-\f12\ga_0+\frac{1}{2}},\label{6.9.5.18}\\
&{u_1}_+^{-\zeta(Z^n)}\|r^\f12 \str{n}\F_\C\|_{L^2(\D_{u_1}^{\ub_1})}\les \dn^\frac{3}{2}M^{-\f12}  {u_1}_+^{-\frac{3\ga_0}{2}-\f12}, \label{6.7.2.18}\\
&{u_1}_+^{-\zeta(Z^n)}\|r^\frac{3}{2} \str{n}\F_\C\|_{L_u^1 L_\ub^2 L_\omega^2(\D_{u_1}^{\ub_1})}\les \dn^\frac{3}{2}{u_1}_+^{-\frac{3\ga_0}{2}-\f12}.\label{6.9.9.18}
\end{align}
\end{proposition}
\begin{proof}
If $n=0$, the commutator is identically $0$. Thus the corresponding  estimates  are trivially true.  Thus for the commutator estimates (\ref{6.5.8.18}) and (\ref{6.9.3.18}), we only need to consider the cases $1\le n\le 3$.
We first prove (\ref{6.5.8.18}).
Denote by $I_n$ and $II_n$ the terms on the right of (\ref{7.6.1.18}) respectively.

We  apply (\ref{4.25.3.18}) to $ Z\rp{i} H, i\le 1$ and bound $\p^2 Z\rp{n-1}\phi$ by using (\ref{7.5.4.18}). Thus,
\begin{align*}
\|r^\f12 I_n\|_{L^2(\D_{u_1}^{\ub_1})}&\les \sum_{Z^1\sqcup Z^{n-1}=Z^n}\|r( |Z\rp{1} H|+u_+^{\zeta(Z^1)}|H|)\|_{L^\infty(\D_{u_1}^{\ub_1})}\|r^{-\f12} \p^2 Z\rp{n-1}\phi\|_{L^2(\D_{u_1}^{\ub_1})}\\
&\les M^{-\f12} \dn {u_1}_+^{-\ga_0+\zeta(Z^n)}.
\end{align*}
For the term
$
II_n=\sum_{Z^{a_1}\sqcup Z^{b_1}\sqcup Z^b= Z^n,  b\le  n-2}u_+^{\zeta(Z^{b_1})}|Z\rp{a_1} \Phi|  |\p^2 Z\rp{b}\phi|,
$
we first derive
\begin{align*}
&\|r^\f12 II_n\|_{L^2(\D_{u_1}^{\ub_1})}\\
&\les \sum_{Z^{a_1}\sqcup Z^{b_1}\sqcup Z^b= Z^n}^{ 1\le b\le  n-2}  {u_1}_+^{\zeta(Z^{b_1})}\| u_+^{-\f12 \ga_0-\frac{3}{2}}r^\f12 Z\rp{a_1} \Phi\|_{L_\ub^2 L_u^2 L_\omega^4(\D_{u_1}^{\ub_1})}\c \|r  u_+^{\f12\ga_0+\frac{3}{2}} \p^2 Z\rp{b}\phi\|_{L^\infty L_\omega^4(\D_{u_1}^{\ub_1})}\\
&+\sum_{Z^{a_1}\sqcup Z^{b_1}= Z^n} {u_1}_+^{\zeta(Z^{b_1})}\| u_+^{-\f12\ga_0-\frac{3}{2}}r^{-\f12} Z\rp{a_1} \Phi\|_{L^2(\D_{u_1}^{\ub_1})}\c \|r u_+^{\f12\ga_0+\frac{3}{2}} \p^2 \phi\|_{L^\infty(\D_{u_1}^{\ub_1})}.
\end{align*}
With the following estimates
\begin{align}
&\| u_+^{-\frac{\ga_0}{2}-\frac{3}{2}}r^\f12 Z\rp{a_1} \Phi\|_{L_\ub^2 L_u^2 L_\omega^4(\D_{u_1}^{\ub_1})}\les \dn^\f12M^{-\f12} {u_1}_+^{-\ga_0+\zeta(Z^{a_1})},\, a_1\le n-1\label{6.17.6.18}\\
 &\|u_+^{\frac{\ga_0}{2}+\frac{3}{2}} r \p^2 Z\rp{b}\phi\|_{L^\infty L_\omega^4(\D_{u_1}^{\ub_1})}\les \dn^\f12 {u_1}_+^{\zeta(Z^b)},\,\, b\le n-2\label{6.17.7.18}\\
 & \| u_+^{-\frac{\ga_0}{2}-\frac{3}{2}}r^{-\f12} Z\rp{a_1} \Phi\|_{L^2(\D_{u_1}^{\ub_1})}\les \dn^\f12 M^{-\f12}{u_1}_+^{\zeta(Z^{a_1})-\ga_0}, \, \, a_1\le n\label{6.17.8.18}\\
&  \|u_+^{\frac{\ga_0}{2}+\frac{3}{2}}r \p^2 \phi\|_{L^\infty(\D_{u_1}^{\ub_1})}\les \dn^\f12\label{6.17.9.18},
\end{align}
we can directly obtain
\begin{equation*}
\|r^\f12 II_n\|_{L^2(\D_{u_1}^{\ub_1})}\les M^{-\f12}\dn {u_1}_+^{-\ga_0+\zeta(Z^n)}.
\end{equation*}
Now we derive the estimates (\ref{6.17.6.18})-(\ref{6.17.9.18}). We first apply (\ref{6.17.3.18}) and (\ref{4.25.3.18}) to obtain (\ref{6.17.7.18}) and (\ref{6.7.1.18}).

  By using  (\ref{6.5.7.18}) and (\ref{6.17.2.18}), we can obtain (\ref{6.17.6.18}) and (\ref{6.17.8.18}) if $\Phi$ is simply $\phi$.
  By using (\ref{5.18.11.18}) and (\ref{7.5.3.18}), we can obtain (\ref{6.17.6.18})  for $\Phi$; by using (\ref{5.18.11.18}) and (\ref{7.5.4.18}), (\ref{6.17.8.18}) is proved for $\Phi$.

 By combining the estimates of $I_n$ and $II_n$, (\ref{6.5.8.18}) is proved.

Next we prove (\ref{6.9.3.18}).
  By using (\ref{4.25.3.18}) and (\ref{4.25.1.18}), we derive
\begin{align*}
&{u_1}_+^{-\zeta(Z^n)+\f12\ga_0}\|r^\frac{3}{2} I_n\|_{L_u^1 L_\ub^2 L_\omega^2(\D_{u_1}^{\ub_1})}\\
&\les {u_1}_+^{-\zeta(Z^{n-1})+\f12\ga_0}\dn^\f12\| u_+^{-\frac{\ga_0}{2}+\f12}r^{\f12} \p^2 Z\rp{n-1}\phi\|_{L_u^1 L_\ub^2 L_\omega^2 (\D_{u_1}^{\ub_1})}\\
&\les {u_1}_+^{\f12\ga_0}\dn^\f12  M^{-\f12} \sup_{-\ub_1\le u\le u_1}E^\f12[\p Z\rp{n-1} \phi](\H_u^{\ub_1}) u_+^{1+\frac{\ga_0}{2}-\zeta(Z^{n-1})}\c\int_{-\ub_1}^{u_1}  u_+^{-\ga_0-\f12}du  \\
&\les {u_1}_+^{ -\frac{\ga_0}{2}+\f12}\dn M^{-\f12},
\end{align*}
which holds for any  $Z^{n-1}\subset Z^n$.

To estimate $II_n$,  by repeating the derivation of (\ref{6.17.6.18}) and (\ref{6.17.8.18}),  we have
\begin{align}
&\| u_+^{-\frac{\ga_0}{2}-\f12}r^\f12 Z\rp{a_1} \Phi\|_{L_\ub^2 L_u^2 L_\omega^4(\D_{u_1}^{\ub_1})}\les \dn^\f12 M^{-\f12} {u_1}_+^{-\ga_0+1+\zeta(Z^{a_1})},\, a_1\le n-1\nn\\
 & \| u_+^{-\frac{\ga_0}{2}-\frac{1}{2}}r^{-\f12} Z\rp{a_1} \Phi\|_{L^2(\D_{u_1}^{\ub_1})}\les \dn^\f12 M^{-\f12}{u_1}_+^{-\ga_0+1+\zeta(Z^{a_1})}, a_1\le n\nn.
\end{align}
 Combining (\ref{6.17.7.18}), (\ref{6.17.9.18}) with the above two estimates, by a standard H\"older's inequality, we can derive
\begin{align*}
&\|r^\frac{3}{2} II_n \|_{L_u^1 L_\ub^2 L_\omega^2(\D_{u_1}^{\ub_1})}\les\| r^\frac{1}{2} u_+ II_n\|_{L^2(\D_{u_1}^{\ub_1})}\|u_+^{-1}\|_{L_u^2 L^\infty(\D_{u_1}^{\ub_1})}\les  \dn M^{-\f12}  {u_1}_+^{\zeta(Z^n)- \ga_0+\f12}.
\end{align*}
 (\ref{6.9.3.18}) follows by combining the estimates for $I_n$ and $II_n$ .

We now consider $\str{n}\F_\Q$.  Recall from (\ref{6.7.3.18}) in the proof of Proposition \ref{6.2.5.18} that $\str{n}\F_\Q= \str{n}\F_{\Q,1}+\str{n}\F_{\Q,2}$. We will control the terms of $\str{n} \F_{\Q,1}$ and $\str{n} \F_{\Q,2}$ in a similar way.

Similar to the estimate of $I_n$, by using (\ref{6.7.1.18}) and using (\ref{7.5.4.18}) due to $b\le n$, we can derive
\begin{align*}
\|r^\f12 \str{n}\F_{\Q, 1}\|_{L^2(\D_{u_1}^{\ub_1})}&\les \sum_{ Z^b \sqcup Z^c=Z^n}{u_1}_+^{\zeta(Z^c)-\f12-\frac{\ga_0}{2}}\|u_+^{\f12+\frac{\ga_0}{2}}r\p  \phi \|_{L^\infty(\D_{u_1}^{\ub_1})}\| r^{-\f12}\p Z\rp{b}\phi \|_{L^2(\D_{u_1}^{\ub_1})}\\
&\les  \dn M^{-\f12} {u_1}_+^{-\ga_0+\zeta(Z^n)}.
\end{align*}
Note that  $1 \le b \le  a\le n-1$ in $\str{n}\F_{\Q, 2}$. We can employ (\ref{7.5.3.18}) for the term $\p Z\rp{b}\phi$ and apply (\ref{6.17.3.18}) to $\p Z\rp{a}\phi$
\begin{align*}
\|r^\f12\str{n} \F_{\Q,2}\|_{L^2(\D_{u_1}^{\ub_1})}&\les \sum_{Z^a \sqcup Z^b \sqcup Z^c=Z^n}{u_1}_+^{\zeta(Z^c)} \|r^\f12 \p Z\rp{b}\phi\|_{L_\ub^2 L_u^2 L_\omega^4(\D_{u_1}^{\ub_1})}\|r\p Z\rp{a}\phi\|_{L^\infty  L_\omega^4(\D_{u_1}^{\ub_1})}\\
&\les  \dn M^{-\f12} {u_1}_+^{-\ga_0+\zeta(Z^n)}.
\end{align*}
By combining the above two estimates, we can derive (\ref{6.8.2.18}).

Next we prove  (\ref{6.9.5.18}).  We first consider $\str{n} \F_{\Q, 1}$. By using (\ref{6.7.1.18}) and (\ref{7.5.4.18}), we have
\begin{align*}
\|r^\frac{3}{2} \str{n}\F_{\Q, 1}\|_{L_u^1 L_\ub^2 L_\omega^2(\D_{u_1}^{\ub_1})}&\les \sum_{Z^b\sqcup Z^c=Z^n}u_+^{\zeta(Z^c)}\| r \p \phi\|_{L^2_u L^\infty(\D_{u_1}^{\ub_1})}\|r^{-\f12} \p Z\rp{b}\phi\|_{L^2(\D_{u_1}^{\ub_1})}\\
&\les \dn \sum_{Z^b\sqcup Z^c=Z^n}{u_1}_+^{-\frac{\ga_0}{2}+\zeta(Z^c)} M^{-\f12} {u_1}_+^{-\frac{\ga_0}{2}+\zeta(Z^b)+\f12}\\
&= \dn M^{-\f12} {u_1}_+^{-\ga_0+\f12+\zeta(Z^n)}.
\end{align*}
 Due to  $1\le b\le a\le n-1$ in $\str{n}\F_{\Q,2}$,
we can derive by using (\ref{6.17.11.18}) and (\ref{7.5.3.18}) that
\begin{align*}
\|r^\frac{3}{2} \str{n}\F_{\Q,2}\|_{L_u^1 L_\ub^2 L_\omega^2(\D_{u_1}^{\ub_1})}&\les \sum_{z^a\sqcup Z^b\sqcup Z^c= Z^n }u_+^{\zeta(Z^c)}\| r \p Z\rp{b} \phi\|_{L_u^2 L_\ub^\infty L_\omega^4(\D_{u_1}^{\ub_1})}\|r^\f12 \p Z\rp{a}\phi\|_{L_u^2 L_\ub^2 L_\omega^4(\D_{u_1}^{\ub_1})}\\
&\les \dn M^{-\frac{1}{2}} \sum_{Z^a \sqcup Z^b\sqcup Z^c=Z^n}{u_1}_+^{-\frac{\ga_0}{2}+\zeta(Z^b)+\zeta(Z^c)}  {u_1}_+^{\zeta(Z^a)-\frac{\ga_0}{2}+\f12}\\
&\les \dn M^{-\frac{1}{2}}  {u_1}_+^{-\ga_0+\f12+\zeta(Z^n)}.
\end{align*}

By combining the estimates for $\str{n}\F_{\Q,1}$ and $\str{n}\F_{\Q, 2}$, we can obtain  (\ref{6.9.5.18}).

At last we consider the term $\str{n}\F_\C$. We recall the definition of $\str{n}\F_\C$ from Lemma \ref{6.3.5.18} and will first show (\ref{6.7.2.18}).  Note that  at least one of $b$ or $c$ is $\le 1$. Assume, without loss of generality, $b\le c$. Since $a\ge 1$ in  $\str{n}\F_\C$, $b+c\le n-1$. Thus $ 0\le b\le c\le n-1$ and $b\le 1$.
In view of (\ref{6.7.1.18}), (\ref{6.14.3.18}) and (\ref{6.14.10.18}),
we have
\begin{align*}
&\|r^\f12\str{n }\F_\C\|_{L^2(\D_{u_1}^{\ub_1})}\\
&\les\sum_{Z^a \sqcup Z^b \sqcup Z^c \sqcup Z^d=Z^n, a\ge 1, b\le c}u_+^{\zeta(Z^d)}\|r\p Z\rp{b}\phi\|_{L^\infty(\D_{u_1}^{\ub_1})}\| r^\f12 \p Z\rp{c}\phi\|_{L_\ub^2 L_u^\infty L_\omega^4(\D_{u_1}^{\ub_1})}\\
&\cdot\| Z\rp{a} \phi\|_{L^\infty_\ub L_u^2 L_\omega^4(\D_{u_1}^{\ub_1})}\les \dn^{\frac{3}{2}} M^{-\f12} {u_1}_+^{-\frac{3}{2}\ga_0-\f12+\zeta(Z^n)},
\end{align*}
  which gives (\ref{6.7.2.18}).

 By using (\ref{6.7.1.18}), (\ref{6.17.11.18}) and (\ref{6.3.11.18}), we have
\begin{align*}
&\|r^\frac{3}{2}\str{n }\F_\C\|_{L_u^1 L_\ub^2 L_\omega^2(\D_{u_1}^{\ub_1})}\\
&\les\sum_{Z^a \sqcup Z^b \sqcup Z^c \sqcup Z^d=Z^n, a\ge 1, b\le c}u_+^{\zeta(Z^d)}\|r\p Z\rp{b}\phi\|_{L^\infty(\D_{u_1}^{\ub_1})}\| r \p Z\rp{c}\phi\|_{L_u^2 L_\ub^\infty L_\omega^4(\D_{u_1}^{\ub_1})}\\
&\cdot\|r^{-\f12} Z\rp{a} \phi\|_{ L_u^2 L_\ub^2 L_\omega^4(\D_{u_1}^{\ub_1})}\les \dn^\frac{3}{2}{u_1}_+^{-\frac{3\ga_0}{2}-\f12+\zeta(Z^n)}.
\end{align*}
Thus (\ref{6.9.9.18}) is proved.
\end{proof}
\subsection{Boundedness theorem}\label{7.10.5.18}
We now combine the energy and weighted energy inequalities in Section \ref{eng} and the error estimates in Section \ref{error} to give the boundedness of energy and the weighted energy. The proof follows similarly as for Proposition \ref{3.25.3.18}.
\begin{theorem}[Boundedness of energies]\label{6.30.5.18}
For $n\le 3$, under the assumptions (\ref{6.10.4.18}) and (\ref{6.10.3.18}), there hold for $(u, \ub)\in \I$ that
\begin{align}
\sup_{-\ub\le u' \le u} &{u'}_+^{\ga_0-2\zeta(Z^n)}( E[Z\rp{n} \phi](\H_{u'}^\ub)+E[Z\rp{n}\phi](\Hb_\ub^{u'}))\les\E_{n,\ga_0}+M^{-2} \dn^2 u_+^{1-\ga_0},  \label{6.30.2.18}\\
\sup_{-\ub\le u' \le u }& {u'}_+^{-2\zeta(Z^n)-1+\ga_0}\left(\W_1[Z\rp{n} \phi](\H_{u'}^{\ub})+\W_1[Z\rp{n}\phi](\Hb_\ub^{u'})+\W_1[Z\rp{n} \phi](\D_{u'}^\ub)\right)\label{6.30.3.18}\\
&\les \E_{n, \ga_0}+M^{-2}\dn^2 u_+^{1-\ga_0}\nn.
\end{align}
\end{theorem}
\begin{remark}
We then can find a universal constant $C_3\ge 1$ such that both of the inequalities are bounded by
\begin{equation}
C_3 (\E_{n,\ga_0}+M^{-2} \dn^2 u_+^{1-\ga_0}).
\end{equation}
\end{remark}
\begin{proof}
We first consider (\ref{6.30.2.18}). Similar to (\ref{5.25.3.18}), we have $\widetilde{\Box}_{\bg}Z\rp{n}\phi-qZ\rp{n} \phi=\sF$, where
\begin{equation*}
\sF^\sharp=[\widetilde{\Box}_\bg, Z\rp{n}]\phi+\str{n} \F_\C+\str{n}\F_\Q, \qquad\sF^\flat =[Z\rp{n}, q]\phi.
\end{equation*}
 Then we apply the first inequality in Lemma \ref{5.25.2.18} to treat $\sF^\flat$, and apply (\ref{6.5.8.18}), (\ref{6.8.2.18}) and (\ref{6.7.2.18}) to treat $\sF^\sharp$. By using Proposition \ref{4.29.4.18},  we have due to $\dn M^{-1}< 1$ that
\begin{align*}
&\sup_{-\ub\le u' \le u} {u'}_+^{\ga_0-2\zeta(Z^n)}( E[Z\rp{n}\phi](\H_{u'}^\ub)+E[Z\rp{n} \phi](\Hb_\ub^{u'}))\\
&\les \E_{n,\ga_0}+M^{-2}\dn^2 (u_+^{1-\ga_0 } +M u_+^{-2\ga_0})+ \sup_{-u\le \ub'\le \ub}\sum_{i=1}^n  E[Z\rp{n-i} \phi](\Hb_{\ub'}^u) u_+^{\ga_0-2\zeta(Z^{n-i})},
\end{align*}
where the last term vanishes when $n=0$. Thus under the assumptions (\ref{6.10.4.18}) and (\ref{6.10.3.18}), (\ref{6.30.2.18}) holds true by induction.

Let $(u_1, \ub_1)\in \I$ be fixed.  To see the weighted energy estimate for $Z\rp{n}\phi$, by using Proposition \ref{4.29.5.18}, (\ref{6.23.7.18}), (\ref{6.25.1.18}) and the fact that  $\|r^\f12 \p Z\rp{a}\phi\|^2_{L^2(\Sigma_0^{u_1, \ub_1})}\les {u_1}_+^{-\ga_0+1}\E_{a, \ga_0}$ for $a\le n$,  we derive
 \begin{align*}
  \sup_{-\ub_1 \le u\le u_1} {u}_+^{\ga_0-1-2\zeta(Z^n)}&\big(\W_1[Z\rp{n} \phi](\D_{u}^{\ub_1})+ \W_1[Z\rp{n}\phi](\H_{u}^{\ub_1})+\W_1[Z\rp{n}\phi](\Hb_{\ub_1}^{u})\big)\\
&\les {u_1}_+^{-\ga_0}\dn^2 (M^{-1}+1)+\sum_{Z^m\subsetneq Z^n} {u_1}_+^{-2\zeta(Z^m)-1+\ga_0} \W_1[Z\rp{m} \phi](\D_{u_1}^{\ub_1})\\
&  + \sup_{ -\ub_1\le -\ub \le u\le u_1}\{ u_+^{\ga_0-2\zeta(Z^n)}  (E[Z\rp{n}\phi](\H_u^\ub)+ E[Z\rp{n}\phi](\Hb_\ub^u))\} + \E_{n, \ga_0},
  \end{align*}
  where we also employed the second inequality in Lemma \ref{5.25.2.18} and (\ref{6.9.3.18}), (\ref{6.9.5.18}) and (\ref{6.9.9.18}).
We then substitute the result of (\ref{6.30.2.18}) followed with an induction argument   to derive (\ref{6.30.3.18}).
\begin{align*}
\sup_{-\ub_1 \le u\le u_1} {u}_+^{\ga_0-1+2\zeta(Z^n)}&\big(\W_1[Z\rp{n}\phi](\D_{u}^{\ub_1})+ \W_1[Z\rp{n}\phi](\H_{u}^{\ub_1})+\W_1[Z\rp{n}\phi](\Hb_{\ub_1}^{u})\big)\\
&\les {u_1}_+^{-\ga_0}\dn^2 (M^{-1}+1) +M^{-2} \dn^2 {u_1}_+^{1-\ga_0} +  \E_{n, \ga_0}.
\end{align*}
Thus (\ref{6.30.3.18}) is proved.
\end{proof}
\subsection{Proof of Theorem \ref{thm_quasi}}
With $\dn=\C_1 \E_{3,\ga_0}$ and $\C_1=4 C_3 C$,   in view of (\ref{6.30.2.18}), (\ref{6.30.3.18}) and $M=3C \delta_1^{\f12}$,  to improve the bootstrap assumptions (\ref{4.25.1.18}) and (\ref{4.25.2.18}), we need to have
\begin{align*}
&C_3(\C_1^{-1}\dn +(3C)^{-2}\delta^{-1}_1\dn \C_1 \E_{3,\ga_0} u_+^{1-\ga_0})<2 \dn,
\end{align*}
where $C, \C_1, C_3>1$. Identically,
\begin{equation*}
C_3(\C_1^{-1}+(3C)^{-2}\C_1 u_+^{1-\ga_0})<2
\end{equation*}
which requires
\begin{equation}\label{6.30.6.18}
\frac{4}{9} C_3^2 C^{-1}u_+^{1-\ga_0}\le\frac{4}{9} C_3^2 C^{-1}{u_0(R)}_+^{1-\ga_0} <\frac{7}{4}.
\end{equation}
Next we determine $R$ so that (\ref{6.10.4.18})  can be improved and (\ref{6.10.3.18}) can be satisfied.

Note that  by using (\ref{4.25.3.18}), for $(u, \ub)\in \I$,
\begin{equation}\label{7.16.4.18}
r| H^{\a\b}(u,\ub, \omega)|\le C_2\dn^\f12 u_+^{-\f12\ga_0+\f12}.
\end{equation}
Thus, in view of  (\ref{6.30.4.18}),
\begin{equation*}
 \sup_{\a,\b}r (h-H^{\a\b})> C 3 \delta^{\f12}_1-  u_+^{-\frac{\ga_0}{2}+\f12}\C^\f12_1 C_2 \E^\f12_{3,\ga_0}.
\end{equation*}
With
\begin{equation}\label{6.30.7.18}
u_+^{-\frac{\ga_0}{2}+\f12} \C^\f12_1 C_2\le {u_0(R)}_+^{-\frac{\ga_0}{2}+\f12} \C^\f12_1 C_2<2C,
\end{equation}
the strict inequality in (\ref{6.10.4.18}) holds true.

By a direct substitution, to make (\ref{6.10.3.18})  hold, we need
\begin{equation*}
\C_1^\f12 (\frac{R}{2})^{\f12-\frac{\ga_0}{2}}< \f12 C\ti C^{-1}.
\end{equation*}
 Thus we require  $u_0(R)$ to satisfy the second inequalities in (\ref{6.30.6.18}), (\ref{6.30.7.18}) and the above inequality. With the help of  ${u_0(R)}_+>\f12 R$, we can fix $ R(\ga_0, C)$, the lower bound of $R$, such that these inequalities hold.

\section{Einstein scalar fields}\label{Eins}
In this section, we apply the approach in Section \ref{quasi} to prove  the nonlinear  stability result for Einstein scalar fields, exterior to a schwarzschild cone with small positive mass, which is stated in Theorem \ref{eins_thm}.

 Under the wave coordinates \begin{footnote}{The wave coordinates $\{x^\mu\}_{\mu=0}^3 $ are required to be the solution of $\Box_\bg x^\mu=0$ where $\Box_\bg$ is the Laplace-Beltrami operator of the Einstein metric $\bg$. }\end{footnote}, let $\bh_{\mu\nu}= \bg_{\mu\nu}-\bm_{\mu\nu}$.  The Einstein equation with scalar fields  takes the form of
\begin{equation}\label{7.7.1.18}
\left\{\begin{array}{lll}
&\widetilde{\Box}_\bg \bh_{\mu\nu}=(A^{\a\b}_{\mu\nu}+ G^{\a\b}_{\mu\nu}(\bh) )\p_\a \bh  \p_\b \bh+\p_\mu \phi\c \p_\nu \phi+ \bg_{\mu\nu} q_0 \phi^2 \\
&\widetilde{\Box}_\bg \phi-q_0 \phi= 0
\end{array}\right.,
\end{equation}
where we assume the constant $0\le q_0\le 1$ without loss of generality. For each fixed $(\mu, \nu)$,  $A_{\mu\nu}$  is a matrix of constant components.
    For each fixed $(\mu,\nu)$, $G_{\mu\nu}^{\a\b}(\bh)$ are smooth functions of $\bh$. They  represent the  product   $\bh_{\a\b}$ or  $H^{\a\b}$ with  components of  $\bg$.  We will symbolically represent all such functions as $G(\bh)$. Such  $G(\bh)$   vanishes at $(\bh_{\a\b})\equiv({\bf 0})$.   Other constant coefficients on the  righthand side  of (\ref{7.7.1.18}) have been simplified to be $1$.  \begin{footnote}{See \cite{Lindrod2} for the more detailed structure. We do not need the weak null structure to prove  the result of Theorem \ref{eins_thm}.}\end{footnote}

Let $m_0>0$ be a fixed small number.
Denote $\cir{\bh}{}_ {\mu\nu}=\frac{m_0}{r}\delta_{\mu\nu} .$  We decompose
$$\bh_{\mu\nu}=\bh^1_{\mu\nu}+\cir{\bh}{}_{\mu\nu}.$$

 This reduces (\ref{7.7.1.18}) to the equations for $(\bh^1, \phi)$,
\begin{equation}\tag{ES}\label{7.7.2.18}
\left\{\begin{array}{lll}
&\widetilde{\Box}_\bg \bh^1_{\mu\nu}=\N(\bh)(\p \cir{\bh}+\p \bh^1) \c \p \bh^1 +\p \phi\c \p \phi+ \bg_{\mu\nu} q_0 \phi^2+\S_{\mu\nu} \\
&\widetilde{\Box}_\bg \phi-q_0 \phi= 0
\end{array}\right.,
\end{equation}
where $\N(\bh)=1+G(\bh)$ symbolically and
\begin{equation}\label{7.9.1.18}
S_{\mu\nu}=-\widetilde{\Box}_\bg \cir{\bh}{}_{\mu\nu}+\N(\bh)\p \cir{\bh}\p \cir{\bh}.
\end{equation}

We remark that the structure of wave coordinates implies (see \cite[Lemma 8.1]{Lindrod2})
\begin{equation}\label{wave_coord}
|\Lb H^{\Lb \Lb}|\les |\bar\p H|+|H\c \p H|
\end{equation}
which can provide  some convenience to the proof of boundedness of energy.  This  will be shown shortly.   In this section, we  prove Theorem \ref{eins_thm} by  applying our approach in Section \ref{quasi} to $(\bh^1, \phi)$ with potentials $(0, q_0)$.

Let  $1<\ga_0<2$ be fixed. We define for the initial data $\bh^1[0]=(\bh^1(0), \p_t \bh^1(0))$  and $\phi[0]=( \phi(0), \p_t\phi(0))$ the weighted norm
\begin{equation*}
\E_{3,\ga_0, R}(\bh^1, \phi)= \E_{3,\ga_0, R, 0}(\bh^1[0])+\E_{3,\ga_0, R, q_0}(\phi[0]).
\end{equation*}
The extra subindex $C$ of $\E_{3,\ga_0, R, C}$  on the right denotes  the constant potential function of each equation.
In this section, we assume
\begin{equation}\label{7.25.1.18}
\E_{3, \ga_0, R}(\bh^1, \phi)\le C_0 m_0^2
\end{equation}
where $C_0\ge 1$ is a fixed constant, $R\ge 2$ with  the lower bound determined later.

$\E_{3,\ga_0, R, 0}(\bh^1[0])<\infty $  implies  $
\liminf_{|x|\rightarrow\infty} \bg_{\mu\nu}(0,x)= \cir{\bg}{}_{\mu\nu}.
$
where
$\cir{\bg}{}_{\mu\nu}=\bm_{\mu\nu}+\frac{m_0}{r} \delta_{\mu\nu}$.
It is direct to compute
\begin{equation}\label{7.7.6.18}
\cir{\bg}{}^{L L}=\cir{\bg}{}^{\Lb \Lb}<\frac{-m_0 }{2r}.
\end{equation}

To prove Theorem \ref{eins_thm}, we fix
$$h=-\frac{m_0}{20 r},  \mbox{ i.e. } M_0=-\frac{m_0}{20}$$
 and  show that with $M=m_0$ and the constant potentials $(0, q_0)$, all the norms in (\ref{3.19.1}) for $Z\rp{i}(\bh^1, \phi), \, i\le 3$  remain small in the region $\{u\le u_0(M_0)\}$ provided that $m_0$ in  (\ref{7.25.1.18}) is sufficiently small and  $R\ge R_0(\ga_0, C_0)$. The constant $R_0(\ga_0, C_0)$ will be specified at the end of the proof.

\subsection{Preliminaries}
\begin{lemma}\label{7.29.4.18}
If   $\E_{3,\ga_0, R}(\bh^1[0])\le C_0 m_0^2, \, C_0\ge 1$, there exists  a constant $\ti C_1\ge 1$ depending on $C_0$ and $\ga_0$, if $ \ti C_1 R^{-\frac{\ga_0}{2}+\f12}\le 1$, there hold
\begin{equation}\label{7.10.2.18}
r(h-H^{\Lb\Lb})>\frac{m_0}{3},\quad r (h-H^{LL})>\frac{m_0}{3}, \mbox{ on } \Sigma_0\cap \{r\ge R\}.
\end{equation}
\end{lemma}
\begin{proof}
It follows from (\ref{6.24.9.18}) that
\begin{equation}\label{7.10.4.18}
r |\p\rp{l} \bh^1(u, \ub, \omega)|\les u_+^{-\frac{\ga_0}{2}+\f12-l} \E^\f12_{2,\ga_0}, \qquad l\le 1
\end{equation}
Thus  $r|\bh|\les  (\E^\f12_{2,\ga_0}+m_0)$.

 For small $\bh$, $H^{\mu\nu}=-\bh^{\mu\nu}+\O^{\mu\nu}(\bh^2) $ where $\bh^{\mu\nu}=\bm^{\mu\mu'}\bm^{\nu\nu'}\bh_{\mu'\nu'}$ and $\O^{\mu\nu}(\bh^2)$ vanishes to second order at $\bh=0$.
Thus
\begin{equation}\label{7.29.3.18}
|H^{\mu\nu}-\cir{H}{}^{\mu\nu}|\les |\bh^1|+|\bh|^2
\end{equation}
where $\cir{H}{}^{\mu\nu}=\cir{\bg}{}^{\mu\nu}-\bm^{\mu\nu}.$
By using the above two  inequalities, (\ref{7.10.4.18}) and  (\ref{7.7.6.18}), we can derive
\begin{equation*}
r (h-H^{\Lb \Lb})\ge  \frac{9}{20} m_0 -C (u_+^{-\frac{\ga_0}{2}+\f12} \E^\f12_{2,\ga_0}+r^{-1} (m_0+\E^\f12_{2,\ga_0})^2)
\end{equation*}
where the universal constant $C\ge 1$. Due to $|u(R)|>\frac{R}{2}$ and $m_0<1$, if  we choose $R$ such that
\begin{equation*}
C((\frac{R}{2})^{\f12-\frac{\ga_0}{2}}C_0^\f12+R^{-1}(1+C_0^\f12)^2 )<\frac{1}{10}
\end{equation*}
 Therefore we may require $ \f12 C (\frac{R}{2})^{\frac{1-\ga_0}{2}}(C^\f12_0+2)^2\le \frac{1}{10}$, then
 (\ref{7.10.2.18})  holds.
 The same treatment works the same for $r(h-H^{LL})$.  The proof is completed.
\end{proof}

The above result shows that  the assumption of (\ref{6.10.4.18}) holds true on $\Sigma_0\cap \{r\ge R\}$ with  $M=m_0$ if $R^{\f12-\frac{\ga_0}{2}}\le {\ti C_1}^{-1}$. We will further specify the lower bound $R_0(\ga_0, C_0)$ during the proof of Theorem \ref{eins_thm}.

{\bf Bootstrap assumptions}
To prove Theorem \ref{eins_thm}, we make the following bootstrap assumptions.

Let $\ub_*>-u_0$ be any fixed number, where $u_0=-r_*(-\frac{m_0}{20}, R).$ In $\I=\{(u,\ub), -\ub_*\le -\ub\le u\le u_0\}$, suppose   (\ref{6.10.4.18}) holds,  and  the  assumption   $(\BA{3})$ holds for $(\bh^1, \phi)$ with  $\dn=C_1 m_0^2 $, $C_1\ge 1$ to be chosen.  $(\BA{3})$ is restated as below for $n\le 3$ and $Z\in\{\Omega, \p\}$,
\begin{equation}\label{7.26.9.18}
\begin{split}
&E[Z\rp{n} (\bh^1, \phi)](\H_u^\ub)+E[Z\rp{n} (\bh^1, \phi)](\Hb_\ub^u)\le  2\dn u_+^{-\ga_0+2\zeta(Z^n)}\\
&\W_1[Z\rp{n} (\bh^1, \phi)](\H_u^\ub)+\W_1[Z\rp{n} (\bh^1,\phi)](\Hb_\ub^u)+\W_1[Z\rp{n}(\bh^1, \phi)](\D_u^\ub)\le  2\dn u_+^{-\ga_0+1+2\zeta(Z^n)}
\end{split}
\end{equation}
for all $(u, \ub)\in \I$.

Thus the full set of decay estimates in Section \ref{decay} hold for $(\bh^1,\phi)$. We will quote the results in the proofs whenever necessary.

Before starting to prove Theorem \ref{eins_thm}, we highlight the difference in analysis in comparison with Section \ref{quasi}.
\begin{itemize}
\item In the  Einsteinian case, there holds $H^{\mu\nu}=-\bh^{\mu\nu}+O(\bh^2)$. Thus $H$ does not depend on $\p(\bh^1, \phi)$. In terms of regularity it seems better than in Section \ref{quasi}. Nevertheless, the  small static part $\bh^0$ in $\bh$ slows down the decay of $H$. In the proofs of Theorem \ref{thm2} and Theorem \ref{thm_quasi}, we take advantage of the fact that  $R^{1-\ga_0}$ can be small so as to improve the bootstrap assumption, while in this section, we lose such extra smallness for the critical terms. This requires us to separate carefully the evolutionary part of the metric $\bh^1$ from $\bh^0$ in analysis. For the borderline terms appeared in the error estimates, we  bound them  by energies which have been controlled in the induction. Although these  energies are of the same order, the signature of $Z^n$ derivative is strictly smaller. See Lemma \ref{7.28.2.18} and Section \ref{einsbd}.

    \item  Comparing the equations of (\ref{7.7.2.18}) with the standard equation (\ref{eqn_1}), (\ref{7.7.2.18}) have quite a few extra error terms on the righthand side.  The term $q_0 \bg_{\mu\nu}\phi^2$ in (\ref{7.7.2.18}) requires an improved decay  estimate for the scalar field if $q_0>0$. The small static part $\bh^0$ of the metric also gives a set of error terms. Such set of error estimates are included in Lemma \ref{7.8.15.18}.

   \item The changed asymptotic behavior of the metric $H$ influences the  proofs of Proposition \ref{4.29.4.18} and Proposition \ref{4.29.5.18}.We will  go   through the analysis in the proof of Proposition \ref{7.28.3.18}.
\end{itemize}

 We first give an improve decay estimate of $\phi$  compared with (\ref{3.24.11.18}) in  Proposition \ref{5.20.1.18}.
\begin{proposition}
For $(u, \ub)\in \I$, there holds
\begin{equation}\label{7.8.3.18}
q_0|r^\frac{5}{4} Z\rp{l}\phi|^2 \les (\E_{l+2,\ga_0}+\dn)u_+^{-\ga_0+2\zeta(Z^l)+\frac{1}{2}},\quad  l\le 1.
\end{equation}
\end{proposition}
\begin{proof}
To begin with, noting  that $0\le q_0\le 1$ is now a fixed constant, (\ref{6.24.4.18}) can be improved to be
\begin{equation}\label{7.8.4.18}
q_0^2 \int_{S_r} r^{6+2\ga_0-4\zeta(Z^i)}|Z\rp{i}\phi|^4 d\omega \les \E^2_{i+1,\ga_0}, \qquad i\le 2.
\end{equation}
Similar to \cite[Section 3.2]{mMaxwell}, by applying (\ref{6.24.11.18}) to $f=Z\rp{l}\phi$ with $l\le 1$, $\ga=\frac{5}{4}$, $\ga_0'=\f12, \ga_2=1$,
\begin{align*}
\sup_{S_{u, \ub}}|r^{\frac{5}{4}} Z\rp{l}\phi|^4
&\les  \int_{S_{u, -u}} r^ 3| \Omega\rp{\le 1}Z\rp{l} \phi|^4 +\int_{\H_u^\ub} |\Omega\rp{\le 2} Z\rp{l}\phi|^2\\
&\times  \int_{\H_u^{\ub}} r^{-\frac{3}{2}}|L' (\Omega\rp{\le 1} (r^{\frac{5}{4}}Z\rp{l} \phi))|^2.
\end{align*}
By using (\ref{3.23.4.18}) and  $|L'(r ^\frac{5}{4} f)| \les r^{\frac{1}{4}} (|f|+|L '(rf)|)$, we derive
\begin{align*}
\int_{\H_u^{\ub}} r^{-\frac{3}{2}}q_0 |L' (\Omega\rp{\le 1} (r^{\frac{5}{4}}Z\rp{l} \phi))|^2&\les q_0\int_{\H_u^\ub} r^{-1} |\{L'(r\cdot\}\rp{\le 1} \Omega\rp{\le 1} Z\rp{l}\phi)|^2\\
&\les \W_1[\Omega\rp{\le 1} Z\rp{l}\phi](\H_u^\ub)+u_+^{-1}E[\Omega\rp{\le 1} Z\rp{l}\phi](\H_u^\ub)  \\
&\les u_+^{-\ga_0+1+2\zeta(Z^l)}\dn.
\end{align*}
Thus we conclude by using (\ref{7.8.4.18}) that
\begin{align*}
q_0^2\sup_{S_{u, \ub}}|r^{\frac{5}{4}} Z\rp{l}\phi|^4&\les  q_0^2\int_{S_{u, -u}} r^ 3| {\Omega_{ij}}^{\le 1}Z\rp{l} \phi|^4 +E[\Omega\rp{\le 2}Z\rp{l}\phi](\H_u^\ub) u_+^{-\ga_0+1+2\zeta(Z^l)}\dn \\
&\les u_+^{-2\ga_0+1+4\zeta(Z^l)}\dn^2.
\end{align*}
\end{proof}

If $Z\in \{\p, \Omega_{\mu\nu}, 0\le \mu<\nu\le 3\}$,
\begin{equation}\label{7.7.7.18}
|Z\rp{l} \cir{\bh}|\les m_0 r^{-1+\zeta(Z^l)}, l \ge 0.
\end{equation}
If  $Z\in \{\Omega, \p\}$, we can use the fact,
\begin{equation}\label{7.7.9.18}
|Z\rp{l} \cir{\bh}|\les m_0 r^{-1-l}, l \ge 0
\end{equation}
to treat the error terms.

\begin{proposition}[Decay estimates]\label{Decay_eins}
\begin{align}
&|\bh, H|\les r^{-1}(m_0+\dn^\f12 u_+^{-\frac{\ga_0}{2}+\f12}),\label{7.26.11.18}\\
&|\D_*  H|\les |\D_*\bh^1|+r^{-2}m_0, \quad \D_*\in \{ \p , \ud \p, \bar \p\},\label{7.26.12.18}\\
&r u_+ |\p H|\les \dn^\f12 +m_0\label{7.27.5.18},  \\
&|Z\rp{n} H, Z\rp{n} G(\bh)|\les \sum_{Z^a \sqcup Z^b=Z^n}u_+^{\zeta(Z^b)}( |Z\rp{a} \bh^1|+ m_0 r^{-1+\zeta(Z^a)}), n\le 3,\label{7.29.1.18}\displaybreak[0]\\
& |Z\rp{n} H|\les \sum_{Z^a \sqcup Z^b=Z^n, a\ge 1}u_+^{\zeta(Z^b)}( |Z\rp{a} \bh^1|+ m_0 r^{-1-a}), \quad 1\le n \le 3.\label{7.27.10.18}
\end{align}
\end{proposition}
\begin{proof}
The proofs of (\ref{7.26.11.18})-(\ref{7.27.5.18}) can follow directly from the property of $H^{\mu\nu}(\bh)$, $\bh=\cir{\bh}+\bh^1$ and applying (\ref{3.24.18.18}) and  (\ref{3.24.11.18}) to $\bh^1$.
(\ref{7.29.1.18}) can be similarly proved as Lemma \ref{5.17.2.18}. (\ref{7.27.10.18}) is an improved version, which relies on (\ref{7.7.9.18}) instead of (\ref{7.7.7.18}).
\end{proof}

\begin{proposition}\label{7.28.3.18}
Consider
 \begin{equation*}
\widetilde{\Box}_{\bg} \varphi= \sF+q\varphi
\end{equation*}
where  $q\ge 0$ is a constant,  $\bh^1$ in the  Lorentzian metric  $\bg_{\mu\nu}=\bh^1_{\mu\nu}+\cir{\bh}_{\mu\nu}+\bm_{\mu\nu}$ satisfy the bootstrap assumptions (\ref{7.26.9.18}) and (\ref{7.25.1.18}). By  assuming (\ref{6.10.4.18}),
  Proposition \ref{4.29.5.18} holds true,  and Proposition \ref{4.29.4.18} holds if assuming (\ref{wave_coord}) instead of  (\ref{6.10.3.18}).
\end{proposition}
\begin{proof}
We first confirm that if the background metric is the Einstein metric $\bg(\bh)$, the weighted energy estimate in Proposition \ref{4.29.5.18} holds.
The proof is carried out in two steps.

{\bf  Step 1}
For the terms $I$, $II$ and $III$ defined in (\ref{7.27.8.18}),  we will show there holds  for $(u_1, \ub_1)\in \I$ that
\begin{equation}\label{7.26.1.18}
\begin{split}
\int_{\D_{u_1}^{\ub_1}}|I|+|II|+|III|&\les (\dn^\f12 m_0^{-\f12}+m_0^\f12)\{{u_1}_+^{-\frac{3\ga_0}{2}+1} (\sup_{-\ub_1\le u\le u_1} E[\varphi](\H_u^{\ub_1}) u_+^{\ga_0}\\
&+ \sup_{-u_1\le \ub \le \ub_1}E[\varphi](\H_{\ub}^{u_1}){u_1}_+^{\ga_0})+ \int_{-\ub_1}^{u_1} ({u}_+^{-\frac{\ga_0+1}{2}}+u_+^{\frac{\ga_0}{2}-2})\W_1[\varphi](\H_{u}^{\ub_1}) du\}.
\end{split}
\end{equation}
This implies (\ref{5.5.2.18}) holds with the bound replaced by the righthand side of (\ref{7.26.1.18}).

We first give the following straightforward calculations.
\begin{align}
\int_{\D_{u_1}^{\ub_1}} r^{-2} (\ud \p \varphi)^2 &\les {u_1}_+^{-1}\sup_{-u_1\le \ub \le \ub_1} E[\varphi](\H_{\ub}^{u_1}),\label{7.27.2.18}\\
\int_{\D_{u_1}^{\ub_1}} r^{-1}|\p \varphi\c \bar \p \varphi|&\les \int_{\D_{u_1}^{\ub_1}} r^{-1} (|\bar\p\varphi||\ud \p \varphi|+|\bar \p \varphi|^2)\label{7.27.3.18}\\
&\les \|r^{-1}\ud \p \varphi \|_{L^2(\D_{u_1}^{\ub_1})}\|\bar\p \varphi\|_{L^2(\D_{u_1}^{\ub_1})}+\int_{-\ub_1}^{u_1} u_+^{-1} E[\varphi](\H_u^{\ub_1}) du \nn\\
&\les {u_1}_+^{-\ga_0}\big(\sup_{-u_1\le \ub\le \ub_1}E[\varphi](\H_{\ub}^{u_1}) {u_1}_+^{\ga_0}+\sup_{-\ub_1\le u\le u_1} E[\varphi](\H_u^{\ub_1}) u_+^{\ga_0} \big),\nn
\end{align}
\begin{align}
\int_{\D_{u_1}^{\ub_1}}r^{-2}| L(r\varphi) \p \varphi  |&\les \int_{-\ub_1}^{u_1}\|r^{-\f12} L (r\varphi)\|_{L^2(\H_u^{\ub_1})}\| r^{-\f12}\p\varphi \|_{L^2(\H_u^{\ub_1})} r^{-1} du  \label{7.27.6.18},\\
&\les m_0^{-\f12}\int_{-\ub_1}^{u_1} \{\W_1[\varphi](\H_u^{\ub_1}) u_+^{\frac{\ga_0}{2}-2}+  u_+^{-\frac{3\ga_0}{2}}E[\varphi](\H_u^{\ub_1}) u_+^{\ga_0} \}du.\nn
\end{align}

We also note that by using (\ref{7.26.11.18}) and  (\ref{7.26.12.18})
\begin{align}
&|H||\p H|\les r^{-1} (m_0+\dn^\f12)(r^{-2}m_0+|\p \bh^1|),\label{7.26.13.18}\\
&r( |LH|+|H||\p H|)\les r |L \bh^1|+(m_0+\dn^\f12)|\p \bh^1|+r^{-1}(m_0+\dn^\f12)\label{7.27.1.18}
\end{align}
The parts of $\p \bh^1$ will be  treated similar to Section \ref{quasi}.
We first control  $I$  and $II$ in view of (\ref{6.17.13.18}).  By using (\ref{7.27.2.18}),   (\ref{7.26.13.18}) and  (\ref{4.2.1.18}), we have
\begin{align*}
&\int_{\D_{u_1}^{\ub_1}} r |H||\p H| |\ud \p \varphi|^2\\
&\les (m_0+\dn^\f12) \big(\int_{\D_{u_1}^{\ub_1}}  |\p \bh^1||\ud\p \varphi|^2 +m_0 r^{-2} |\ud \p \varphi|^2 \big)\\
&\les (m_0+\dn^\f12)\{ \|r^\f12 \p \bh^1\|_{L_\ub^2 L^\infty(\D_{u_1}^{\ub_1})}\|r^{-\f12} \ud \p \varphi\|_{L^2(\D_{u_1}^{\ub_1})}\| r\ud \p \varphi\|_{L_\ub^\infty L_u^2 L_\omega^2(\D_{u_1}^{\ub_1})}\\
&\qquad \qquad \qquad\quad+m_0 {u_1}_+^{-1}\sup_{-u_1\le \ub \le \ub_1} E[\varphi](\H_{\ub}^{u_1})\}\\
&\les (m_0+\dn^\f12) (\dn^\f12 m_0^{-1}+m_0){u_1}_+^{-\frac{3}{2} \ga_0+\f12} (\sup_{-\ub_1 \le u\le u_1} E[\varphi](\H_u^{\ub_1}) u_+^{\ga_0}+\sup_{-u_1\le \ub\le \ub_1} E[\varphi](\Hb_\ub^{u_1}) {u_1}_+^{\ga_0}).
\end{align*}
Next by using (\ref{7.27.1.18}), (\ref{7.27.3.18}), (\ref{3.24.18.18}) and  (\ref{7.26.11.18}), we can derive
\begin{equation}\label{7.27.4.18}
\begin{split}
\int_{\D_{u_1}^{\ub_1}} &r(|L H|+|H||\p H| +|H|) |\p \varphi\c \bar \p \varphi|\\
&\les \int_{\D_{u_1}^{\ub_1}} (r |L \bh^1 |+r^{-1}(m_0+\dn^\f12))|\p \varphi\c \bar\p \varphi|\nn\\
&\les (\dn^\f12 m_0^{-\f12}+m_0) {u_1}_+^{-\frac{3}{2}\ga_0+1} (\sup_{-\ub_1 \le u\le u_1} E[\varphi](\H_u^{\ub_1}) u_+^{\ga_0}+\sup_{-u_1\le \ub\le \ub_1} E[\varphi](\Hb_\ub^{u_1}) {u_1}_+^{\ga_0}),\nn
\end{split}
\end{equation}
where we used the estimate (\ref{7.28.4.18}) to treat the product term with $L\bh^1$.

By using  (\ref{7.26.12.18})  and (\ref{7.27.5.18}), we have
\begin{align*}
\int_{\D_{u_1}^{\ub_1}} r|\p H| |\bar \p \varphi|^2 &\les (\dn^\f12+m_0)\int_{-\ub_1}^{u_1} \|\bar\p \varphi|^2_{L^2(\H_u^{\ub})}  du \\
&\les (\dn^\f12 +m_0) {u_1}_+^{-\ga_0} \sup_{-\ub_1 \le u \le u_1}u_+^{\ga_0}E[\varphi](\H_u^{\ub_1})
\end{align*}
and
\begin{align*}
\int_{\D_{u_1}^{\ub_1}} |L(r\varphi) \p \varphi\c \p H| &\les \int_{\D_{u_1}^{\ub_1}} | L(r\varphi) \c \p \varphi\c \p \bh^1|+\int_{\D_{u_1}^{\ub_1}}r^{-2} m_0 |L(r \varphi)| |\p\varphi|.
\end{align*}
Note that the first term on the right  can be treated as (\ref{7.27.7.18}) and  the second term  is treated  by using (\ref{7.27.6.18}). Thus
\begin{align}
\int_{\D_{u_1}^{\ub_1}} |L(r\varphi) \p \varphi\c \p H|&\les (\dn^\f12 m_0^{-\f12}+m_0^\f12 )\big(\int_{-\ub_1}^{u_1} ({u}_+^{-\frac{\ga_0+1}{2}}+u_+^{\frac{\ga_0}{2}-2})\W_1[\varphi](\H_{u}^{\ub_1}) du\nn\\
&+{u_1}_+^{-\frac{3}{2}\ga_0+1}\sup_{-\ub_1\le u\le u_1} E[\varphi](\H_{u}^{\ub_1})u_+^{\ga_0}\big)\nn.
\end{align}
Hence we completed the estimates for $I$ and $II$.

 Next we consider the term $III$ defined in (\ref{7.27.8.18}) with the bound given in (\ref{7.26.3.18}). The first term on the right of (\ref{7.26.3.18})  has been treated. We will treat the remaining terms in the sequel.
We note that by using Proposition \ref{Decay_eins},
\begin{equation}\label{7.27.9.18}
(r |\p H|+|H|)|H|\les r^{-2} (m_0+\dn^\f12)^2 +(m_0+\dn^\f12)|\p \bh^1|.
\end{equation}
Thus by using (\ref{3.24.18.18}) for $\p \bh^1$
\begin{align*}
\int_{\D_{u_1}^{\ub_1}} (r |\p H|+|H|)|H| |\p \varphi|^2& \les (m_0+\dn^\f12)\int_{\D_{u_1}^{\ub_1}} \left(r^{-1} (m_0+\dn^\f12)+\dn^\f12 u_+^{-\frac{\ga_0+1}{2}}\right)r^{-1} |\p \varphi|^2\\
&\les m_0^{-1} (m_0+\dn^\f12)^2 {u_1}_+^{-\ga_0} \sup_{-\ub_1\le u\le u_1}E[\varphi](\H_u^{\ub_1}) u_+^{\ga_0}.
\end{align*}
Similarly,
\begin{align*}
\int_{\D_{u_1}^{\ub_1}} (r |\p H|+|H|) q\varphi^2 &\les (m_0+\dn^\f12) \int_{\D_{u_1}^{\ub_1}} u_+^{-1} q\varphi^2\\
& \les (m_0+\dn^\f12) {u_1}_+^{-\ga_0} \sup_{-\ub_1\le u\le u_1} E[\varphi](\H_u^{\ub_1})u_+^{\ga_0}.
\end{align*}
Thus the estimate for $III$ is completed and  (\ref{7.26.1.18}) is proved.

{\bf Step 2}  We  give the estimates of $\Er_1$ defined in (\ref{4.30.6.18}) and  $\Er_2$ defined in (\ref{7.26.2.18}).

Noting that  (\ref{7.26.11.18}) implies  $r|H|\les m_0+\dn^\f12$,  we can follow the same treatment as in the proof of Proposition \ref{4.29.5.18} to  obtain
\begin{align*}
|\Er_1|&\les  (\dn ^\f12+m_0)\big({u_1}_+^{-1}\int_{\H_{u_1}^{\ub_1}} r(L(r\varphi)-H^{\Lb \Lb} \Lb \varphi)^2  du d\omega+E[\varphi](\H_{u_1}^{\ub_1})+E[\varphi](\H_{\ub_1}^{u_1})\\
&+\sup_{-\ub_1\le u\le u_1}\int_{S_{u, -u}} r\varphi^2 d\omega\big),\\
|\Er_2|& \les (\dn^\f12 m_0^{-\f12}+ m_0^ \f12)\big(\W_1[\varphi](\Hb_{\ub_1}^{u_1})+
E[\varphi](\Hb_{\ub_1}^{u_1})+\sup_{-\ub_1\le u\le u_1} \|r^{-\f12} \varphi\|^2_{L^2(S_{u,-u})}\big).
\end{align*}
By substituting these two estimates in to (\ref{6.22.1.18}) and noting that
\begin{align*}
\int_{\H_{u_1}^{\ub_1}} r^3 |H|^2 |\Lb \varphi|^2 d\ub d\omega& \les \int_{\H_{u_1}^{\ub_1}}  r (\dn^\f12+m_0)^2 |\Lb \varphi|^2 d\ub d\omega\\
&\les (\dn^\f12+m_0)^2 m_0^{-1} E[\varphi](\H_{u_1}^{\ub_1}),
\end{align*}
we have
\begin{align*}
&\W_1[\varphi](\D_{u_1}^{\ub_1})+ \W_1[\varphi](\H_{u_1}^{\ub_1})+\W_1[\varphi](\Hb_{\ub_1}^{u_1})\nn\\
&\les \| |r^\f12 \p \varphi|+r^{-\f12} |\varphi|+q^\f12_0 r^\f12|\varphi|\|^2_{L^2(\Sigma_0^{u_1, \ub_1})}+\int_{\D_{u_1}^{\ub_1}} |\sF (X\varphi+\varphi)| +\frac{q}{2} \varphi^2   \nn\\
&+ (m_0^\f12+\dn^\f12 m_0^{-\f12}) \big(E[\varphi](\H_{u_1}^{\ub_1})+E[\varphi](\Hb_{\ub_1}^{u_1})+\sup_{-\ub_1\le u\le u_1} \int_{S_{u, -u}} r \varphi^2 d\omega \big)\nn\\
&+ (\dn ^\f12 m_0^{-\f12}+m_0^\f12)\{{u_1}_+^{1-\frac{3\ga_0}{2}+2p} (\sup_{-\ub_1\le u\le u_1} E[\varphi](\H_u^{\ub_1})u_+^{\ga_0-2p}\\
&+{u_1}_+^{\ga_0-2p}\sup_{-u_1\le \ub\le \ub_1} E[\varphi](\Hb_{\ub}^{u_1}))+\int^{u_1}_{-\ub_1} (u_+^{-\frac{\ga_0+1}{2}}+u_+^{\frac{\ga_0}{2}-2}) \W_1[\varphi](\H_u^{\ub_1})du\},\nn
\end{align*}
where $p\le 0$ is any constant. By repeating the same treatment on  (\ref{7.12.1.18}), we can derive the same inequality in Proposition \ref{4.29.5.18}.

Next we show Proposition \ref{4.29.4.18} holds without the assumption of (\ref{6.10.3.18}).
 Indeed,
note that due to the wave coordinate condition (\ref{wave_coord}), we can derive in view of (\ref{3.24.18.18}), (\ref{7.26.12.18}), (\ref{7.27.5.18}) and (\ref{6.2.2.18}) that
\begin{equation}\label{7.10.3.18}
 |\Tr[\varphi]|\les (\dn^\f12+m_0) r^{-1} u_+^{-1}((\frac{u_+}{r})^\f12 |\ud\p \varphi|^2+| \p \varphi| |\bar \p \varphi|).
\end{equation}
Substituting the above inequality to (\ref{6.10.2.18}) followed with using Gronwall inequality, we can obtain Proposition \ref{4.29.4.18} without the assumption of (\ref{6.10.3.18}).
\end{proof}
\subsection{Error estimates}
The commutator $[\Box_\bg, Z\rp{n}]\varphi$ with $\varphi\in\{\bh^1,\phi\}$ contains borderline terms, which are treated in the following result.
\begin{lemma}\label{7.28.2.18}
For $\varphi\in \{\bh^1, \phi\}$ and $n\le 3$, there hold for $(u_1, \ub_1)\in \I $ that
\begin{align}
{u_1}_+^{-\zeta(Z^n)+\f12\ga_0+\f12}&\|r^\f12 [\widetilde{\Box}_\bg, Z\rp{n}] \varphi\|_{L^2(\D_{u_1}^{\ub_1})}\les (\dn m_0^{-\f12}+m_0){u_1}_+^{-\f12\ga_0+\f12}\nn\\
&+m_0^{\f12} \sup_{-\ub_1\le u\le u_1, Z^b \subset Z^{n-1}\subset Z^n} E[\p Z\rp{b} \varphi]^\f12(\H_u^{\ub_1}) u_+^{-\zeta(Z^b)+\frac{\ga_0}{2}+1},\label{7.26.7.18}\\
{u_1}_+^{-\zeta(Z^n)+\f12\ga_0-\f12}&\|r^\frac{3}{2} [\widetilde{\Box}_\bg, Z\rp{n}]\varphi\|_{L_u^1 L_\ub^2 L_\omega^2(\D_{u_1}^{\ub_1})}\les (\dn m_0^{-\f12}+m_0) {u_1}_+^{-\f12\ga_0}+\dn^\f12 m_0^\f12 {u_1}_+^{-\f12},\label{7.26.8.18}
\end{align}
where the last term in (\ref{7.26.7.18}) vanishes if $Z^n=\p^n$.
\end{lemma}
\begin{proof}
We first rewrite  the terms in (\ref{4.29.1.18})-(\ref{6.1.1.18}) into
\begin{equation}\label{7.26.6.18}
|[\widetilde \Box_\bg, Z\rp{n}]\varphi|\les |H|\sum_{Z^a\sqcup Z^b=Z^n, b\le n-1}u_+^{\zeta(Z^a)} |\p^2 Z\rp{b} \varphi|+ \sum_{Z^a\sqcup Z^b\sqcup Z^c=Z^n,  a\ge 1}u_+^{\zeta(Z^c)} |Z\rp{a} H|| \p^2 Z\rp{b}\varphi|,
\end{equation}
where  the first term on the right  vanishes completely  if $Z^n=\p^n$.

By using (\ref{7.27.10.18}), we have
\begin{equation}\label{7.27.11.18}
 \sum_{Z^a \sqcup Z^b \sqcup Z^c=Z^n, a\ge 1}|Z\rp{a} H|| \p^2 Z\rp{b}\varphi| \les\sum_{Z^a \sqcup Z^b \sqcup Z^c=Z^n, a\ge 1} u_+^{\zeta(Z^c)} (|Z\rp{a}\bh^1|+r^{-1-a}m_0)  |\p^2 Z\rp{b} \varphi|,
\end{equation}
where  $a\ge 1$ and $b\le n-1$.
Thus
\begin{align*}
 &\sum_{Z^a \sqcup Z^b \sqcup Z^c=Z^n, a\ge 1}{u_1}_+^{\zeta(Z^c)}\| r^\frac{1}{2} Z\rp{a} H \c \p^2 Z\rp{b}\varphi\|_{L^2(\D_{u_1}^{\ub_1})}\\
 &\les \sum_{Z^a \sqcup Z^b \sqcup Z^c=Z^n, a\ge 1}{u_1}_+^{\zeta(Z^c)} \big(\|r^{\f12} Z\rp{a} \bh^1 \c\p^2 Z\rp{b} \varphi\|_{L^2(\D_{u_1}^{\ub_1})}+\| r^{-\f12-a} m_0 \p^2 Z\rp{b} \varphi\|_{L^2(\D_{u_1}^{\ub_1})}\big).
\end{align*}
By the definition of the standard energies followed with  direct integrations,  we derive with $0\le \a<\f12$ and $a\ge 1$,
\begin{align}
{u_1}_+^{-\a}\|u_+^{\f12+\a} r^{-\f12-a} &m_0 \p^2 Z\rp{b} \varphi\|_{L^2(\D_{u_1}^{\ub_1})}+\| r^{-a+\f12} m_0 \p^2 Z\rp{b} \varphi\|_{L_u^1 L_\ub^2 L_\omega^2(\D_{u_1}^{\ub_1})}\nn\\
&\les m_0 {u_1}_+^{-a-\f12\ga_0-\f12+\zeta(Z^b)}\{\sup_{-\ub_1\le u\le u_1} u_+^{\frac{\ga_0}{2}+1-\zeta(Z^b)}E[\p Z\rp{b} \varphi]^\f12(\H_u^{\ub_1})\nn\\
&+{u_1}_+^{\frac{\ga_0}{2}+1-\zeta(Z^b)}\sup_{-u_1\le \ub\le \ub_1}E[\p Z\rp{b}\varphi]^\f12(\Hb_\ub^{u_1})\}.\label{7.27.12.18}
\end{align}
Since  $\bh^1$ verifies (\ref{6.7.1.18}) which is stronger than the estimate (\ref{4.25.3.18}) of  $H$ for (\ref{eqn_1}), moreover  $\varphi\in (\bh^1,  \phi)$ verifies all the estimates in Section \ref{decay}, by repeating the same procedure for proving (\ref{6.5.8.18}) and (\ref{6.9.3.18}) in Proposition \ref{6.17.4.18}, we have
\begin{align*}
\sum_{Z^a \sqcup Z^b \sqcup Z^c=Z^n, a\ge 1}&\|u_+^{\zeta(Z^c)} r^{\f12} Z\rp{a} \bh^1 \c\p^2 Z\rp{b} \varphi\|_{L^2(\D_{u_1}^{\ub_1})}\\
+\sum_{Z^a \sqcup Z^b \sqcup Z^c=Z^n, a\ge 1}&{u_1}_+^{-\f12}\|u_+^{\zeta(Z^c)} r^{\frac{3}{2}} Z\rp{a} \bh^1 \c\p^2 Z\rp{b} \varphi\|_{L_u^1L_\ub^2 L_\omega^2(\D_{u_1}^{\ub_1})}\\
&\les m_0^{-\f12} \dn {u_1}_+^{-\ga_0+\zeta(Z^n)}.
\end{align*}
By combining the above two sets of estimates and noting that $-a\le \zeta(Z^a)$, the terms contributed by $\cir{\bh}$ in (\ref{7.27.12.18})  decay better since $1<\ga_0<2.$
Thus we obtain
\begin{equation}\label{7.28.5.18}
\begin{split}
\sum_{Z^a \sqcup Z^b \sqcup Z^c=Z^n, a\ge 1}&\|u_+^{\zeta(Z^c)} r^{\f12} Z\rp{a} H \c\p^2 Z\rp{b} \varphi\|_{L^2(\D_{u_1}^{\ub_1})}\\
+\sum_{Z^a \sqcup Z^b \sqcup Z^c=Z^n, a\ge 1}&{u_1}_+^{-\f12}\|u_+^{\zeta(Z^c)} r^{\frac{3}{2}} Z\rp{a} H \c\p^2 Z\rp{b} \varphi\|_{L_u^1L_\ub^2 L_\omega^2(\D_{u_1}^{\ub_1})}\\
&\les (m_0^{-\f12} \dn+m_0) {u_1}_+^{-\ga_0+\zeta(Z^n)}.
\end{split}
\end{equation}
Next we consider the borderline terms in  (\ref{7.26.6.18}). By using (\ref{7.26.11.18}), we can derive
\begin{align*}
&\sum_{Z^a \sqcup Z^b=Z^n, b\le n-1}\|r^{\f12}u_+^{\zeta(Z^a)}H \p^2 Z\rp{b} \varphi\|_{L^2(\D_{u_1}^{\ub_1})}\\
 &\les \sum_{Z^a\sqcup Z^b=Z^n, b\le n-1} \| r^\frac{1}{2} (r^{-1} m_0 +r^{-1} u_+^{\frac{1-\ga_0}{2}}\dn^\f12) u_+^{\zeta(Z^a)} \p^2 Z\rp{b} \varphi\|_{L^2(\D_{u_1}^{\ub_1})}\\
 &\les {u_1}_+^{\zeta(Z^n)}(m_0^\f12 {u_1}_+^{-\frac{\ga_0}{2}-\f12}+\dn^\f12 m_0^{-\f12} {u_1}_+^{-\ga_0}) \sup_{-\ub_1\le u\le u_1, Z^b\subset Z^{n-1}}  E[\p Z\rp{b} \varphi]^\f12(\H_u^{\ub_1}) u_+^{-\zeta(Z^b)+\frac{\ga_0}{2}+1}.
\end{align*}
We then substitute  (\ref{7.26.9.18}) into the  inequality if the product contains $\dn^\f12 m_0^{-\f12}$.  Combining the results  with (\ref{7.28.5.18}) implies (\ref{7.26.7.18}).  Similarly,
\begin{align*}
&\sum_{Z^a \sqcup Z^b=Z^n, b\le n-1}{u_1}_+^{-\f12}\|r^{\frac{3}{2}}u_+^{\zeta(Z^a)}H \p^2 Z\rp{b} \varphi\|_{L_u^1 L_\ub^2 L_\omega^2(\D_{u_1}^{\ub_1})}\\
&\les \sum_{Z^a\sqcup Z^b=Z^n, b\le n-1} {u_1}_+^{-\f12}\| r^\frac{3}{2} (r^{-1} m_0 +r^{-1} u_+^{\frac{1-\ga_0}{2}}\dn^\f12) u_+^{\zeta(Z^a)} \p^2 Z\rp{b} \varphi\|_{L_u^1 L_\ub^2 L_\omega^2 (\D_{u_1}^{\ub_1})}\\
 &\les {u_1}_+^{\zeta(Z^n)}(m_0^\f12 {u_1}_+^{-\frac{\ga_0}{2}-\f12}+\dn^\f12 m_0^{-\f12} {u_1}_+^{-\ga_0}) \sup_{-\ub_1\le u\le u_1, Z^b\subset Z^{n-1}}  E[\p Z\rp{b} \varphi]^\f12(\H_u^{\ub_1}) u_+^{-\zeta(Z^b)+\frac{\ga_0}{2}+1}.
\end{align*}
(\ref{7.26.8.18}) then follows  with the substitution of  (\ref{7.26.9.18}).
\end{proof}

\begin{remark}\label{7.28.1.18}
If $Z\in \{\p, \Omega_{\mu\nu}, 0\le \mu<\nu\le 3\}$, the result in this lemma still holds. For the proof, we only need to modify   (\ref{7.27.11.18}) by separating the case that $Z^n=\p^n$  from the general case.
\end{remark}

\begin{lemma}\label{7.8.15.18}
  For $(u_1, \ub_1)\in \I$ and  $n\le 3$,
\begin{align}
&\|r Z\rp{n} \S\|_{L^1_u L_\ub^2 L_\omega^2(\D^{\ub_1}_{u_1})}+\|r Z\rp{n} \S\|_{L_\ub^1 L_u^2 L_\omega^2(\D_{u_1}^{\ub_1})}\les {u_1}_+^{-\frac{3}{2}+\zeta(Z^n)}(\dn+ m^2_0),\label{7.7.3.18}\\
&\| r^\frac{3}{2} Z\rp{n} \S\|_{L_u^1 L_\ub^2 L_\omega^2(\D^{\ub_1}_{u_1})}\les {u_1}_+^{-1+\zeta(Z^n)}(\dn+ m^2_0),\label{7.7.4.18}\\
&\|r^\f12 q_0 Z\rp{n}( \bg_{\mu\nu}  \phi^2)\|_{L^2(\D^{\ub_1}_{u_1})}\les \dn {u_1}_+^{-\ga_0+\zeta(Z^n)}, \label{7.8.1.18}\\
&\|r^\frac{3}{2} q_0 Z\rp{n} (\bg_{\mu\nu} \phi^2)\|_{L_u^1 L_\ub^2 L_\omega^2(\D_{u_1}^{\ub_1})}\les \dn {u_1}_+^{-\ga_0+\f12+\zeta(Z^n)}. \label{7.8.2.18}
\end{align}
For $\varphi\in(\bh^1, \phi)$,
\begin{equation}\label{7.8.16.18}
\begin{split}
\|r Z\rp{n} (\N(\bh) \p \cir{\bh} \p\varphi)&\|_{L_u^1 L_\ub^2 L_\omega^2(\D_{u_1}^{\ub_1})}+\|r Z\rp{n} (\N(h) \p \cir{\bh} \p\varphi)\|_{L_\ub^1 L_u^ 2 L_\omega^2(\D_{u_1}^{\ub_1})}\\
&\qquad\quad\les({u_1}_+^{-1}m_0+{u_1}_+^{-1-\frac{\ga_0}{2}}\dn^\f12){u_1}_+^{-\frac{\ga_0}{2}+\zeta(Z^n)}\dn^\f12,
\end{split}
\end{equation}
\begin{equation}\label{7.8.17.18}
\|r^\frac{3}{2} Z\rp{n} (\N(\bh) \p \cir{\bh} \p\varphi)\|_{L_u^1 L_\ub^2 L_\omega^2(\D_{u_1}^{\ub_1})}\les ({u_1}_+^{-\f12}m_0+{u_1}_+^{-\frac{\ga_0}{2}-\f12}\dn^\f12){u_1}_+^{-\frac{\ga_0}{2}+\zeta(Z^n)}\dn^\f12.
\end{equation}
\end{lemma}
\begin{proof}
We  first  analyse  $\S$ which is defined in (\ref{7.9.1.18}), by using (\ref{7.7.7.18}) and (\ref{7.29.1.18}).
\begin{align*}
|Z\rp{n}(H^{\mu\nu}\p_\mu\p_\nu \cir{\bh})|&\les m_0\sum_{Z^a\sqcup Z^b\sqcup Z^c=Z^n } (|Z\rp{a}\bh^1|+|Z\rp{a}\cir{\bh}|) r^{-3+\zeta(Z^b)}u_+^{\zeta(Z^c)}\nn\\
&\les \sum_{Z^a\sqcup Z^b\sqcup Z^c=Z^n } m_0(|Z\rp{a}\bh^1|+m_0r^{-1+\zeta(Z^a)})  r^{-3+\zeta(Z^b)}u_+^{\zeta(Z^c)},
\end{align*}
\begin{align*}
|Z\rp{n} (G (\bh) \p \cir{\bh} \p \cir{\bh})|&\les\sum_{Z^a\sqcup Z^b \sqcup Z^c \sqcup Z^d=Z^n} |Z\rp{a}\bh|  |Z\rp{b}\p \cir{\bh}||Z\rp{c}\p\cir{\bh}|u_+^{\zeta(Z^d)}\\
&\les \sum_{Z^a\sqcup Z^b \sqcup Z^c \sqcup Z^d=Z^n }m_0^2 |Z\rp{a} \bh| r^{-4+\zeta(Z^b Z^c)} u_+^{\zeta(Z^d)}\\
&\les \sum_{Z^a\sqcup Z^b \sqcup Z^c \sqcup Z^d=Z^n} m_0^2(|Z\rp{a}\bh^1|+m_0r^{-1+\zeta(Z^a)})  r^{-4+\zeta(Z^b Z^c )}u_+^{\zeta(Z^d)},
\end{align*}
and
\begin{equation*}
|Z\rp{n}(\p \cir{\bh}  \p \cir{\bh})|\les m_0^2 r^{-4+\zeta(Z^n)}.
\end{equation*}
We combine the above calculations in view of $\N(\bh)=1+G(\bh)$,
\begin{align}
|Z\rp{n}\S|&\les \sum_{Z^a\sqcup Z^b\sqcup Z^c=Z^n } m_0(|Z\rp{a}\bh^1|+m_0r^{-1+\zeta(Z^a)})  r^{-3+\zeta(Z^b)}u_+^{\zeta(Z^c)}\label{7.7.8.18}.
\end{align}
 By using (\ref{4.29.6.18}) and (\ref{6.23.7.18}), we can obtain for $(u, \ub)\in \I$ that
\begin{equation*}
\| r^{-\f12} Z\rp{a} (\bh^1, \phi)\|_{L^2(S_{u, \ub})}\les u_+^{-\f12 \ga_0+\zeta(Z^a)}(\E^\f12_{a,\ga_0}+\dn^\f12).
\end{equation*}
We then combine the above two inequalities and integrate in $\D_{u_1}^{\ub_1}$ to obtain
\begin{align*}
&\|r Z\rp{n}\S\|_{L_u^1 L_\ub^2 L_\omega^2(\D_{u_1}^{\ub_1})}+\|r Z\rp{n}\S\|_{L_\ub^1 L_u^2 L_\omega^2(\D_{u_1}^{\ub_1})}\\
&\les m_0\sum_{Z^a\sqcup Z^b\sqcup Z^c=Z^n }\sup_{-\ub\le u\le u_1}\|r^{-\f12} Z\rp{a}\bh^1\|_{L^2(S_{u,\ub})}{u_1}_+^{-1+\zeta(Z^b Z^c)}+m_0^2{u_1}_+^{-\frac{3}{2}+\zeta(Z^n)}\\
&\les m_0(m_0 +\dn^\f12 {u_1}_+^{\f12-\frac{\ga_0}{2}}){u_1}_+^{-\frac{3}{2}+\zeta(Z^n)}.
\end{align*}
Similarly, we can obtain
\begin{equation*}
\| r^\frac{3}{2} Z\rp{n}\S\|_{L_u^1 L_\ub^2 L_\omega^2(\D_{u_1}^{\ub_1})}\les  m_0(m_0 +\dn^\f12 {u_1}_+^{\f12-\frac{\ga_0}{2}}){u_1}_+^{-1+\zeta(Z^n)}.
\end{equation*}
 Thus we have completed the proof of (\ref{7.7.3.18}) and (\ref{7.7.4.18}) for $\S$ defined in   (\ref{7.9.1.18}).

Next, by using (\ref{7.7.7.18})
\begin{align}
&|Z\rp{n}(\bg_{\mu\nu} \c \phi^2)|\nn\\
&\les \sum_{Z^a\sqcup Z^b =Z^n} |Z\rp{a} \phi||Z\rp{b} \phi|+\sum_{Z^a\sqcup Z^b\sqcup Z^c=Z^n} |Z\rp{a}\bh \c Z\rp{b} \phi \c  Z\rp{c}\phi|\nn\\
&\les \sum_{Z^a\sqcup Z^b=Z^n} |Z\rp{a} \phi||Z\rp{b}\phi|+ \sum_{Z^a\sqcup Z^b \sqcup Z^c= Z^n}(m_0r^{-1+\zeta(Z^a)}+|Z\rp{a} \bh^1|)||Z\rp{b}\phi||Z\rp{c}\phi|\nn\\
&\les \sum_{Z^a\sqcup Z^b\sqcup Z^c=Z^n} (r^{\zeta(Z^a)}+|Z\rp{a}\bh^1|) |Z\rp{b}\phi||Z\rp{c}\phi|. \label{7.29.2.18}
\end{align}
 Assume without loss of generality that $b\le c$. Thus $b\le 1$. By using (\ref{7.8.3.18}), we derive
\begin{align*}
q_0\| r^\frac{1}{2} Z\rp{b} \phi \c Z\rp{c}\phi\|_{L^2(\D_{u_1}^{\ub_1})}&\les \|r^\frac{5}{4} u_+^{-\frac{1}{4}+\frac{\ga_0}{2}}q_0^\f12 Z\rp{b}\phi\|_{L^\infty(\D_{u_1}^{\ub_1})}\|q_0^\f12 u_+^{\frac{1}{4} -\frac{\ga_0}{2}}r^{-\frac{3}{4}} Z\rp{c}\phi\|_{L^2(\D_{u_1}^{\ub_1})}\\
&\les (\int_{-\ub_1}^{u_1} E[Z\rp{c}\phi](\H_u^{\ub_1})u_+^{-\ga_0} du )^\f12 \dn^\f12 {u_1}_+^{\zeta(Z^b)-\f12}\\
&\les \dn {u_1}_+^{-\ga_0+\zeta(Z^b Z^c)}
\end{align*}
and
\begin{align*}
&q_0\| r^\frac{3}{2} Z\rp{b} \phi \c Z\rp{c}\phi\|_{L_u^1 L_\ub^2 L_\omega^2 (\D_{u_1}^{\ub_1})}\\
&\les \|r^\frac{5}{4} u_+^{-\frac{1}{4}+\frac{\ga_0}{2}}q_0^\f12 Z\rp{b}\phi\|_{L^\infty(\D_{u_1}^{\ub_1})}\|q_0^\f12 u_+^{\frac{1}{4} -\frac{\ga_0}{2}}r^\frac{1}{4} Z\rp{c}\phi\|_{L_u^1 L_\ub^2 L_\omega^2 (\D_{u_1}^{\ub_1})}\\
&\les \dn {u_1}_+^{-\ga_0+\f12+\zeta(Z^b Z^c)}.
\end{align*}
In view of $|r^{\zeta(Z^a)}|\les {u_1}_+^{\zeta(Z^a)}$, we can obtain the estimates for the first term on the right of (\ref{7.29.2.18}).
Next we consider the other term.  Since at least two of the indices $a,b,c$ are $\le 1$, and for convenience, we denote by $\Psi=(\bh^1, \phi)$ and  assume $a\le b\le c$. By using (\ref{5.1.1.18}) and (\ref{6.23.7.18}), for any $(u,\ub)\in \I$ there holds
\begin{align*}
\| r^{-1} Z\rp{c} \Psi\|_{L^2(\H_u^\ub)}\les u_+^{-\frac{\ga_0}{2}}(\E^\f12_{c,\ga_0}+\dn^\f12).
\end{align*}
We also note that in the symbolic notation $Z\rp{a}\Psi$, $Z\rp{b}\Psi$, at least one of the functions $\Psi$ is actually $\phi$.  In this case,
by using  (\ref{3.24.11.18}), (\ref{7.8.3.18}) and the above inequality, we have
\begin{align*}
\sum_{Z^a \sqcup Z^b\sqcup Z^c=Z^n}&q_0\|r^\frac{1}{2} Z\rp{a}\Psi\c Z\rp{b} \Psi\c Z\rp{c}\Psi\|_{L^2(\D_{u_1}^{\ub_1})}\\
&\les \sum_{Z^a \sqcup Z^b\sqcup Z^c=Z^n} q_0\|u_+^{\ga_0-\frac{3}{4}}r^{2+\frac{1}{4}}  Z\rp{a}\Psi Z\rp{b} \Psi\|_{L^\infty(\D_{u_1}^{\ub_1})}\| u_+^{\frac{3}{4}-\ga_0 }r^{-\frac{7}{4}} Z\rp{c}\Psi\|_{L^2 (\D_{u_1}^{\ub_1})}\\
&\les{u_1}_+^{\zeta(Z^n)+\f12-\frac{3}{2}\ga_0}\dn^\frac{3}{2}
\end{align*}
and
\begin{align*}
\sum_{Z^a \sqcup Z^b\sqcup Z^c=Z^n}&q_0\|r^\frac{3}{2} Z\rp{a}\Psi\c Z\rp{b} \Psi\c Z\rp{c}\Psi\|_{L_u^1 L_\ub^2 L_\omega^2(\D_{u_1}^{\ub_1})}\les {u_1}_+^{\zeta(Z^n)+1-\frac{3}{2}\ga_0}\dn^\frac{3}{2}.
\end{align*}
Thus we completed the proof of (\ref{7.8.1.18}) and (\ref{7.8.2.18}).

At last we prove (\ref{7.8.16.18}) and (\ref{7.8.17.18}).
By using (\ref{5.18.11.18}) and (\ref{7.7.7.18}), we derive
\begin{equation*}
| Z\rp{n} (\p \cir{\bh} \c \p \varphi)|\les \sum_{Z^b \sqcup Z^c =Z^n}m_0 r^{-2+\zeta(Z^b)}| \p Z\rp{c} \varphi|.
\end{equation*}
For  the cubic part $ Z\rp{n}(G(\bh) \p \cir{\bh} \p \varphi)$, we first derive by using (\ref{7.29.1.18}), (\ref{5.17.3.18}) and (\ref{7.7.7.18})
\begin{align*}
| Z\rp{n}(G(\bh) \p \cir{\bh} \p \varphi)|&\les \sum_{Z^a\sqcup Z^b \sqcup Z^c \sqcup Z^d=Z^n }u_+^{\zeta(Z^d)}(|Z\rp{a} \bh^1| +|Z\rp{a} \cir{h}|)|Z\rp{b}\p \cir{\bh}||\p Z\rp{c} \varphi|\nn\\
&\les \sum_{Z^a\sqcup Z^b \sqcup Z^c \sqcup Z^d =Z^n}m_0 u_+^{\zeta(Z^d)}r^{-2+\zeta(Z^b)} (| Z\rp{a} \bh^1|+m_0 r^{-1+\zeta(Z^a)})|\p Z\rp{c}\varphi|.
\end{align*}
Thus
\begin{align}
| Z\rp{n}(\N(\bh) \p \cir{\bh} \p \varphi)|&\les \sum_{Z^a\sqcup Z^b \sqcup Z^c \sqcup Z^d =Z^n}m_0 u_+^{\zeta(Z^d)}r^{-2+\zeta(Z^b)} (| Z\rp{a} \bh^1|+ r^{\zeta(Z^a)})|\p Z\rp{c}\varphi|.\label{7.8.6.18}
\end{align}
We first note that  with $0\le \a<\f12$ by (\ref{7.19.2.18}),
\begin{equation}\label{7.8.7.18}
\|r^{-1+\a}  \p Z\rp{c}\varphi\|_{L^2(\D_{u_1}^{\ub_1})}\les {u_1}_+^{\a-\f12-\frac{\ga_0}{2}+\zeta(Z^c)}\dn^\f12, \quad c\le 3.
\end{equation}
For simplicity, we denote $\|\cdot\|_\flat$ either the norm $\| \cdot\|_{L_u^1 L_\ub^2 L_\omega^2(\D_{u_1}^{\ub_1})}$ or $\| \cdot\|_{L_\ub^1 L_u^2 L_\omega^2(\D_{u_1}^{\ub_1})}$.
Thus, by using the above inequality with $\a=0$,
\begin{equation}\label{7.8.12.18}
\begin{split}
\sum_{Z^b \sqcup Z^c \sqcup Z^d =Z^n}&\|ru_+^{\zeta(Z^d)} m_0 r^{-2+\zeta(Z^b)} \p Z\rp{c} \varphi||_\flat\\
&\les\sum_{Z^{b'}\sqcup Z^c=Z^n} m_0{u_1}_+^{\zeta(Z^{b'})}\|r^{-1}\p Z\rp{c}\varphi\|_{L^2(\D_{u_1}^{\ub_1})}(\|r^{-1}\|_{L^2_\ub L^\infty (\D_{u_1}^{\ub_1})}+\|r^{-1}\|_{L_u^2 L^\infty(\D_{u_1}^{\ub_1})})\\
&\les{u_1}_+^{-1-\frac{\ga_0}{2}+\zeta(Z^n) }m_0 \dn^\f12.
\end{split}
\end{equation}
 With $0<\a<\f12$ in (\ref{7.8.7.18}) and $u_+\les r$
 \begin{align}
&\sum_{Z^b \sqcup Z^c \sqcup Z^d =Z^n}\|r^\frac{3}{2}u_+^{\zeta(Z^d)} m_0 r^{-2+\zeta(Z^b)} \p Z\rp{c} \varphi\|_{L_u^1 L_\ub^2 L_\omega^2(\D_{u_1}^{\ub_1})}\nn\\
&\les  m_0\sum_{Z^{b'}\sqcup Z^c=Z^n} {u_1}_+^{\zeta(Z^{b'})}\|u_+^\a r^{-1}\p Z\rp{c}\phi\|_{L^2(\D_{u_1}^{\ub_1})}\|r^{-\f12}u_+^{-\a}\|_{L^2_u L^\infty (\D_{u_1}^{\ub_1})}\nn\\
&\les{u_1}_+^{-\f12-\frac{\ga_0}{2}+\zeta(Z^n) }m_0 \dn^\f12.\label{7.8.13.18}
\end{align}

By the Sobolev embedding on the unit sphere and (\ref{7.8.7.18}), we have   for $c\le 2$  and $0\le \a <\f12$ that
\begin{equation}\label{7.8.8.18}
\begin{split}
\| r^{\a}\p Z\rp{c} \varphi\|_{L_u^2 L_\ub^2 L_\omega^4(\D_{u_1}^{\ub_1})}&\les \|r^{\a-1}\p \Omega\rp{\le 1} Z\rp{c}\varphi \|_{L^2(\D_{u_1}^{\ub_1})}\\
&\les {u_1}_+^{\a-\f12-\frac{\ga_0}{2}+\zeta(Z^c)} \dn^\f12.
\end{split}
\end{equation}
We  can also use (\ref{5.22.2.18}), (\ref{6.24.4.18})  and the bootstrap assumption to obtain for $a\le 3$ that
\begin{equation}\label{7.8.9.18}
\|r^\f12 Z\rp{a} \bh^1\|_{L_\omega^4(S_{u,\ub})}\les u_+^{-\frac{\ga_0}{2}+\zeta(Z^a)}\dn^\f12, \quad (u, \ub)\in \I.
\end{equation}
With $0<\a<\f12$,
we then employ (\ref{7.8.8.18}) and (\ref{7.8.9.18}) to treat the first term on the right if $c\le n-1$
\begin{align*}
 &\sum_{Z^a\sqcup Z^b \sqcup Z^c=Z^n, c\le n-1}{u_1}_+^{\zeta(Z^b)}m_0\| r^{-\f12+\a}u_+^\f12\p Z\rp{c}\varphi\c Z\rp{a} \bh^1\|_{L^2(\D_{u_1}^{\ub_1})}\\
 &\les  \sum_{Z^a\sqcup Z^b \sqcup Z^c=Z^n, c\le n-1}{u_1}_+^{\zeta(Z^b)}m_0\| r^\a \p Z\rp{c}\varphi \|_{L_u^2 L_\ub^2 L_\omega^4(\D_{u_1}^{\ub_1})} \|r^\f12 u_+^\f12 Z\rp{a} \bh^1\|_{L^\infty L_\omega^4(\D_{u_1}^{\ub_1})}\\
 &\les {u_1}_+^{\a-\ga_0+\zeta(Z^n)}m_0\dn.
\end{align*}
When $c=n$, by using (\ref{7.8.7.18}) and (\ref{3.24.11.18}) for $\bh^1$, we derive
\begin{align*}
 m_0\| r^{-\f12+\a}u_+^\f12\p Z\rp{n}\varphi\c  \bh^1\|_{L^2(\D_{u_1}^{\ub_1})}&\les m_0\| r^{-1+\a} \p Z\rp{n}\varphi \|_{L^2(\D_{u_1}^{\ub_1})} \|u_+^\f12r^\f12 \bh^1\|_{L^\infty(\D_{u_1}^{\ub_1})}\\
 &\les {u_1}_+^{\a-\ga_0+\zeta(Z^n)}m_0\dn.
\end{align*}
Thus
\begin{align*}
 &\sum_{Z^a\sqcup Z^b \sqcup Z^c=Z^n}{u_1}_+^{\zeta(Z^b)}m_0\| r^{-\f12+\a} u_+^\f12\p Z\rp{c}\varphi\c Z\rp{a} \bh^1\|_{L^2(\D_{u_1}^{\ub_1})}\les {u_1}_+^{\a-\ga_0+\zeta(Z^n)}m_0\dn.
\end{align*}
Also by using H\"{o}lder's inequality in $\D_{u_1}^{\ub_1}$,
\begin{align*}
&\sum_{Z^a\sqcup Z^b \sqcup Z^c=Z^n}{u_1}_+^{\zeta(Z^b)}m_0\| r^{-1}\p Z\rp{c}\varphi\c Z\rp{a} \bh^1\|_\flat\les {u_1}_+^{-1-\ga_0+\zeta(Z^n)}m_0\dn
\end{align*}
and
\begin{align*}
&\sum_{Z^a\sqcup Z^b \sqcup Z^c=Z^n}{u_1}_+^{\zeta(Z^b)}m_0\| r^{-\f12}\p Z\rp{c}\varphi\c Z\rp{a} \bh^1\|_{L_u^1 L_\ub^2 L_\omega^2(\D_{u_1}^{\ub_1})}\les {u_1}_+^{-\ga_0-\f12+\zeta(Z^n)}m_0\dn.
\end{align*}
In view of (\ref{7.8.6.18}), we combine these two estimates with (\ref{7.8.12.18})  and (\ref{7.8.13.18}) to obtain (\ref{7.8.16.18}) and (\ref{7.8.17.18}).
\end{proof}

\subsection{Boundedness of energy}\label{einsbd}

Note that, according to the notation in  Lemma \ref{6.3.5.18}, for the equation  (\ref{7.7.2.18}), we can set
\begin{equation*}
\F[\bh^1]= \N(\bh)\p \bh^1 \p \bh^1+ \p \phi\c \p \phi; \qquad  \F[\phi]=0.
\end{equation*}

When consider the energy of $\bh^1$, we decompose $\sF_{\bh^1}$ into two parts,
 \begin{equation*}
 \sF^\sharp_{\bh^1}= [\widetilde{\Box}_\bg, Z\rp{n}]\bh^1+\str{n}\F[\bh^1]+ q_0Z\rp{n}(\bg_{\mu\nu} \phi^2); \quad \sF^\flat_{\bh^1}=Z\rp{n}\S+ Z\rp{n}(\N(\bh) \p \cir{\bh} \p \bh^1)
 \end{equation*}
and for the energy of  the scalar field $\phi$, we have
\begin{equation*}
\sF_\phi=\sF^\sharp_\phi= [\widetilde{\Box}_\bg, Z\rp{n}] \phi.
\end{equation*}
Hence for  $\bh^1$,  we can treat $\sF_{\bh^1}^\flat $ by using (\ref{7.8.16.18}) and (\ref{7.7.3.18}). To treat $\sF_{\bh^1}^\sharp$, we can apply (\ref{6.8.2.18})-(\ref{6.9.9.18}) to treat $\str{n}\F[\bh^1]$ and treat the remaining terms  by using (\ref{7.8.1.18}), (\ref{7.8.2.18}) and  Lemma \ref{7.28.2.18}.
\begin{equation}\label{7.16.1.18}
\begin{split}
&\|r\sF^\flat_{\bh^1} \|_{L_u^1 L_\ub^2 L_\omega^2(\D_{u_1}^{\ub_1})}+\| r \sF^\flat_{\bh^1}\|_{L_\ub^1 L_u^2 L_\omega^2(\D_{u_1}^{\ub_1})}\les {u_1}_+^{-\frac{3}{2}+\zeta(Z^n)}(\dn+m_0^2)\\
&\| r^\frac{3}{2} \sF_{\bh^1}\|_{L_u^1 L_\ub^2 L_\omega^2(\D_{u_1}^{\ub_1})}\les (\dn m_0^{-\f12} {u_1}_+^{-\ga_0+\f12}+\dn {u_1}_+^{-1}) {u_1}_+^{\zeta(Z^n)}+{u_1}_+^{-\f12\ga_0 +\zeta(Z^n)}m_0^\f12\dn^\f12\\
&\qquad\qquad \qquad\qquad\quad\les (\dn m_0^{-\f12}+\dn^\f12 m_0^{\f12}) {u_1}_+^{-\f12 \ga_0+\zeta(Z^n)} \\
&\| r^\f12 \sF^\sharp_{\bh^1}\|_{L^2(\D_{u_1}^{\ub_1})}\les \dn m_0^{-\f12} {u_1}_+^{-\ga_0+\zeta(Z^n)}\\
&+{u_1}_+^{-\f12\ga_0-\f12 +\zeta(Z^n)}m_0^\f12\sup_{-\ub_1\le u \le u_1, Z^b \subset Z^{n-1}\subset Z^n} E[\p Z\rp{b} \bh^1]^\f12(\H_u^{\ub_1}) u_+^{-\zeta(Z^b)+\f12\ga_0+1}, \\
\end{split}
\end{equation}
where the last term in the last  inequality vanishes if $\zeta(Z^n)=-n$.

We use  Lemma \ref{7.28.2.18} to treat $\sF_\phi$,
\begin{equation}\label{7.16.2.18}
\begin{split}
&\|r^\frac{3}{2} \sF_\phi\|_{L_u^1 L_\ub^2 L_\omega^2(\D_{u_1}^{\ub_1})}\les {u_1}_+^{-\f12\ga_0+\zeta(Z^n)}(m_0^\f12\dn^\f12+\dn m_0^{-\f12})\\
&\|r^\f12 \sF_\phi\|_{L^2(\D_{u_1}^{\ub_1})}\les \dn m_0^{-\f12} {u_1}_+^{-\ga_0+\zeta(Z^n)}\\
&+{u_1}_+^{-\f12\ga_0-\f12 +\zeta(Z^n)}m_0^\f12\sup_{-\ub_1\le u \le u_1, Z^b \subset Z^{n-1}\subset Z^n} E[\p Z\rp{b} \phi]^\f12(\H_u^{\ub_1}) u_+^{-\zeta(Z^b)+\f12\ga_0+1},
\end{split}
\end{equation}
where the last terms in the above inequalities vanish if $\zeta(Z^n)=-n$.

   By a  substitution of (\ref{7.16.1.18}) and (\ref{7.16.2.18})  to the energy inequalities in Proposition   \ref{4.29.4.18} and Proposition \ref{4.29.5.18}, we can obtain the boundedness of energies for $\Psi= (\bh^1,\phi)$ on  $\I=\{(u,\ub):-\ub_1\le -\ub\le u\le u_1\}$
\begin{equation}\label{7.27.14.18}
\begin{split}
\sup_{(u,\ub)\in \I} & u_+^{-2\zeta(Z^n)+\ga_0}(E[Z\rp{n}\Psi](\H_u^{\ub})+ E[Z\rp{n}\Psi](\Hb_\ub^{u}))\\
 &\les \E_{n, \ga_0}(\bh^1, \phi) +{u_1}_+^{\ga_0-3}(\dn^2 +m_0^4)+ {u_1}_+^{1-\ga_0}m_0^{-2} \dn^2\\
&+\sup_{-\ub_1\le u \le u_1, Z^b \subset Z^{n-1}\subset Z^n} E[\p Z\rp{b} \Psi](\H_u^{\ub_1}) u_+^{-2\zeta(Z^b)+\ga_0+2},
\end{split}
\end{equation}
 where  the last term in the above  inequality vanishes if $Z^n=\p^n$.  We then run the induction on the signature $\zeta(Z^n)$.
  When $\zeta(Z^n)=-n+l$, $1\le l\le n$, we substitute the estimates  for $Z\rp{n}\Psi$ with  $\zeta(Z^n)=-n+l-1$. This procedure allows us to conclude  that   (\ref{7.27.14.18})  can be bounded by
   \begin{equation}\label{7.31.1.18}
  C_3( m_0^4 {u_1}_+^{\ga_0-3}  + {u_1}_+^{1-\ga_0}m_0^{-2} \dn^2  +\E_{n,\ga_0}(\bh^1, \phi)).
  \end{equation}
 We also have the bound for $r$-weighted energy,
 \begin{equation}\label{7.27.15.18}
 \begin{split}
  \sup_{-\ub_1 \le u\le u_1} {u}_+^{\ga_0-1-2\zeta(Z^n)}&\big(\W_1[Z\rp{n}\Psi](\D_{u}^{\ub_1})+ \W_1[Z\rp{n}\Psi](\H_{u}^{\ub_1})+\W_1[Z\rp{n}\Psi](\Hb_{\ub_1}^{u})\big)\\
&\les {u_1}_+^{-1}(\dn m_0+\dn^2 m_0^{-1}) +\E_{n,\ga_0}(\bh^1, \phi)\\
&+ \sup_{ -\ub_1\le -\ub \le u\le u_1}\{ u_+^{\ga_0-2\zeta(Z^n)}  (E[Z\rp{n}\Phi](\H_u^\ub)+ E[Z\rp{n}\Phi](\Hb_\ub^u))\}.
\end{split}
  \end{equation}
 For the last line, we directly substitute the bound (\ref{7.31.1.18}),  thus we  conclude that (\ref{7.27.15.18}) is bounded by (\ref{7.31.1.18}) with a larger constant $C_3$.

Since $\E_{n, \ga_0}(\bh^1, \phi)\le C_0m_0^2$,  with $\dn =C_1 m_0^2$  and $C_0, C_1\ge 1$
we need
\begin{equation*}
C_3(C_0 C_1^{-1} \dn+u_+^{\ga_0-3} \dn^2 C_1^{-2}+u_+^{1-\ga_0}\dn C_1)<2\dn.
\end{equation*}
Let  $C_1=4 C_0 C_3$ and $\dn<1$, due to  $u_+^{-\ga_0+1}\le (\frac{R}{2})^{1-\ga_0}$, as long as
\begin{equation}\label{7.16.5.18}
2C_3 C_1 (\frac{R}{2})^{1-\ga_0}<\frac{7}{4},
\end{equation}
we can obtain
for $n\le 3 $ and $-\ub_*\le -\ub\le u\le u_0$  that
\begin{align*}
&E[Z\rp{n} (\bh^1, \phi)](\H_u^\ub)+E[Z\rp{n} (\bh^1, \phi)](\Hb_\ub^u)< 2\dn u_+^{-\ga_0+2\zeta(Z^n)},\\
&\W_1[Z\rp{n} (\bh^1, \phi)](\H_u^\ub)+\W_1[Z\rp{n} (\bh^1,\phi)](\Hb_\ub^u)+\W_1[Z\rp{n}\phi](\D_u^\ub)< 2\dn u_+^{-\ga_0+1+2\zeta(Z^n)}.
\end{align*}

It remains to improve (\ref{6.10.4.18}) with $\ge $ replaced by $>$. Due to (\ref{7.29.3.18}), (\ref{3.24.11.18}) and (\ref{7.26.11.18}),
\begin{equation*}
r|H^{\Lb \Lb }-\cir{H}{}^{\Lb \Lb}|\le  C(\dn^\f12u_+^{-\frac{\ga_0}{2}+\f12}+r^{-1}(m_0+\dn^\f12)^2 ).
\end{equation*}
Similar to Lemma \ref{7.29.4.18}, with $0<m_0<1$, we choose
\begin{equation}\label{7.29.5.18}
C(C_1^\f12 (\frac{R}{2})^{\frac{1-\ga_0}{2}}+R^{-1} (1+C_1^\f12)^2)<\frac{1}{10}
\end{equation}
which implies
\begin{align*}
r(h-H^{\Lb\Lb}),\quad r (h-H^{LL})>\frac{ 7 m_0}{20 }
\end{align*}
which improves (\ref{6.10.4.18}). By choosing  the lower bound $R(\ga_0, C_0)$ of $R$ such that  (\ref{7.29.5.18}),  (\ref{7.16.5.18}) and  $ \ti C_1 R^{-\frac{\ga_0}{2}+\f12}\le 1$ as requested in Lemma \ref{7.29.4.18}  hold,  the proof of Theorem \ref{eins_thm} is completed.

\end{document}